\documentclass[12pt,reqno]{amsart}
\usepackage{amsmath,amssymb,amsthm,mathtools,wasysym,calc,verbatim,enumitem,tikz,pgfplots,url,hyperref,mathrsfs,cite,fullpage,bbm}
\usepackage{comment}

\newtheorem{theorem}{Theorem}[section]
\newtheorem{lemma}[theorem]{Lemma}

\newtheorem{definition}[theorem]{Definition}

\newtheorem{claim}[theorem]{Claim}
\newtheorem{fact}[theorem]{Fact}

\theoremstyle{definition}

\theoremstyle{remark}

\newcommand\N{\mathbb{N}}
\newcommand\R{\mathbb{R}}
\newcommand\Z{\mathbb{Z}}

\newcommand\cA{\mathcal{A}}
\newcommand\cB{\mathcal{B}}
\newcommand\cN{\mathcal{N}}

\newcommand\cX{\mathcal{X}}
\newcommand\cW{\mathcal{W}}
\newcommand\cQ{\mathcal{Q}}
\newcommand\cU{\mathcal{U}}
\newcommand\cE{\mathcal{E}}
\newcommand\cF{\mathcal{F}}
\newcommand\cM{\mathcal{M}}

\newcommand{\HS}{\mathrm{HS}}
\newcommand{\rk}{\mathrm{rk}}
\newcommand\cT{\mathcal{T}}
\def\Pr{\mathbb{P}}
\def\S{\mathcal{S}}
\newcommand\Ex{\mathbb{E}}

\newcommand\eps{\varepsilon}
\renewcommand{\P}{\mathbb{P}}
\renewcommand{\leq}{\leqslant}
\renewcommand{\geq}{\geqslant}

\renewcommand{\to}{\rightarrow}

\def\eps{\varepsilon}

	\def\E{\mathbb{E}}
	
	\def\R{\mathbb{R}}
		
	\def\Z{\mathbb{Z}}
	\def\N{\mathbb{N}}
	\def\PP{\mathbb{P}}
	\def\S{\mathbb{S}}
	\def\1{\mathbbm{1}}
	
	\def\k{\kappa}
	
	\def\s{\sigma}
	\def\t{\theta}
	
	\def\g{\gamma}
	
	\def\G{\Gamma}	
	
	\def\la{\langle}
	\def\ra{\rangle}

		\def\Inc{\mathrm{Incomp\,}}
		\def\Comp{\mathrm{Comp\,}}

	\def\EE{\mathbb{E}}
		
		\def\T{\mathbb{T}}
	\def\cA{\mathcal{A}}
	\def\cQ{\mathcal{Q}}	
	
	\def\cL{\mathcal{L}}
	\def\cI{\mathcal{I}}

	\def\L{\Lambda}
	\def\Q{\mathcal{Q}}
	\def\vp{\varphi}

\begin{document}

\title{The singularity probability of a random symmetric matrix is exponentially small}

\author{Marcelo Campos, Matthew Jenssen, Marcus Michelen and Julian Sahasrabudhe}

\begin{abstract}
Let $A$ be drawn uniformly at random from the set of all $n\times n$ symmetric matrices with entries in $\{-1,1\}$. We show that 
\[ \PP( \det(A) = 0 ) \leq e^{-cn},\]
where $c>0$ is an absolute constant, thereby resolving a well-known conjecture.
\end{abstract}

\address{Instituto de Matem\'atica Pura e Aplicada (IMPA). }
\email{marcelo.campos@impa.br}
\thanks{The first named author is partially supported by CNPq.}
\address{University of Birmingham, School of Mathematics.}
\email{m.jenssen@bham.ac.uk}
\address{University of Illinois at Chicago. Department of Mathematics, Statistics and Computer Science.}
\email{michelen.math@gmail.com}
\address{University of Cambridge. Department of Pure Mathematics and Mathematics Statistics.} 
\email{jdrs2@cam.ac.uk}\maketitle

\maketitle

\section{Introduction}

Let $B$ be a random $n \times n$ matrix whose entries are chosen independently and uniformly from $\{-1,1\}$. It is an old problem, likely stemming from multiple origins, to determine the probability that $B$ is singular. While a moment's thought reveals the lower bound of $(1+o(1))2n^2 2^{-n}$, the probability that two rows or columns are equal up to sign,
establishing the corresponding \emph{upper bound} remains an extremely challenging open problem. Indeed, it is widely believed that 
\begin{equation} \label{eq:asym-conj} \P(\det(B) = 0) = (1+o(1))2n^2 2^{-n}. \end{equation}
While this precise asymptotic has so far eluded researchers, a huge amount is now known about this fascinating problem. The first advances were made by 
Koml\'os \cite{komlos} in the 1960s, who showed that the singularity probability is $O(n^{-1/2})$ (see also \cite{komlos68} and \cite{BB}).

Nearly 30 years later Kahn, Koml\'os and Szemer\'edi \cite{KKS}, in a remarkable paper, showed that the singularity probability is exponentially small. 
At the heart of their paper is an ingenious argument with the Fourier transform that allows them to give vastly more efficient descriptions of ``structured'' subspaces of $\R^n$ that are spanned by $\{-1,1\}$-vectors. Their method was then developed by Tao and Vu \cite{TV-RSA,TV-JAMS} who showed a bound of $(3/4+o(1))^n$, by proving an interesting link between the ideas of \cite{KKS} and the structure of set addition and, in particular, Freiman's theorem. This trajectory was then developed further by Bourgain, Vu and Wood \cite{BVW}, who proved a bound of $(2^{-1/2} + o(1))^n$, and by Tao and Vu \cite{tao-vu-ILO}, who pioneered the development of ``inverse Littlewood-Offord theory,'' now an integral aspect of random matrix theory (see Section~\ref{subsec:LwO-intro}).

In 2007, Rudelson and Vershynin, in an important and influential paper \cite{RV}, gave a different proof of the exponential upper bound on the singularity probability of $B$. The key idea was to construct efficient $\eps$-nets for points on the sphere that have special anti-concentration properties and are thus more likely to be in the kernel of $B$. This then led them to prove an elegant inverse Littlewood-Offord type result, inspired by \cite{tao-vu-ILO}, in a geometric setting.

This perspective was then developed further in the 2018 breakthrough work of Tikhomirov \cite{Tikhomirov}, who proved 
\[ \P(\det (B) = 0) = (1/2 + o(1))^n,\] thereby essentially proving the conjectured upper bound. One of the key innovations in \cite{Tikhomirov} was to
adopt a probabilistic viewpoint of the (discretized) sphere: instead of directly proving that efficient nets exist by latching onto some sort of structure,
he shows that the probability of randomly selecting a ``structured'' point on the discrete sphere is incredibly unlikely. While this change in perspective may not immediately sound useful, Tikhomirov's ``inversion of randomness'' gives him access to a whole host of probabilistic tools.

Another major advance on the problem was made recently by Jain, Sah and Sawhney \cite{jain2020singularityI, jain2020singularity}, who (building on the recent work of Litvak and Tikhomirov \cite{Lit-Tik}), proved the natural analogue of \eqref{eq:asym-conj} for random matrices with independent entries chosen from a finite set $S$, for any \emph{non-uniform} distribution on $S$. For the case of $\{-1,1\}$-matrices, however, they do not improve on the bound of Tikhomirov.

While the problem for matrices $B$ with all entries independent is now very well understood, the situation for \emph{symmetric} random matrices remains somewhat more mysterious. Indeed all of the previously mentioned works on random matrices depend deeply on the fact that the entries of $B$ are independent, and often treat 
$B$ as $n$ independent copies of a row, thus allowing for an essentially ``one-dimensional'' treatment of the problem. In the symmetric case, no such perspective is available.

Let $A$ be a random $n \times n$ symmetric matrix, uniformly drawn from all symmetric matrices with entries in $\{ -1,1 \}$. 
Again, it is generally believed that $\P(\det A = 0 ) = \Theta(n^2 2^{-n})$ (see, e.g. \cite{costello-tao-vu, vu-randomdiscretematrices, vu-2020-survey}) but progress has come more slowly. The problem of showing that $A$ is almost surely 
non-singular goes back, at least, to Weiss in the early 1990s but was not resolved until 2005 by Costello, Tao and Vu \cite{costello-tao-vu}, who obtained the bound
\begin{equation}\label{eq:CostelloTaoVuBnd} \PP( \det(A) = 0 ) \leq n^{-1/8+o(1)}. \end{equation}

The first super-polynomial bounds were obtained by Nguyen \cite{nguyen-singularity} and, simultaneously, Vershynin \cite{vershynin-invertibility},
the latter obtaining a bound of the form $\exp(-n^c)$. Nguyen \cite{nguyen-singularity} developed the quadratic Littlewood-Offord theory introduced in \cite{costello-tao-vu}, while Vershynin \cite{vershynin-invertibility} worked in the geometric framework pioneered in his work with Rudelson \cite{RV,RV-ICM,RV-rectangle}.

In 2019, a more combinatorial perspective for inversion of random discrete matrices was introduced by Ferber, Jain, Luh and Samotij \cite{ferber2019counting} and applied by Ferber and Jain \cite{ferber-jain} to show 
\[ \P(\det A = 0 )\leq \exp(-c n^{1/4}(\log n)^{1/2} )\,.\] 
 In a similar spirit, Campos, Mattos, Morris and Morrison \cite{CMMM} then improved this bound to 
\begin{equation} \label{eq:cmmm-bnd} \P(\det A = 0) \leq \exp(-c n^{1/2}), 
\end{equation} by proving a``rough'' inverse Littlewood-Offord theorem, inspired by the theory of hypergraph containers (see \cite{containersBMS,containersST}). This bound was then improved by Jain, Sah and Sawhney \cite{JSS-symmetric}, who improved the exponent to $-c n^{1/2} \log^{1/4} n$, and, simultaneously, by the authors of this paper \cite{RSM1} who improved the exponent to $-c (n \log n)^{1/2}$. 

The convergence of these results onto the exponent of $-c (n \log n)^{1/2}$ is no coincidence and in fact represents a natural barrier in the problem.
Indeed, all of the results up to now have treated ``structured'' vectors by only using the top-half of the matrix (i.e.\ the half above the diagonal), which conveniently consists of independent entries. However, as pointed out in \cite{CMMM}, if one is restricted to working in the top-half of $A$ one cannot obtain an exponent better than $-c (n \log n)^{1/2}$. Thus to get beyond this obstruction, somehow the randomness of the matrix must ``reused''.

In this paper we prove an exponential upper-bound on the singularity probability of a symmetric random matrix, thereby breaking though this barrier and
giving the optimal bound, up to the constant in the exponent.

\begin{theorem}\label{thm:main} Let $A$ be uniformly drawn from all $n\times n$ symmetric matrices with entries in $\{-1, 1\}$. Then 
\begin{equation}\label{eq:expbound}  \PP( \det(A) = 0 ) \leq e^{-cn},  \end{equation} 
where $c>0$ is an absolute constant.
\end{theorem}

The main technical innovations of this paper are a new inverse Littlewood-Offord type theorem for ``conditioned'' random walks
and a new ``inversion of randomness'' technique that allows us to ``reuse'' the randomness of our matrix by pushing some of the randomness 
onto the random selection of a vector from our discretized sphere.
In fact, there is a delicate tradeoff between these two ingredients;
a loss in the second ingredient allows for an improvement in the first, \emph{unless} some specific ``arithmetic'' structure arises (see Section~\ref{sec:sketch}).

We note that there is a natural sister problem regarding the probability that the \emph{least singular value} of $A$ is at most $\eps$. Our methods also apply to this problem and we will detail these results in a forthcoming paper, where more general coefficient distributions are also treated. In this paper we have opted to 
sacrifice generality to reduce clutter and technicality and, hopefully, to make the new ideas as clear as possible.

\subsection{Inverse Littlewood-Offord theory}\label{subsec:LwO-intro}

For $v \in \R^n$, we define the concentration function (one of several to come) as 
\[ \rho(v) := \max_{b \in \R} \PP\left(  \sum_{i=1}^n \eps_i v_i = b \right), \]
where $\eps_1,\ldots,\eps_n \in \{-1,1\}$ are uniform and independent. The study of $\rho(v)$ goes back at least to the classical work of Littlewood and Offord~\cite{LwO-1,LwO-3} on the zeros of random polynomials, but perhaps begins in earnest with the beautiful 1945 result of Erd\H{o}s \cite{erdos-LwO}: if $v \in \R^n$ has all non-zero coordinates then 
\[\rho(v) \leq 2^{-n} \binom{n}{\lfloor n/2 \rfloor} = O(n^{-1/2}).\] 
This was then developed by Szemer\'{e}di and S\'{a}rk\"{o}zy\cite{SzemerediUNDSarkozy}, who showed that if all of the $v_i$ are \emph{distinct} then one can obtain the much stronger bound of $O(n^{-3/2})$, and by Stanley \cite{stanley} who determined the \emph{exact} maximum, using algebraic tools. A higher-dimensional version of this problem also received considerable attention (going under the name of \emph{the} Littlewood-Offord problem) and was studied by several authors \cite{sali,griggs,kleitman,katona} before it was ultimately resolved in the work of Frankl and F\"{u}redi~\cite{frankl1988solution} (see also~\cite{tao2012littlewood}). 

Of these early results, the most important for us here is the work of Hal\'{a}sz \cite{halasz} who made an important connection with the Fourier transform to prove (among other things) the following beautiful result: if there are $N_k$ solutions to $x_1 + \cdots + x_k = x_{k+1} + \cdots + x_{2k}$ among the entries of $v$, then $\rho(v) = O(n^{-2k-1/2}N_k)$.

More recently the question has been turned on its head by Tao and Vu \cite{tao-vu-ILO}, who pioneered the study of ``inverse'' Littlewood-Offord theory. They suggested that if $\rho(v)$ is ``large'' then $v$ must exhibit some particular arithmetic structure. For example, Tao and Vu~\cite{tao-vu-ILO, tao2010sharp}, and Nguyen and Vu \cite{nguyen-vu-1, nguyen-vu-2} proved that if $v$ is such that $\rho(v) > n^{-C}$ then $O(n^{1-\eps})$ of the elements $v_i$ of $v$ can be efficiently covered with a generalized arithmetic progression of rank $r = O_{\eps,C}(1)$. 

While these results provide a very detailed picture in the range $\rho(v) > n^{-C}$, they begin to break down\footnote{Technically these results break down if $\rho(v) < n^{-\log\log n}$.} if $\rho(v) = n^{-\omega(1)}$ and therefore are of limited direct use in showing that the singularity probability is exponentially small. Inverse results which work for smaller $\rho$ bring us to the ``counting'' Littlewood-Offord theorem of Ferber, Jain, Luh and Samotij \cite{ferber2019counting}, and the ``weak''
inverse Littlewood-Offord theorems of Campos, Mattos, Morris and Morrison \cite{CMMM} and of the present authors in \cite{RSM1}, which are useful for $\rho(v)$ as small as $\exp( -c(n\log n)^{1/2} )$, but afford less structure.

Our novel inverse Littlewood-Offord theorem in this paper is most similar to that of Rudelson and Vershynin \cite{RV}, also developed in \cite{vershynin-invertibility} and \cite{JSS-symmetric}, who showed that if $\rho(v) \gg e^{-cn}$ then there exists
$\phi >0 $ with\footnote{In what follows, we will be somewhat vague with our use of $\approx$.} $\phi \approx 1/\rho(v)$ for which the dilated vector $\phi \cdot v$ is exceptionally close to the integer lattice $\Z^n$. 
These Littlewood-Offord theorems, styled after Rudelson-Vershynin, tend to be a little bit subtler; instead of determining the structure of the whole vector, we only show there is some ``correlation'' with the rigid object $\Z^n$. 

To state our inverse Littlewood-Offord theorem, we formulate an important notion introduced by Rudelson and Vershynin. We switch to working in $\R^d$, for $d \approx cn$, hinting at the later context of these results. If $A \subseteq \R^d$ and $x \in \R^d$ then define $d(x,A) := \inf_{a\in A}\{ \|x-a\|_2 \}$. Now, for $\alpha \in (0,1)$, define
the \emph{least common denominator} of a vector $v \in \R^d$ to be the smallest $\phi > 0$ for which $\phi \cdot v$ is within $\sqrt{\alpha d}$ of a non-zero integer point. 
That is,
\[D_\alpha(v)= \inf\left\{\phi>0:~d(\phi \cdot v, \Z^d \setminus \{0\}) \leq \sqrt{\alpha d} \right\}.\]
Note here that we have excluded the origin from $\Z^d$ in the definition since $\phi \cdot v \approx 0$ does not tell us any interesting about $v$.
Indeed, given \emph{any}  $v \in \S^{d-1}$, one could always set $\phi < \sqrt{\alpha d}$ and obtain 
$d(\phi\cdot v, \Z^d) \leq d(\phi\cdot v,0) \leq \sqrt{\alpha d}$, and so this degenerate case needs to be excluded somehow. 
In fact, in the course of the paper, we will work with a slightly different non-degeneracy condition (see \eqref{eq:LCD-def}).  

Our Littlewood-Offord theorem shows that a similar conclusion to that of \cite{RV} can be obtained in the presence of a large number $(k \approx n)$ of additional ``soft'' constraints on the random walk. In particular we prove the following result, which is in fact weaker than what we really need (see Lemma~\ref{lem:CondWalkLCMfinal}), but captures its essence. We say that a random vector with entries in $\{-1,0,1\}$ is $\mu$-lazy if each entry is independent and is equal to $0$ with probability $1-\mu$ and is equal to each of $-1, 1$ with probability $\mu/2$. 

\begin{theorem}\label{thm:invLwO} There exist $R,c_1,c_2 >0$, for which the following holds for every  $d \in \N$, $\alpha \in (0,1)$, $0\leq k \leq c_1 \alpha d$ and $t \geq \exp(-c_1\alpha d)$. Let $v \in \S^{d-1}$, let $w_1,\ldots,w_k \in \S^{d-1}$ be orthogonal and let $W$ be the matrix with rows $w_1,\ldots,w_k$.

If $\tau \in \{-1,0,1 \}^d$ is a $1/4$-lazy random vector and
\begin{equation}\label{eq:thminvLwO} \PP\left( |\la \tau, v \ra| \leq t\, \text{ and }\, \|W \tau \|_2 \leq c_2\sqrt{k} \right) \geq R te^{- c_1 k},\end{equation}
then $D_{\alpha}(v) \leq 16/t$. \end{theorem}

Here we interpret $\| W\tau \|_2 \leq c_2\sqrt{k}$ as encoding the ``soft'' constraints and $|\la \tau, v \ra| \leq t $ as the ``hard'' constraint. It is useful to think 
of $t \approx \rho(v)$, although we actually set $t$ relative to a related notion tailored specifically to our application.

To understand the quantitative aspect of Theorem~\ref{thm:invLwO}, it is best to consider the contrapositive of Theorem~\ref{thm:invLwO}, which roughly
says that if $v$ is ``unstructured at scale $t$'' (that is, $D_{\alpha}(v) > 16/t$) then the soft and hard constraints are roughly negatively dependent\footnote{Here, we say events $S,T$ are negatively\ dependent if $\P(S\cap T) \leq \P(S)\P(T)$.}.
Indeed, if $v$ is sufficiently ``unstructured at scale $t$'' then we might expect $\la \tau,v \ra$ to approximate a Gaussian and, in particular, to have
\[ \PP( |\la \tau, v \ra| \leq t ) \approx  t. \]
On the other hand, since $w_1,\ldots, w_k \in \S^{d-1}$ are orthogonal, it turns out that (see Lemma~\ref{lem:HansonWright})
\[ \PP(\|W \tau \|_2 \leq c_2\sqrt{k} ) \leq e^{-c_1k}, \]
where $c_1>0$ is a suitably small constant depending on $c_2>0$. If these two events were negatively dependent then we would expect a bound of 
\[  \PP\left( |\la \tau, v \ra| \leq t\, \text{ and }\, \|W \tau \|_2 \leq c_2\sqrt{k} \right)   \leq te^{-c_1k}. \]
Theorem~\ref{thm:invLwO} says something \emph{almost} as strong as this, giving the inequality up to a constant $R$ and the value of $c_1$.

For us, the main difficulty lies in ``decoupling'' the soft and hard constraints, which is ultimately achieved by a somewhat complicated geometric argument on
the Fourier side. However, we should point out that Theorem~\ref{thm:invLwO} is non-trivial even in the case of $k=0$ and in fact reduces, in this case, to the inverse Littlewood-Offord result proved by Rudelson and Vershynin in \cite{RV}.

In fact, the $k=0$ case is useful for understanding the sort of structure that the conclusion $D_{\alpha}(v) < c/t$ provides. It is not hard to show that if one 
chooses $v \in \S^{n-1}$ very close to a point on the lattice $(Ct)\Z^n$, where $C \gg 1 $, then $v$ satisfies 
\begin{equation}\label{eq:anticonentation} \PP( |\la v, \tau \ra| \leq t )  \gg t . \end{equation} 
Thus the inverse theorem of \cite{RV,RV-ICM} says, roughly speaking, that \emph{all} vectors satisfying \eqref{eq:anticonentation} must have this structure. 
Our Theorem~\ref{thm:invLwO} says the same is true even in the presence of a large number of additional constraints.

\subsection{Proof sketch and a new ``inversion of randomness'' technique}\label{sec:sketch}
Here we briefly sketch how our inverse Littlewood-Offord result is used alongside a novel scheme for ``reusing randomness'' to prove Theorem~\ref{thm:main}.
As hinted at before, we will be helped along by treating the discretized sphere as a probability space, which will allow us to ``recover'' some of the randomness lost due to the symmetry of $A$. We keep our discussion here loose and impressionistic and we will take up our careful study in the following section. 

Our first move will be to ``locally replace'' $A$ with a random matrix $M$ that has many of the entries zeroed out. This will allow us to 
untangle some of the more subtle and complicated dependencies and has the advantage that various associated Fourier transforms are non-negative. Indeed 
let\footnote{Here we use the notation $[n] := \{1,\ldots,n\}$; for a vector $v \in \R^n$ and $S \subseteq [n]$, we use the notation $v_S := (v)_{i\in S}$
 and for a $n \times m$ matrix $A$, and $R \subseteq [m]$, we use the notation $A_{S \times R}$ for the $|S| \times |R|$ matrix $(A_{i,j})_{i \in S, j \in R}$. }
\begin{align}\label{eq:Mdef}
 M =  
\begin{bmatrix}
{\bf 0 }_{[d] \times [d]} & H_1^T \\
H_1 & { \bf 0}_{[d+1,n] \times [d+1,n]}  
\end{bmatrix}, \end{align}
where $d = cn$ and $H_1$ is a $(n-d) \times d$ random matrix with i.i.d.\ entries that are $\mu$-lazy, meaning that $(H_1)_{i,j} = 0$ with probability $1-\mu$ and $(H_1)_{i,j} = \pm 1 $ with probability $\mu/2$. We stress here that we cannot ``globally'' replace $A$ with $M$, and we may need to permute coordinates, depending on what part of the sphere
we are working on. 

We now follow the strategy of \cite{RV,Tikhomirov} and partition the sphere $\S^{n-1}$ based on the anti-concentration properties of the various $v \in \S^{n-1}$. Indeed, for each $v \in \S^{n-1}$, we find a corresponding ``scale'' $\eps \in (0,1)$ for which 
\begin{equation}\label{eq:scale} \PP( \| Mv\|_2 < \eps \sqrt{n} ) \approx (L\eps)^n, \end{equation}
where $L$ is a large constant. Notice here that we have defined this ``scale'' relative to the symmetric matrix $M$, rather than $A$ or $\rho(v)$, and so we expect it to capture the anti-concentration properties of $v$, specific to the matrix $M$. This $\eps$ should be interpreted as ``the scale at which the anti-concentration properties of $v$ just start to be felt'', as we imagine gradually decreasing $\eps$ from $1$ to $0$. For example, if $v$ is a random point on the sphere, it is not hard to see that $v$ will typically have $\eps \leq e^{-cn}$, which is in fact \emph{so} small that we can safely ignore $v$ (due to previous work). On the other hand, the constant vector $n^{-1/2}(1,\ldots,1)$ will have $\eps \approx n^{-1/2}$. Interestingly, this latter fact is not easy to establish rigorously, but is heuristically not hard to guess in analogy with the modified setting where $M$ has iid entries. 

We now study all vectors $v \in \S^{n-1}$ at a given scale $\eps \geq e^{-cn}$. While this is an uncountable set, we build an efficient $\eps$-net for these vectors in two steps. 
We first discretize the whole sphere by taking an $\eps$-net for $\S^{n-1}$, which we call $\L_{\eps}$. 
We can then say something like 
\[ \PP( Av = 0 \text{ for some } v \text{ at scale } \eps ) \leq \PP\left( \| Mv \|_2 \leq \eps \sqrt{n} \text{ for some } v \in \L_\eps \right). \]
One's first instinct might be to simply union bound over all $v \in \L_{\eps}$; however it turns out that even the most efficient epsilon nets have $|\L_{\eps}| \approx (C/\eps)^n$, which is too large to say anything.

The key insight here is that most of $\Lambda_{\eps}$ is not used when approximating 
$v \in \S^{n-1}$ \emph{at scale} $\eps$ and so we can refine our net $\Lambda_{\eps}$ by discarding all vectors $w \in \L_{\eps}$ with $\PP(\|Mv\|_2 \leq \eps \sqrt{n}) \ll (L\eps)^n$. So if we let $\cN_{\eps} \subseteq \L_{\eps}$ be the collection of vectors with $\PP( \|Mv\|_2 \leq \eps \sqrt{n}) \geq (L\eps)^n$, our problem reduces to showing that 
\begin{equation}\label{eq:Neps-goal} |\cN_{\eps}| \leq L^{-2n}|\L_{\eps}| \leq \left( \frac{C}{L^2\eps} \right)^n,\end{equation} 
which brings us to the technical heart of the paper (see Theorem~\ref{thm:netThm}). We point out that the factor of $L^{-2n}$, rather than $L^{-n}$, in \eqref{eq:Neps-goal} is important for us as it allows us to drown out the $L^n$ coming from \eqref{eq:scale} \emph{and} the factor $C^n$ in \eqref{eq:Neps-goal}, when we union bound over $\cN_{\eps}$.

To prove~\eqref{eq:Neps-goal} we take a probabilistic perspective inspired by \cite{Tikhomirov}; although we stress that our methods are considerably different. To show
\eqref{eq:Neps-goal} it is enough to show, for $v \in \L_{\eps}$ chosen uniformly at random
\begin{align}\label{eq:secondmomentsketch}
 \PP_{v \in \L_{\eps}}\left( v \in \cN_{\eps} \right) 
\approx \PP_{v \in \L_{\eps}} \bigg( \PP_M \left( \|Mv \|_2 \leq \eps \sqrt{n} \right) \geq (L\eps)^n \bigg) \leq L^{-2n}\, ,\end{align}
(see Lemma~\ref{thm:invertrandom}, for the rigorous statement).
To get a feel for how we tackle this, let us consider the event $\| Mv \|_2 \leq \eps n^{1/2}$. 
Indeed recalling~\eqref{eq:Mdef}, the definition of $M$, we have that
\[
Mv= \begin{bmatrix}
H_1 v_{[d]}  \\
 H_1^T v_{[d+1, n]}

\end{bmatrix}
\] 
and so to control the event $\|Mv\|_2 \leq \eps \sqrt{n}$, it is enough to control the intersection of events  $\|H_1 v_{[d]}\|_2  \leq \eps n^{1/2}$ and $ \| H_1^T v_{[d+1, n]}\|_2 \leq \eps n^{1/2}$. Note that if we simply ignore the second event and bound
\[ \PP( \|Mv\|_2 \leq \eps n^{1/2} ) \leq \PP( \| H_1 v_{[d]}\|_2 \leq \eps n^{1/2}),\] 
we land in a situation very similar to previous works; where half of the matrix is neglected entirely and we are thus limited by the $(n\log n)^{1/2}$ obstruction, described above. So to overcome this barrier, we need to control these two events simultaneously. 

The key idea here is to use the \emph{randomness in} $H_1$ to control the event $ \| H_1 v_{[d]} \|_2 \leq \eps n^{1/2}$ and to use the \emph{randomness in} $v \in \Lambda_{\eps} $ to control the event $ \| H_1^T v_{[d+1, n]} \|_2 \leq \eps n^{1/2}$. To get this to work, we crucially partition the outcomes in $H_1$, 
based on a robust notion of rank. Indeed, let
\[ \cE_k = \left\lbrace H_1 : \s_{d-k}(H_1)\geq c\sqrt{n} \text{ and } \s_{d-k+1}(H_1)<  c\sqrt{n} \right\rbrace\, , \]
where $\s_1(H_1)\geq \cdots \geq \s_d(H_1)$ denote the singular values of $H_1$. We may then bound $\P_{M}( \| Mv\|_2 \leq \eps n^{1/2})$ above by (only using the randomness in $H_1$, for the moment)
\begin{equation}\label{eq:ranksplit}
\sum_{k=0}^d  \P_{H_1}\Big( \|H_1^T v_{[d+1, n]}\|_2 \leq \eps n^{1/2}\, \big\vert\, \|H_1 v_{[d]}\|_2 \leq \eps n^{1/2} , \cE_k \Big) \cdot
\P_{H_1}\Big(\|H_1 v_{[d]}\|_2 \leq \eps n^{1/2}, \cE_k \Big)\, .\end{equation}
It is here that we can see the link with our inverse Littlewood-Offord theorem, Theorem~\ref{thm:invLwO}, which we use (after a good deal of preparation) to bound the probabilities 
\[ \P_{H_1}( \|H_1 v_{[d]}\|_2 \leq \eps \sqrt{n}, \cE_k),\]
that appear in \eqref{eq:ranksplit}. The event $\| H_1 v_{[d]}\|_2 \leq \eps n^{1/2}$ corresponds to the ``hard'' constraint  $|\la \tau, v \ra| \leq t$ in Theorem~\ref{thm:invLwO}, while the event $\cE_k$ corresponds to the ``soft'' constraint $\|W \tau \|_2 \leq c_2\sqrt{k}$, where we think of $\tau$ as a single row of $H_1$. 
And so, after a certain amount of work with Theorem~\ref{thm:invLwO}, we are able to conclude that 
\begin{align}\label{eq:invLwOtensor}
\P_{H_1}( \| H_1 v_{[d]} \|_2 \leq \eps \sqrt{n},\, \cE_k)\leq (C\eps e^{-c k})^{n-d}
\end{align}
\emph{unless}  $v_{[d]}$ is structured, in which case we do something different (and substantially easier). Thus, for all non-structured $v$, we have \eqref{eq:ranksplit} is roughly at most
\begin{equation}\label{eq:ranksplit2} 
(C\eps)^{n-d} \sum_{k=0}^d e^{-c k(n-d)}  \P_{H_1}\Big( \|H_1^T v_{[d+1, n]}\|_2 \leq \eps n^{1/2}\, \big\vert\, \|H_1 v_{[d]}\|_2 \leq \eps n^{1/2}, \cE_k\Big) .\end{equation}
Up to this point,  we have not appealed to the randomness in the choice of $v \in \Lambda_{\eps}$, beyond demanding that $v$ is non-structured. 
To see how we might take advantage of the randomness in $v$, let us consider the first moment of the quantity $\P_{M}( \| Mv\|_2 \leq \eps n^{1/2})$, which we view as a random variable in $v$.
Indeed, for $v\in \Lambda_\eps$ taken uniformly at random, we show that  
\begin{align}\label{eq:expvsketch}
 \E_{ v \in \Lambda_{\eps}}\, \P_{H_1}\left( \| H_1^T v_{[d+1, n]} \|_2 \leq \eps \sqrt{n} \mid \| H_1 v_{[d]}\|_2 \leq \eps n^{1/2}, \cE_k\right) \leq (C\eps)^{d-k}\, .
\end{align}
We establish this bound by swapping expectations, and bounding the probabilities 
\begin{align}\label{eq:projectionsketch}
\P_{v_{[d+1, n]}}( \| H_1^T v_{[d+1, n]}\|_2 \leq \eps n^{1/2} )\, ,
\end{align}
where $H_1$ is a \emph{fixed} matrix satisfying $\cE_k \cap \{H_1 : \| H_1 v_{[d]}\|_2 \leq \eps n^{1/2} \}$.
The idea here is that since $H_1$ has $d-k$  singular values of size $\approx n^{1/2}$, we should expect
\begin{equation}\label{eq:movingv} \PP_{v_{[d+1, n]}}( \| Hv_{[d+1,k]} \|_2 \leq \eps n^{1/2} )  \approx (C\eps)^{d-k}, \end{equation}
which is suggested, for example, by a Gaussian heuristic. This then results in the bound at~\eqref{eq:expvsketch}. See Section~\ref{subsec:AX-anticoncentration} for details on this step.
Putting \eqref{eq:expvsketch} and \eqref{eq:ranksplit2} together, and using that $\eps > e^{-cn}$, we have
\[
\E_v\, \P_{M}(\| Mv\|_2 \leq \eps n^{1/2} ) \leq (C\eps)^n.
\] Observe that the loss from the rank at \eqref{eq:movingv} is compensated by the gain 
afforded by the extra constraint added to our Littlewood-Offord step. 

While this is a good bound on the expectation, this is \emph{not} enough for our purposes, as the first moment only tells us, via Markov, that
\[ \PP_{v \in \L_{\eps}} \bigg( \PP_M \left( \| Mv\|_2 \leq \eps n^{1/2} \right) \geq (L\eps)^n \bigg) \leq L^{-n}\, ,
\] falling short of our desired $L^{-2n}$ bound. 

So to prove our result, we instead study\footnote{Actually, we need a slight variant, where we cut out structured vectors.} the \emph{second moment} of $\PP_M ( \| Mv\|_2 \leq \eps n^{1/2} )$,
\[ \E_v \left( \PP_M ( \| Mv\|_2 \leq \eps n^{1/2} ) \right)^2 , \]
in much the same way, but with a few added technicalities. 

To say a few words about how the second moment is different, we will see (Fact~\ref{fact:2ndMoment})
\[
\left( \PP_M \left( \|Mv\|_2 \leq \eps n^{1/2} \right) \right)^2\leq \P( \|H_1v_{[d]}\|_2 \leq \eps n^{1/2}, \|H_2 v_{[d]}\|_2 \leq \eps n^{1/2}
 \text{ and } \| H^Tv_{[d+1, n]}\|_2 \leq 2 \eps n^{1/2})\, ,
\] where $H_2$ is an independent copy of $H_1$ and $H:=[H_1, H_2]$ is the concatenation of these two matrices. 
We then proceed in much the same way as above, but treating $H$, in place of $H_1$, and carrying forward the two ``hard'' constraints resulting from the two copies of $v_{[d]}$.
This explains the shape of our ``real'' inverse Littlewood-Offord theorem, Lemma~\ref{lem:CondWalkLCMfinal}, where we allow for these two hard constraints. 
Ultimately, we arrive at the bound 
\[
\E_v \left( \PP_M ( \| Mv\|_2 \leq \eps n^{1/2} ) \right)^2 \leq (C\eps)^{2n},
\] 
which implies the desired conclusion at~\eqref{eq:secondmomentsketch}.

\subsection{A few remarks on presentation}

This paper is roughly divided into three parts. The first part consists of Sections~\ref{sec:ILwO-CondWalks}-\ref{sec:ILwO-CondWalksIII} which are dedicated to proving our conditioned inverse Littlewood-Offord result, Lemma~\ref{lem:CondWalkLCMfinal}, which is the ``real'' version of Theorem~\ref{thm:invLwO}. These sections lay the groundwork for Section~\ref{sec:matrixWalks},
where we prove Theorem~\ref{lem:rankH}, which is the only result we carry forward into later sections. 

The second part of the paper consists of Section~\ref{sec:Size-Net} and Section~\ref{sec:approx-w-net}. In Section~\ref{sec:Size-Net}, we obtain our crucial
bound on the size of our net $\cN_{\eps}$ using our novel ``inversion of randomness'' technique, as outline above. On the other hand, Section~\ref{sec:approx-w-net} contains the less exciting proof that $\cN_{\eps}$ is in fact a net for $\Sigma_{\eps}$.

In the final section, Section~\ref{sec:ProofOthm}, we pull together the various elements of this paper, state the reductions we will use from previous work
and prove Theorem~\ref{thm:main}.

In most cases, we have highlighted the main results of each section at the start. So if one does not want to delve into the details of a particular element of the proof, one can simply inspect the top of the section to glean what is needed for going forward.

\section{Central Definitions}\label{sec:Definitions}
We now turn to give a proper treatment of the proof, by laying out the key definitions that will concern us in this paper. 
 We begin by partitioning the sphere $\S^{n-1}$ into ``structured'' and ``unstructured'' vectors. Formally, we set $\g = e^{-cn}$, for sufficiently small $c>0$, and then define the ``structured'' vectors as
\[ \Sigma := \left\lbrace v \in \S^{n-1} : \rho(v) \geq \g \right\rbrace. \]
The invertibility of a random symmetric matrix on the set of ``unstructured'' vectors $ v \in \S^{n-1} \setminus \Sigma$ is already well understood and so we can restrict our attention to this set of structured vectors. We refer the reader to Section~\ref{sec:ProofOthm} for the details here.

Following Rudelson and Vershynin \cite{RV}, we make a further reduction to working with vectors that are reasonably ``flat'' on a large part of their support. For $D \subseteq [n]$, with $|D| = d$, we define 
\begin{equation}\label{eq:defI} \cI(D) := \left\lbrace v \in \S^{n-1} :  (\k_0 + \k_0/2) n^{-1/2} \leq |v_i| \leq  (\k_1 - \k_0/2) n^{-1/2} \text{ for all } i\in D   \right\rbrace,\end{equation}
where $0 < \k_0 < 1 < \k_1$ are absolute constants, fixed throughout the paper and defined in Section~\ref{sec:constants}. We will set $d := c_0^2 n/2$, where $c_0$ is defined below in Section~\ref{sec:constants}.

Now set
\[ \cI := \bigcup_{D \subseteq [n], |D| = d } \cI(D). \]
The case of non-flat $v$ is already taken care of in the work of Vershynin \cite{vershynin-invertibility} (see Section~\ref{sec:ProofOthm}) and so it is enough to work with $\cI \cap \Sigma$. 
Since we will ultimately union bound over $D$, it is enough to work with $\cI(D) \cap \Sigma$, for \emph{some} fixed set $D$, and so, by symmetry it is enough to restrict our attention to vectors $ v\in \cI([d]) \cap \Sigma$. 

Now, with this in mind, we further partition the set $\cI([d]) \cap \Sigma \subseteq \S^{n-1}$, but for this we need to introduce another 
distribution on symmetric matrices. Define the probability space $\cM_n(\mu)$ by defining $M \sim \cM_n(\mu)$ to be the random matrix 
\[ M =  
\begin{bmatrix}
{\bf 0 }_{[d] \times [d]} & H_1^T \\
H_1 & { \bf 0}_{[d+1,n] \times [d+1,n]}  
\end{bmatrix}, \] 
where $H_1$ is a $(n-d) \times d$ random matrix with i.i.d.\ entries that are $\mu$-lazy (that is, $(H_1)_{i,j} = 0$ with probability $1-\mu$ and $(H_1)_{i,j} = \pm 1 $ with probability $\mu/2$). In fact, we will fix $\mu =1/4$ throughout the paper.

Now, given $v\in \cI([d])$ and $L>0$, in the spirit of \cite{Tikhomirov}, we define the \emph{threshold} 
\[ \cT_L(v)= \sup\left\{t\in[0,1]:  \PP(\|Mv\|_2\leq t\sqrt{n}) \geq (4Lt)^n\right\}, \]
or the ``scale'' of $v$, as we called it in Section~\ref{sec:sketch}. Observe carefully here that we are defining $\cT_L$ \emph{relative to the matrix} $M$, rather than our original distribution $A$. 

We may now define our partition of $\cI([d]) \cap \Sigma$. For $\eps\in (0,1)$, let
\[ \Sigma_\eps :=\left\{v\in \cI([d]) : \cT_L(v)\in [\eps,2\eps] \right\}\, .\]
We shall show (as it is not obvious) that indeed 
\[ \Sigma \cap \cI([d]) \subseteq  \bigcup_{\eps > \g^4} \Sigma_{\eps}. \]
With the definition of $\Sigma_{\eps}$ in hand, we are able to define $\cN_{\eps}$ which will be an efficient net for $\Sigma_{\eps}$ at scale $\eps$. It turns out that \emph{defining} this net is not hard, although showing that it satisfies the desired properties will be the main challenge of this paper. For this, we first define the \emph{trivial net at scale} $\eps$ to be\footnote{Here and throughout, $B_n(x,r)$ is the $\ell^2$ ball centered at $x$ with radius $r$.} 
\[ \L_{\eps} := B_n(0,2) \cap \big( 4\eps  n^{-1/2} \cdot \Z^n\big) \cap \cI'([d]), \]
which is a natural net for $\cI([d])$. Where $\cI'(D)$ is similar to $\cI(D)$ but with slightly looser constraints and relative to $\R^n$;
\[ \cI'(D)  := \left\lbrace v \in \R^{n} :  \k_0 n^{-1/2} \leq |v_i| \leq  \k_1 n^{-1/2} \text{ for all } i\in D   \right\rbrace. \]

Since we are only interested in approximating vectors in $\Sigma_{\eps}$, we can get away with a significantly more efficient net.
For this we introduce two more concentration functions. First, we define the \emph{L\'evy concentration function}: if $X$ is a random vector taking values in $\R^n$, define 
\[ \cL(X,t) := \max_{w \in \R^n} \PP\left( \| X - w \|_2 \leq t \right). \]
Second, we define a variant of this concentration function for the uniform distribution on random symmetric matrices with capped operator 
norm\footnote{For a matrix $A$, we use the notation $\|A\| := \max_{x : \|x\|_2 =1} \|A\|_2$ to denote the usual $2 \rightarrow 2$ operator norm.}.

\[ \cL_{A,op}(v,t) := \max_{w \in \R^n} \PP\left(\{ \| Av - w \|_2 \leq t \}  \cap \{ \|A\| \leq 4\sqrt{n} \} \right). \]

Here we are just cutting out the slightly irritating event that $A$ has large operator norm. Intuitively this is an acceptable move as the probability that
$\|A\| \geq 4\sqrt{n}$, is exponentially small (see Lemma~\ref{lem:op-concentration}), however some care is needed as we are mostly concerned with far less likely events. 

We now introduce our nets $\cN_{\eps}$,
\[ \cN_{\eps}  := \left\lbrace  v \in \L_{\eps} :  \PP(\|Mv\|_2\leq 4\eps\sqrt{n}) \geq (L\eps)^n \text{ and }  \cL_{A,op}(v,\eps\sqrt{n}) \leq (2^8L\eps)^n \right\rbrace. \] 
The reader should view the lower bound $\PP(\|Mv\|_2\leq 4\eps\sqrt{n})\geq (L\eps)^n$ as the real core of this definition, while the upper bound for 
$\cL_{A,op}$ is less important. The two main tasks of this paper will be to show that $\cN_{\eps}$ is indeed a net for $\Sigma_\eps$ (an easier task) and secondly that $|N_{\eps}|/|\L_{\eps}|$ is smaller than $\approx L^{-2n}$, where $L$ is a large constant.

\subsection{Discussion of constants and parameters}\label{sec:constants}
We will treat the constants $\k_0,\k_1$ (seen at \eqref{eq:defI}) as absolute throughout the paper, and we allow other absolute constants $C,C',\cdots $ to depend on these exact quantities. In particular, we set $\k_0 = \rho$ and $\k_1 = \delta^{-1/2}/2$, where $\delta,\rho$ are as in Lemma~\ref{lem:compressible} (which is a lemma from \cite{vershynin-invertibility}). While we have not computed these constants, it would not be too much work to do so.

We also note our treatment of $c_0$, which, for most of the paper, will be presented as a parameter and dependencies involving $c_0$ will be explicitly noted.
However, we will ultimately fix $c_0 = \min\{ 2^{-24}, \rho \delta^{1/2} \}$ where, again, $\delta,\rho$ are as in Lemma~\ref{lem:compressible}. Thus it is no harm 
for the reader to view $c_0$ as an absolute constant which is fixed throughout the paper. The reason for the extra care with $c_0$ comes from its delicate relationship to 
$d/n$. Indeed, we will ultimately set $d := \lceil c_0^2 n/2 \rceil$.

Another point to note is our use of $R$, which represents related, but different constants throughout the paper. Roughly speaking, these related values of $R$ increase as we get deeper into the proof.

\section{Inverse Littlewood-Offord for conditioned random walks I: Statement of result and setting up the proof}\label{sec:ILwO-CondWalks}

This section is the first of three sections where we lay out and prove our main inverse Littlewood-Offord type theorem, Lemma~\ref{lem:CondWalkLCMfinal}, which works in the presence of a large number ($k \approx n $) of relatively soft constraints on our random walk. 
As mentioned before, the conclusion of our Littlewood-Offord theorem will be similar to that of Rudelson and Vershynin \cite{RV}, who showed
that vectors $v$, for which the random walk $\la v, \tau \ra$ concentrates, admit non-trivial least common denominators. As we will see,
the proof of Lemma~\ref{lem:CondWalkLCMfinal} is rather involved and consists mainly of a geometric argument on the Fourier side to ``decouple'' the many soft constraints from the few hard constraints.

Given a $2d \times \ell$ matrix $W$ (which encodes these soft constraints on our walk, as in Theorem~\ref{thm:invLwO}) and a vector $Y\in \R^d$, we define the $Y$-augmented matrix $W_Y$ as 
\begin{equation}\label{eq:WYdef} W_Y = \begin{bmatrix}\,  W  ,\, \begin{bmatrix} { \bf 0}_d \\ Y \end{bmatrix} ,\, \begin{bmatrix} Y \\ { \bf 0}_d \end{bmatrix} \end{bmatrix} .  \end{equation}
Here $Y \approx v/t$ will be a re-scaled version of $v$ from Theorem~\ref{thm:invLwO}.
We define, for $\alpha \in (0,1)$, the \emph{least common denominator} of a vector $v\in \R^d$ to be
 \begin{equation} \label{eq:LCD-def} D_\alpha(v):=\inf\left\{\phi>0:~\|\phi\cdot v\|_{\T}\leq \min\left\{ \phi\|v\|_2/2 , \sqrt{\alpha d}\right\}\right\},\end{equation} 
 where $\|x\|_{\T} := \inf\{ \|x - y\|_2 : y \in \Z^d \}$, for $x \in \R^d$, denotes the minimum distance to an integer point. 
 Note the definition at \eqref{eq:LCD-def} is a bit different from the definition presented in the introduction, in that the ``non-degeneracy condition'' is now $\|\phi\cdot v\|_{\T}\leq \phi\|v\|_2/2$.  We will stick with this definition throughout the paper. 
 
We let $\|A\|_{\HS}$ denote the Hilbert-Schmidt norm of a matrix $A$, that is, $\|A\|_{\HS}^2 :=  \sum_{i,j} |A_{i,j}|^2$ and for $\mu \in (0,1)$, $m \in \N$, define the $m$-dimensional $\mu$-lazy random vector $\tau \sim \Q(m,\mu)$ to be the vector with independent entries $(\tau_i)_{i=1}^m$, satisfying
\[ \Pr(\tau_i=-1)=\Pr(\tau_i=+1)=\mu/2 \qquad \text{ and } \qquad \Pr(\tau_i=0)=1-\mu. \]
We now state our main inverse Littlewood-Offord type theorem, which is our ``real'' (and strengthened) version of Theorem~\ref{thm:invLwO}, from Section~\ref{subsec:LwO-intro}.

\begin{lemma}\label{lem:CondWalkLCMfinal}
For $d \in \N$ and $\alpha \in (0,1)$, let $0\leq k\leq 2^{-10}\alpha d$ and $t \geq \exp(-2^{-9}\alpha d)$. For $0< c_0 \leq 2^{-24}$, let $Y \in \R^d$ satisfy $\| Y \|_2 \geq 2^{-10} c_0/t$,
let $W$ be a $2d \times k$ matrix with $\|W\| \leq 2$ and $\|W\|_{\HS}\geq \sqrt{k}/2$. 

If $\tau \sim \cQ(2d,1/4)$ and $D_\alpha(Y)> 16$ then 
\begin{equation}\label{eq:LCM-hypo}  
\cL\left( W^T_Y\tau, c_0^{1/2} \sqrt{k+1} \right) 
\leq \left( R t\right)^2 \exp(-c_0 k), \end{equation} where $R = 2^{32}c_0^{-2}$.
\end{lemma}

\vspace{3mm}
Before we start working towards the proof of Lemma~\ref{lem:CondWalkLCMfinal}, we make a few informal remarks on its statement and its connection to Theorem~\ref{thm:invLwO}. The main difference to note is that there are now two ``hard'' constraints encoded in the left-hand side of \eqref{eq:LCM-hypo}; these are, in the notation of Theorem~\ref{thm:invLwO},
\[ |\la (v,0_{[d]}) ,\tau \ra| < t  \text{ and  } |\la (0_{[d]}, v ) ,\tau \ra| < t .\]
The ``soft'' constraints are now encoded as the columns $w_1,\ldots,w_k$ of $W$. 

To combine the ``hard'' and ``soft'' constraints into a single matrix inequality, we rescale $v$, 
 thinking of $|\la (v,0_{[d]}) ,\tau \ra| < t$ as $|\la c_0^{1/2} t^{-1} (v,0_{[d]}) ,\tau \ra| < c_0^{1/2}$. This explains the scaling on $Y$, which is unusually written as $\| Y \|_2 \geq 2^{-10} c_0/t$, where $t$ should be thought of a very small number $\approx e^{-cn}$.  
 
The scaling of $D_{\alpha}(Y)$ in Lemma~\ref{lem:CondWalkLCMfinal}, in contrast with the statement of Theorem~\ref{thm:invLwO}, is explained in a similar way. If
  $\phi \cdot Y  \sim \Z^d$, where $\phi = O(1)$ then $(\phi/t) = O(1/t)$ satisfies $(\phi/t) \cdot v \sim \Z^d$, as we think of $Y \approx v/t$.

This also makes the numerology of Lemma~\ref{lem:CondWalkLCMfinal} a little more transparent. If $Y$ is a random vector with $\|Y\|_2 \approx 1/t$, we
have $|Y_i| \approx t^{-1}n^{-1/2}$ and thus we expect the one dimensional random walk $\la Y, \tau \ra $ to have
\[ \cL\left( \la Y, \tau \ra , c_0^{1/2} \right) \approx t. \]
Thus we expect $Y$ to have some special structure if $\cL\left( \la Y, \tau \ra , c_0^{1/2} \right) \gg t$. On the other hand, for each $w_i$ we expect that $|\la w_i ,\tau \ra| \approx 1 $ and, since the $w_i$ must be ``approximately orthogonal'' (due to the assumption $\|W\| \leq 2$), we should expect
\[ \cL\left( W \tau , c_0^{1/2}\sqrt{k} \right) \approx e^{-ck}, \]
being somewhat vague about this constant $c>0$.
Second, note that Lemma~\ref{lem:CondWalkLCMfinal} is still interesting even in the case $k=0$, where it is not hard to see that it reduces to 
\[ \cL\left( \la Y, \tau \ra ,\, c_0^{1/2}  \right)  \leq  R t ,\]  
whenever $D_{\alpha}(Y) \leq 16$, which is essentially the statement of the main inverse Littlewood-Offord theorem of Rudelson and Vershynin in \cite{RV}.

Finally, we point out that the contrapositive of Lemma~\ref{lem:CondWalkLCMfinal} is more conducive to the ``inverse Littlewood-Offord'' reading:

\vspace{2mm}

\begin{center}
if \ $\cL(W^T_Y\tau, c_0^{1/2}\sqrt{k+1}) \geq \left( R t\right)^2 \exp(-c_0 k)$ \ then \ $D_{\alpha}(Y) \leq 16$. 
\end{center}

\vspace{2mm}

For the remainder of this section, we take some first steps towards the proof of Lemma~\ref{lem:CondWalkLCMfinal}. We first pass to the Fourier side and set up 
our problem there, describing our goal in terms of a certain ``level set''. We then make a first reduction, by getting some basic control on the fibers of this level set. In the following section, Section~\ref{sec:ILwO-CondWalksII}, we make a more significant reduction about the geometry of our level set. In Section~\ref{sec:ILwO-CondWalksIII} we prove the key Lemma~\ref{lem:CondWalkLCM}, the statement of which is very similar to that of Lemma~\ref{lem:CondWalkLCMfinal}, but with a more complicated quantity replacing the right-hand side of \eqref{eq:LCM-hypo}. Finally, with one further step, we conclude Section~\ref{sec:ILwO-CondWalksIII}, with the proof of Lemma~\ref{lem:CondWalkLCMfinal}.

\vspace{3mm}

\subsection{Passing to the Fourier side}\label{sec:go-Fourier}

To prove Lemma~\ref{lem:CondWalkLCMfinal} we will prove the contrapositive; assume \eqref{eq:LCM-hypo} fails and then obtain an upper bound on the least 
common denominator by finding a non-trivial $\phi >0 $  that satisfies $\phi = O(1)$ and $\| \phi \cdot Y \|_{\T} \leq \sqrt{\alpha d}$.
Our first step in proving Lemma~\ref{lem:CondWalkLCMfinal} is to use the lower bound in the negation of \eqref{eq:LCM-hypo} to obtain a lower bound on a level set of an appropriate Fourier transform. This manoeuvre was pioneered by Hal\'{a}sz \cite{halasz} and has been a key step in all of the Fourier approaches to inverse Littlewood-Offord theory.

For a $2d \times \ell$ matrix $W$, we define the $W$-\emph{level set}, for $t \geq 0$, to be
\[ S_W(t) := \left\lbrace \theta \in \R^{\ell} : \|W\theta\|_{\T} \leq \sqrt{t} \right\rbrace \] 
and we define $\g_\ell$ to be the $\ell$ dimensional Gaussian measure defined by $\g_\ell(S) = \PP(g \in S)$, where $g \sim \cN(0, (2\pi)^{-1} I_{\ell})$ and $I_\ell$ denotes the $\ell\times\ell$ identity matrix.

The following Esseen-type lemma, allows us relate the quantity seen at the left-hand side of \eqref{eq:LCM-hypo} with the Gaussian volume of a  level-set.

\begin{lemma} \label{lem:esseen} Let $\beta >0$, $\nu \in (0,1/4]$, let $W$ be a $2d \times \ell$ matrix and let $\tau\sim\cQ(2d, \nu)$. Then there exists $m>0$ so that 
\[\cL(W^T\tau, \beta\sqrt{\ell}) \leq 2\exp\left(2\beta^2 \ell -\nu m/2 \right) \gamma_{\ell}(S_W(m)). \]
\end{lemma}

The proof of this Lemma is a straightforward exercise with the characteristic function of $W^T\tau$ and is postponed to Appendix~\ref{sec:FourierPrep}.

We can now describe how our least common denominator can be spotted in Fourier space. From Lemma~\ref{lem:esseen} along with the negation of \eqref{eq:LCM-hypo}, we obtain $m >0$ and a set $S_{W_Y}(m) \subseteq \R^{k+2}$ with Gaussian volume bounded below by $(Rt)^2\exp(c_1m-c_2k)$. Now, for reasons that we will not explain here (since it is just a consequence of the Fourier transform), the first $k$-coordinates of the space, correspond to the $k$ ``soft'' constraints while the final two coordinates correspond to the two ``hard'' constraints. 

With this in mind, the idea is to find an element $\psi \in S_{W_Y}(m)$ for which $\| \psi_{[k]} \|_2 = O(\sqrt{k})$, and one of $\psi_{k+1},\psi_{k+2} $ is $O(1)$ and ``non-trivial''.
It will turn out that one of $\psi_{k+1},\psi_{k+2}$ is a good candidate for our desired least common denominator. The condition on the $\psi_{[k]}$ should be thought of as 
just getting these coordinates ``out of the way''. 

To find this desired $\psi \in S_{W_Y}(m)$, for $r,s>0$, we define the \emph{cylinder}
\begin{equation} \label{eq:defCy} 
\G_{r,s}:=\left\{\theta \in \R^{k+2} : ~\left\|\theta_{[k]} \right\|_2\leq r, |\theta_{k+1}|\leq s \text{ and } |\theta_{k+2}|\leq s\right\}.\end{equation} 

We now restate our condition on $\psi$ in terms of $\G_{r,s}$: we want to show that there exists an $x \in S_{W_Y}(m)$ for which 
\begin{equation}\label{eq:goal-in-Gamma}
  (\G_{2\sqrt{k},16} \setminus \G_{2\sqrt{k},s} + x) \cap S_{W_Y}(m)  \not= \emptyset , 
\end{equation}
where $s$ is chosen depending on the non-triviality condition we need.  We shall then ultimately see that if $y \in (\G_{2\sqrt{k},16} \setminus \G_{2\sqrt{k},s} + x)$,
where $x \in S_{W_Y}(m)$, then $(x-y)$ is a good candidate for $\psi$ (see Claims~\ref{claim:condWalkLCM1}-\ref{claim:CondRW3}). In what remains in this section, we warm up by making a first easy reduction on the structure of $S_{W_Y}(m)$ under the assumption that \eqref{eq:goal-in-Gamma} fails.

\subsection{A first reduction: controlling the density on fibers}\label{sec:geoOfS-1stReduction}

For our first reduction, we first record the following easy fact.

\begin{fact}\label{fact:2dclosepoints} For $s >0$, let $S \subseteq \R^2$ be such that $\g_2(S) \geq 8s^2$, then there exists $x,y\in S$ so that $s <\|x-y\|_\infty\leq 16$.
\end{fact}
\begin{proof}
We prove the contrapositive and assume there is no pair $x,y \in S$ with $s < \| x - y \|_{\infty} \leq 16$.  We cover $\R^2 = \bigcup_{p\in 16\cdot \Z^2}Q_p$ where 
$Q_p :=p+[-8,8]^2$. Thus 
$ \g_2(S)\leq\sum_{p\in 16\cdot\Z^2}\g_2(S\cap Q_p)$. Since there is no $x,y\in S$ so that $s < \|x-y\|_\infty \leq 16$, then for each $Q_p$ there is 
$x = x(p)\in Q_p$ so that \[\gamma_2(S\cap Q_p)\leq\g_2(S\cap Q_p \cap (x(p)+[-s,s]^2)) \leq \g_2(x(p)+[-s,s]^2). \] 
Letting $g \sim \cN(0,(2\pi)^{-1})$, we have
 \[\g_2(x+[-s,s]^2)\leq\Pr( x_1-s\leq g\leq x_1+s)\Pr( x_2-s\leq g\leq x_2+s)\leq 4s^2\exp(-\pi\|p\|_2^2/16),\]
where we have used that $(x_i-s)^2\geq p_i^2/8$, which holds since we may assume that $s\leq 1$ (else the statement holds trivially). Now we may bound
  \[\gamma_2(S)\leq\sum_{p\in 16\cdot\Z^2}\gamma_2(S\cap Q_p)\leq 4s^2\sum_{p\in 16\cdot\Z^2}\exp(-\pi\|p\|_2^2/16)< 8s^2,\]\
which completes the proof.
\end{proof}

\vspace{3mm}

Now for $S \subseteq \R^{k+2}$, and $\t_{[k]} \in \R^{k}$, we define the ``vertical fiber'' 
\begin{equation}\label{eq:defS()} S(\t_{[k]}) := \left\lbrace (\t_{k+1},\t_{k+2}) \in \R^2 : (\t_{[k]},\t_{k+1},\t_{k+2}) \in S \right\rbrace .  \end{equation}

The following lemma tells us that if we are unable to find a point in our desired intersection $\left(\Gamma_{r,16} \setminus \Gamma_{r,s} + x\right) \cap S$,
for all $x\in S$, we can obtain good control on the measure of the vertical fibers of $S$.

\begin{lemma}\label{lem:slice-upperBound}
	For $k \in \N$, $r > 0$ and $s > 0$, let $S \subset \R^{k+2}$ be such that for all $x \in S$ we have $$\left(\Gamma_{r,16} \setminus \Gamma_{r,s} + x\right) \cap S = \emptyset\,.$$ 
	Then $$\max_{\theta_{[k]}\in \R^k} \gamma_2 ( S(\theta_{[k]})) \leq 8 s^2\,.$$
\end{lemma}
\begin{proof}
	We prove the contrapositive; let $\psi_{[k]}$ be such that $\g_2\left( S(\psi_{[k]}) \right) > 8s^2$. 
	This implies (Fact~\ref{fact:2dclosepoints}) that there exists $(\t_{k+1},\t_{k+2}), (\t'_{k+1},\t'_{k+2}) \in S(\psi_{[k]})$ with
	\[  s \leq \max\{ |\t_{k+1}-\t'_{k+1}|,  |\t_{k+2} - \t'_{k+2}|\}\leq 16 \,.\]
	Unpacking what this means in the full space $\R^{k+2}$: we have $\t,\t' \in S$ so that 
	$\t_{[k]},\t'_{[k]} = \psi_{[k]}$, and $s \leq \max\{ |\t_{k+1}-\t'_{k+1}|,  |\t_{k+2} - \t'_{k+2}|\} \leq 16$. Thus 
	\[  \theta \in ( \theta' + \G_{r,16}\setminus \G_{r,s}), \]
	as desired.\end{proof}

\vspace{3mm}

In the next section we go on to obtain a more complicated reduction of this form, that will ultimately be key in proving Lemma~\ref{lem:CondWalkLCMfinal}.

\section{Inverse Littlewood Offord II: A geometric inequality}\label{sec:ILwO-CondWalksII} 

We now turn to make a more intricate and subtle reduction from that seen in Section~\ref{sec:geoOfS-1stReduction}, that will be key in finding our least common denominator. The lemma we prove here is purely geometric, but one should always think of it as being applied to an appropriate level set $S = S_{W_Y}(m)$, as seen in 
Lemma~\ref{lem:esseen}.

Given a set $S \subset \R^{k+2}$ and $y \in \R^{k+2}$, define the ``translated horizontal fiber'',
\[ F_y(S;a,b) := \{\theta_{[k]} = (\theta_1,\ldots,\theta_k) \in \R^k : (\theta_1,\ldots,\theta_k,a, b) \in S - y \}\,. \]

Our main goal of this section tells us that under the assumption
\[ (\Gamma_{2\sqrt{k},16} \setminus \Gamma_{2\sqrt{k},s} + x) \cap S = \emptyset,\]
for all $x \in S$,  the total measure of $S$ can be controlled by the measure 
of the $k$-dimensional fibers $F_{y}(S;a,b)$. We state it in the contrapositive form to make the application (in Section~\ref{sec:ILwO-CondWalksIII}) 
a little easier to spot.

\begin{lemma}\label{lem:geo-compare}
	For $k\in \N$ and $s >0$, let $S \subset \R^{k+2}$ be a measurable set which satisfies	
	\begin{equation}\label{eq:geo-comp-cond} 8s^2 e^{-k/8} + 64 s^2 \max_{a,b,y} \big( \g_{k}(F_y(S;a,b) - F_y(S; a,b))\big)^{1/4} < \g_{k+2}(S)\,.\end{equation}
	Then there is an $x \in S$ so that\footnote{Note, in particular, that Lemma~\ref{lem:geo-compare} says that \emph{if} \eqref{eq:geo-comp-cond} is satisfied then we must have $s < 16$.} 
	\begin{equation} \label{eq:geo-comp-conclude} (\Gamma_{2\sqrt{k},16} \setminus \Gamma_{2\sqrt{k},s} + x) \cap S \neq \emptyset\,.\end{equation}
\end{lemma}

To prove this lemma, we will need a few facts about Gaussian space, which we collect in Sections~\ref{sec:gauss-space} and \ref{sec:BM}, before moving on 
to prove Lemma~\ref{lem:geo-compare} in Section~\ref{sec:pf-geo-compare}.

\subsection{A few facts about Gaussian space}\label{sec:gauss-space}
Recall that for $\ell\in\N$, $\g_\ell$ is the $\ell$ dimensional Gaussian measure defined by $\g_\ell(S) = \PP(g \in S)$, where $g \sim \cN(0, (2\pi)^{-1} I_{\ell})$.

\begin{lemma}\label{fact:covFact}
Let $k\geq 0$,  $r>0$ and $S \subset \R^{k+2}$ be measurable. Then there exists $x\in S$, and $h\in \G_{r,8}$ so that 
$$\g_{k+2}(S \cap B) \leq 8 \g_{k+2}((S - x)\cap \G_{2r,16}+h)\, ,$$ where $B := \{ \t \in \R^{k+2} : \|\t_{[k]}\|_2 \leq r \}$.
\end{lemma}
\begin{proof}
Consider translates $\G_{r,8} +y $ where $y_{k+1},y_{k+2}\in 16 \Z^2$ to write
\begin{equation}\label{eq:Btile}
	\g_{k+2}(S \cap B) = \sum_{y \in \{0\}^{k} \times 16 \Z^2} \gamma_{k+2}(S \cap (\G_{r,8} + y ))\,.
\end{equation} We express $\g_{k+2}(S \cap (\G_{r,8} + y ))$ as 
\begin{equation}\label{eq:covfact1} \int_{\R^{k+2}}\1\big[ \theta \in S \cap (\G_{r,8} + y )\big] e^{-\pi\|\theta\|_2^2/2}\, d\t
= \int_{\R^{k+2}}\1\big[\phi \in (S-y) \cap  \G_{r,8}\big]e^{-\pi \|\phi +y\|_2^2/2 }\, d\phi. \end{equation}
Rewriting the exponent in the integrand at \eqref{eq:covfact1}
\[ - \|\phi +y\|_2^2= -\| \phi\|_2^2 
 - 2\phi_{k+1}y_{k+1} - 2\phi_{k+2}y_{k+2} - y_{k+1}^2 - y_{k+2}^2, \]
we use that $|\phi_{k+1}|,|\phi_{k+2}|\leq 8$ whenever $\1[\phi \in (S-y) \cap  \G_{r,8}]\neq 0$, to see  
\begin{equation} \label{eq:covfact2} \gamma_{k+2}(S \cap (\G_{8,r} + y)) \leq \exp\left(-\frac{\pi}{2} y_{k+1}^2 -\frac{\pi}{2} y_{k+2}^2 + 8\pi |y_{k+1}| + 8\pi |y_{k+2}|\right) \gamma_{k+2}((S - y) \cap \G_{r,8})\, .\end{equation}
So, apply \eqref{eq:covfact2} to \eqref{eq:Btile} to get
	\begin{align*}
	\gamma_{k+2}(S \cap B) &\leq 
	\sum_{y \in \{0\}^k \times 16 \Z^2}   \g_{k+2}((S - y) \cap \G_{r,8}) e^{-\frac{\pi}{2} y_{k+1}^2 -\frac{\pi}{2} y_{k+2}^2 + 8\pi |y_{k+1}| + 8\pi |y_{k+2}|} \\
	&\leq \max_y \gamma_{k+2}((S - y) \cap \G_{r,8}) \sum_{y_{k+1}, y_{k+2} \in 16 \Z}e^{-\frac{\pi}{2} y_{k+1}^2 -\frac{\pi}{2} y_{k+2}^2 + 8\pi |y_{k+1}| + 8\pi |y_{k+2}|}\\
	&\leq  16 \max_y \gamma_{k+2}((S - y) \cap \G_{r,8})\, .
	\end{align*}
	Let $y$ be a vector at which the above maximum is attained. Now observe that if $S \cap (\G_{r,8}+y) = \emptyset$ then $(S - y) \cap \G_{r,8} = \emptyset$
	and thus $\g_{k+2}(S\cap B) = 0$; so there is nothing to prove. Thus we may assume $S \cap (\G_{r,8}+y) \not= \emptyset$ and let $x \in S \cap (\G_{8,r}+y)$. Define $h:=x-y \in \G_{r,8}$ and notice that 
	\[ (S-y) \cap \G_{r,8} - h = (S-y-h) \cap (\G_{r,8}-h) \subseteq (S-x) \cap \G_{2r,16},\]
	where the inclusion holds since $h \in \G_{r,8}$. Therefore $(S-y) \cap \G_{r,8} \subseteq (S-x) \cap \G_{2r,16} + h$, allowing us to conclude that 
	\[ \g_{k+2}(S \cap B) \leq 16 \g_{k+2}((S-y) \cap \G_{r,8}) \leq 16 \g_{k+2}((S-x) \cap \G_{2r,16} + h),\]
	as desired. 
\end{proof}

\vspace{3mm}

\noindent We also need the following standard tail estimate on a $k$-dimensional Gaussian.

\begin{fact}\label{fact:Gtail}
$\g_{k}\big(\{x\in \R^k:  \|x\|^2_2 \geq k \}\big) \leq \exp(-k/8)$.
\end{fact}
\begin{proof}
For any $\eps\in(0,1)$ the \emph{standard} Gaussian measure of the set  $\{x\in \R^k:  \|x\|^2_2 \geq k/(1-\eps) \}$ is at most $\exp(-\eps^2k/4)$.  Recalling that 
$\g_k$ has standard deviation $(2\pi)^{-1/2}$ and taking $\eps = 1 - (2\pi)^{-1}$, gives the desired bound.
\end{proof}

\vspace{3mm}

\subsection{A Gaussian Brunn-Minkowski type theorem}\label{sec:BM}
We now lay out a useful tool which gives us some control of the Gaussian measure of the sum set $A+B$, relative to the Gaussian measures of $A$ and $B$.
Indeed, the following theorem due to Borell \cite{borell2008inequalities}, can be viewed as a Brunn-Minkowski-type theorem for Gaussian space. 

For this, let $\Phi(x)$ be the cumulative probability function 
$\Phi(x) := \Pr(Z \leq x )$, for the \emph{standard} one dimensional Gaussian $Z \sim \cN(0,1)$, while $\g_k$ is (still) the $k$-dimensional Gaussian with covariance matrix $(2\pi)^{-1} I_{k}$.

\begin{theorem}[Borell] \label{eq:BorellThm} Let $A,B \subseteq \R^k$ be Borel. Then 
\[ \g_k(A + B) \geq \Phi\bigg( \Phi^{-1}(\g_k(A)) + \Phi^{-1}(\g_k(B)) \bigg) . \]
\end{theorem}
\begin{proof}
In \cite{borell2008inequalities} Theorem~\ref{eq:BorellThm} is proved for the standard Gaussian measure rather than $\g_k$. However we can change the standard deviation of the measure by taking dilates of the sets $A$ and $B$.\end{proof}

\vspace{3mm}

\noindent We will use the following simple consequence of Theorem~\ref{eq:BorellThm}.

\begin{lemma}\label{lem:GaussBM}
	Let $A \subseteq \R^k$ be Borel. Then 
\[ \g_k(A - A) \geq \g_k(A)^4\,. \]
\end{lemma}
\begin{proof}
	By Theorem~\ref{eq:BorellThm}, we have \begin{equation}\label{eq:borell}
\g_k(A - A) \geq \Phi(2 \Phi^{-1}(\g_k(A))  ) = \Phi(2 x ), 
	\end{equation}  where we have set $x = \Phi^{-1}(\g_k(A))$. Note that 
\begin{equation}\label{eq:Phi-ineq}
	\Phi(2x) = \P(Z \leq 2x) = \P\left(Z_1 + Z_2 + Z_3 + Z_4 \leq 4x\right) \geq \P(Z \leq x)^4 = \Phi(x)^4
	\end{equation} where $Z_j$ are i.i.d.\ copies of $Z \sim \cN(0,1)$.  Combining~\eqref{eq:borell} and \eqref{eq:Phi-ineq} completes the proof.\end{proof}

\vspace{3mm}

\subsection{Proof of Lemma~\ref{lem:geo-compare}}\label{sec:pf-geo-compare}

With these pieces now in place, we can move on to prove Lemma~\ref{lem:geo-compare}, our key geometric lemma on the Fourier side.

\begin{proof}[Proof of Lemma \ref{lem:geo-compare}]
Write $r = \sqrt{k}$ for simplicity.  We prove the contrapositive and assume for every $x \in S$ we have  \begin{equation}\label{eq:geo-lemma-contrapositive}
(\Gamma_{2r,16} \setminus \Gamma_{2r,s} + x) \cap S = \emptyset.
\end{equation}  
We define
\[ B := \{ \t \in \R^{k+2} : \|\t_{[k]}\|_2 \leq r \}. \]
and proceed to bound $\g_{k+2}(S)$ from above by first bounding $\gamma_{k+2}(S \setminus B)$ and then bounding $\g_{k+2}(S \cap B )$.

\vspace{3mm}

\noindent \emph{Step 1: Upper bound for $\gamma_{k+2}(S \setminus B )$.} For $\t_{[k]} \in \R^{k}$, let $S(\theta_{[k]})$ be as defined at \eqref{eq:defS()}
\[ S(\t_{[k]}) = \left\lbrace (\t_{k+1},\t_{k+2}) \in \R^2 : (\t_{[k]},\t_{k+1},\t_{k+2}) \in S \right\rbrace .  \] 
We may write 
\begin{equation}\label{eq:step1up1} 
\g_{k+2}(S \setminus B ) = \int_{\|\t_{[k]}\|_2 \geq r } \g_2\left( S(\t_{[k]}) \right)\, d\g_{k}  \end{equation}
and thus
\begin{equation} \label{eq:step1up1.5} \g_{k+2}(S \setminus B ) \leq \bigg( \max_{\t_{[k]} \in \R^{k}}  \g_2\left( S(\t_{[k]}) \right) \bigg) \g_{k}\big( \{ \|\t_{[k]}\|_2 \geq r \} \big).  \end{equation}
Lemma \ref{lem:slice-upperBound} and \eqref{eq:geo-lemma-contrapositive} shows
\begin{equation}\label{eq:step1up2}  \max_{\t_{[k]} \in \R^{k}}  \g_2\left( S(\t_{[k]}) \right) \leq 8 s^2.\end{equation}
Fact~\ref{fact:Gtail} bounds
\begin{equation} \label{eq:largedev} \g_{k}\big( \{ \|\t_{[k]}\|_2 \geq r \} \big) \leq \exp(-k/8) \end{equation}
	and so from \eqref{eq:step1up1.5}, \eqref{eq:step1up2} and \eqref{eq:largedev} we learn 
	\begin{equation} \label{eq:step1Bound} \g_{k+2}(S\setminus B) \leq  8 s^2 e^{-k/8}. \end{equation}
	
	\vspace{4mm}

	\noindent \emph{Step 2: Upper bound for $\g_{k+2}(S \cap B)$.}  By Lemma~\ref{fact:covFact}, there exists $x\in S$ and $h\in\G_{r,8}$ such that
	\begin{equation}\label{eq:CylinderRedx} \g_{k+2}( S  \cap B) \leq 16 \g_{k+2}( (S -x) \cap \G_{2r,16}+ h ). \end{equation}
	Now since we are assuming the claim is false, and $x\in S$, we use \eqref{eq:geo-lemma-contrapositive} to deduce that
	\begin{equation} \label{eq:emptyMantle} (S -x) \cap \G_{2r,16} \subseteq  (S -x) \cap \G_{2r,s} \end{equation}
	and so letting $y= x- h$, we see  
	\begin{equation}\label{eq:Setshift}  (S -x) \cap \G_{2r,s} +h  = (S - x + h ) \cap (\G_{2r,s}+h)  = (S - y) \cap (\G_{2r,s} + h) .\end{equation}
	Thus by~\eqref{eq:CylinderRedx}, \eqref{eq:emptyMantle} and \eqref{eq:Setshift}, we have
	\begin{align}\label{eq:claimHypo}
	\g_{k+2}( S  \cap B) \leq 16 \g_{k+2}( (S -y) \cap ( \G_{2r,s}+ h) )\, .
	\end{align}
	Bound 
	\begin{equation}\label{eq:intExp2} \g_{k+2}( (S-y) \cap (\G_{2r,s}+h) ) \leq \int_{|a - h_{k+1}|,|b-h_{k+2}| \leq s} \g_{k}\big( F_y(S; a,b) \big) \, d\g_2  \end{equation}
	and apply Lemma \ref{lem:GaussBM} to obtain
	\begin{equation}\label{eq:GaussBMapp}
	\g_{k+2}( (S-y) \cap (\G_{2r,s}+h) ) \leq   4s^2 \max_{a,b,y} (\gamma_k (F_y(S;a,b) - F_y(S;a,b)))^{1/4}\,.\end{equation}
	Combining \eqref{eq:claimHypo} and \eqref{eq:GaussBMapp} gives \begin{equation}\label{eq:step2bound}
	\gamma_{k+2}(S \cap B) \leq 64 s^2 \max_{a,b,y} \big( \gamma_k\left(F_y(S;a,b)  - F_y(S;a,b) \right)    \big)^{1/4}
	\end{equation}
	\emph{Putting Step 1 and Step 2 together :} \eqref{eq:step2bound} together with \eqref{eq:step1Bound} implies
	$$\g_{k+2}(S) \leq 8 s^2 e^{-k/8} + 64 s^2 \max_{a,b,y} (\gamma_{k}(F_y(S;a,b) - F_y(S;a,b)))^{1/4},$$
	completing the proof of the contrapositive.
\end{proof}

\section{Inverse Littlewood-Offord III: Comparison to a lazier walk and Proof of Lemma~\ref{lem:CondWalkLCMfinal}}\label{sec:ILwO-CondWalksIII}

In Section~\ref{sec:ILwO-CondWalksII} we proved our key geometric ingredient, Lemma~\ref{lem:geo-compare}, to deal with the geometry of our level set (as seen in Section~\ref{sec:go-Fourier}). We now use this lemma to take the following big step towards Lemma~\ref{lem:CondWalkLCMfinal}.

\begin{lemma}\label{lem:CondWalkLCM}
For $d \in \N$ and $\alpha \in (0,1)$, let $0\leq k\leq 2^{-10}\alpha d$ and $t\geq \exp(-2^{-10}\alpha d)$. 
For $0 < c_0\leq 2^{-24}$, let $Y \in \R^d$ satisfy $\|Y \| \geq  2^{-10} c_0 /t$ and let $W$ be a $2d \times k$ matrix with $\|W\| \leq 2$. Also let  $\tau \sim \cQ(2d,1/4)$ and $\tau' \sim \cQ(2d,2^{-9})$ and $\beta \in [c_0/2^{10},\sqrt{c_0}]$, $\beta' \in (0,1/2) $.

If 
\begin{equation}\label{eq:LCM-hypo-lazy}  
\cL(W^T_Y\tau, \beta\sqrt{k+1}) 
\geq \left( R t\right)^2 \exp(4\beta^2 k)\left(\Pr(\|W^T \tau'\|_2\leq \beta'\sqrt{k}) + \exp(-\beta'^2 k) \right)^{1/4} \end{equation}
then $D_\alpha(Y)\leq 16$. Here we have set $R = 2^{31} /c_0^2$.
\end{lemma}

Of course, Lemma~\ref{lem:CondWalkLCM} looks quite a bit like Lemma~\ref{lem:CondWalkLCMfinal} save for quantity
\begin{equation}\label{eq:decoupled} \Pr(\|W^T \tau'\|_2\leq \beta'\sqrt{k}) + \exp(-\beta'^2 k), \end{equation}
on the right-hand side of \eqref{eq:LCM-hypo-lazy}. One should view this quantity as an approximation of the contribution that the ``soft'' constraints make. Indeed, if one reads this 
lemma in the contrapositive, it says that we can successfully ``decouple'' the ``soft'' constraints from the ``hard'' constraints, provided $Y$ is sufficiently ``unstructured'', meaning $D_{\alpha}(Y) > 16$. Of course, this story
is not quite an honest one; we have to use the lazier vector $\tau'$, rather than $\tau$, to get things to work out, and we also take a loss in the exponent of $1/4$.
The key here is that we obtain the correct power of $t$ in our bound, which is deeply important for our application. We also note that our use of ``decoupling'' 
should not be confused with the ``decoupling'' step in Costello, Tao and Vu \cite{costello-tao-vu}, which is used to deal with very unstructured vectors.

We prove this lemma in Section~\ref{sec:pf-CondWalkLCM} after laying out a few facts on level sets in Section~\ref{sec:levelsets}. We will then conclude this section in Section~\ref{sec:pf-CondWalkLCMfinal} with a proof of Lemma~\ref{lem:CondWalkLCMfinal}, by combining Lemma~\ref{lem:CondWalkLCM} with one further ingredient
to bound \eqref{eq:decoupled}.

\subsection{Working with level sets}\label{sec:levelsets}

To prepare for the proof of Lemma~\ref{lem:CondWalkLCM}, we record two basic facts about level sets. First off, we note a sort of converse to the
Esseen-type inequality that we saw in Section~\ref{sec:ILwO-CondWalks}, Lemma~\ref{lem:esseen}. Again, we will postpone the straightforward proof of this lemma to Appendix~\ref{sec:FourierPrep}. 
Recall that we defined, for a $2d \times \ell$ matrix $W$, the $W$-\emph{level set}, for $t \geq 0$, to be
\[ S_W(t) := \left\lbrace \theta \in \R^{\ell} : \|W\theta\|_{\T} \leq \sqrt{t} \right\rbrace. \] 

\begin{lemma}\label{lem:revEsseen} Let $\beta >0, \mu \in (0,1/4]$, let $W$ be a $2d \times \ell$ matrix, and let $\tau\sim\cQ(2d, \mu)$.
Then for all $t\geq 0 $, we have 
\[ \g_{\ell}(S_W(t))e^{-32\mu t} \leq \Pr_{\tau}\big( \|W^T\cdot \tau\|_2\leq \beta\sqrt{\ell} \big)+ \exp\left(-\beta^2\ell\right). \]
\end{lemma}

\vspace{3mm}

We need also need the following basic fact about level sets. Recall that, for a set $S \subset \R^{k+2}$ and $y \in \R^{k+2}$, we defined
the ``translated horizontal fiber'',
\[ F_y(S;a,b) := \{\theta_{[k]} = (\theta_1,\ldots,\theta_k) \in \R^k : (\theta_1,\ldots,\theta_k,a, b) \in S - y \}\,. \]

\begin{fact}\label{fact:LevelSetTriangleInq} For any $2d \times (k+2)$ matrix $W$. If $m >0$ we have  
\[ S_W(m) - S_W(m) \subseteq S_W(4m). \]
Similarly, for any $y \in \R^{k+2}$ and $a,b \in \R$ we have 
\begin{equation}\label{eq:Fiberdiff} F_y(S_W(m);a,b) - F_y(S_W(m) ;a,b) \subseteq F_0(S_{W}(4m);0,0). \end{equation}
\end{fact}
\begin{proof}
Notice that if $x,y\in S_W(m)$ then by definition $\| W x\|_{\T},\|W y\|_{\T}\leq \sqrt{m}\, $. Thus, by the triangle inequality, 
\[ \| W (x-y) \|_{\T} \leq \| W x \|_{\T} + \|W y\|_{\T}\leq 2\sqrt{m}\,.\]
For \eqref{eq:Fiberdiff}, let $\t_{[k]},\t'_{[k]} \in F_y(S;a,b)$. We have that
\[(\t_1,\ldots,\t_{k},a,b), (\t'_1,\ldots,\t'_{k},a,b) \in S_{W}(m)-y \]
and so $\t'' := (\t_1-\t_1',\ldots,\t_{k}-\t'_{k},0,0) \in S_{W}(4m)$. Thus $\t_{[k]} -\t'_{[k]} \in F_0(S_{W}(4m);0,0)$, implying \eqref{eq:Fiberdiff}.
\end{proof}

\subsection{Proof of \ref{lem:CondWalkLCM}}\label{sec:pf-CondWalkLCM}

We may now turn to prove Lemma~\ref{lem:CondWalkLCM}, our big step towards Lemma~\ref{lem:CondWalkLCMfinal}.

\begin{proof}[Proof of Lemma~\ref{lem:CondWalkLCM}]
Apply  Lemma~\ref{lem:esseen} to find $m>0$ such that the level set 
\[ S := S_{W_Y}(m) = \{\theta\in \R^{k+2}: \|W_{Y} \theta \|_{\T }\leq \sqrt{m} \},\]  
satisfies 
\begin{equation} \label{eq:EssenApp} 
e^{-\frac{1}{8}m + 2\beta^2k}\g_{k+2}(S) \geq \cL(W^T_Y\tau, \beta\sqrt{k+1}) .  
\end{equation}
Thus \eqref{eq:EssenApp} together with our hypothesis \eqref{eq:LCM-hypo-lazy} gives a lower bound 
\begin{equation} \label{eq:gS-LB} \g_{k+2}(S) \geq  \frac{1}{4} e^{\frac{1}{8}m - 2 \beta^2 k} \left( R t\right)^2  T^{1/4}, \end{equation}
where we have set 
\[ T := \Pr(\|W^T \tau'\|_2\leq \beta'\sqrt{k}) + \exp(-\beta'^2 k) . \]
We now make the following important designations,
\begin{equation} \label{eq:defs} r_0 :=  \sqrt{k} \qquad \text{ and } \qquad s_0 := 2^{16} c_0^{-1}(\sqrt{m}+\sqrt{k})t  .\end{equation}
Recall from~\eqref{eq:defCy} that for $r,s >0$ we defined the \emph{cylinder}
\[ \G_{r,s}:=\left\{\theta \in \R^{k+2} : ~\left\|\theta_{[k]} \right\|_2\leq r\, \text{ and }  |\theta_{k+1}|\leq s, ~|\theta_{k+2}|\leq s, \right\}.\]

\begin{claim}\label{claim:condWalkLCM1}
There exists $x \in S \subseteq \R^{k+2}$ so that\footnote{Note that this claim shows, in particular, that $s_0 < 16$.} 
\begin{equation}\label{eq:non-empty}  \big( \G_{2r_0,16} \setminus \G_{2r_0,s_0} + x \big) \cap S \not= \emptyset.\end{equation}
\end{claim}
\begin{proof}[Proof of Claim \ref{claim:condWalkLCM1}]
We look to apply Lemma \ref{lem:geo-compare} with $s = s_0$. For this, we bound
\[ M := \max_{a,b,y}\bigg\lbrace \g_k\Big(F_y(S;a,b) - F_y(S;a,b)\Big)\bigg\rbrace,  \]
above by $e^{m/4}T$, thus giving a lower bound on $\g_{k+2}(S)$ and allowing us to apply Lemma~\ref{lem:geo-compare}. 
Use Fact~\ref{fact:LevelSetTriangleInq} to see that for any $y,a,b$, we have 
\begin{equation} \label{eq:claim-diff2}
 F_y(S;a,b) - F_y(S;a,b) \subseteq F_0(S_{W_Y}(4m);0,0)\,.
\end{equation} Now carefully observe that 
\[  F_0(S_{W_Y}(4m);0,0) = \left\lbrace \theta_{[k]} \in \R^{k} : \|W\theta_{[k]} \|_{\T} \leq \sqrt{4m}   \right\rbrace =S_W(4m), \]
which is a  level-set corresponding to the (``decoupled'') event  $\PP_{\tau'}( \| W^T \tau'\|_2 \leq \beta' \sqrt{k} )$, where $\tau' \sim \cQ(2d,2^{-9})$
and $\beta' \in (0,1/2)$ is as in the hypothesis. 
Thus we may apply Lemma~\ref{lem:revEsseen} along with  \eqref{eq:claim-diff2} to obtain 
\begin{equation} \label{eq:S(m)upBound2} 
M  \leq  \g_{k}(F_0(S_{W_Y}(4m),0,0)) = \g_{k}(S_W(4m))\leq e^{m/4}T \,. \end{equation} 
So combining \eqref{eq:S(m)upBound2} with \eqref{eq:gS-LB}, gives
\begin{equation}\label{eq:claimcondwalklcm1} \g_{k+2}(S) \geq (1/4)e^{m/16 + 2\beta^2k} (Rt)^2 M^{1/4} \geq 8s_0^2e^{-k/8} + 64s_0^2 M^{1/4}, \end{equation}
allowing us to apply Lemma~\ref{lem:geo-compare} and complete the proof of the claim. 
The last inequality at \eqref{eq:claimcondwalklcm1} follows from a simple check: each term on the right-hand side of \eqref{eq:claimcondwalklcm1} is at most half of the left-hand side.  First note that \begin{equation}\label{eq:s0-bound}
s_0^2 = 2^{32}c_0^{-2}(\sqrt{m} + \sqrt{k})^2t^2 < 2^{33}(k+m)(t/c_0)^2\end{equation}
and so 
$$8 s_0^2 e^{-k/8} \leq \frac{1}{8} e^{m/8-2\beta^2 k}(Rt)^2 e^{-\beta'^2 k/4}$$ follows from $\beta' \leq 1/2$ and the definition of $R$.  On the other hand, use \eqref{eq:s0-bound} to bound $$64 s_0^2 e^{m/16} \leq 2^{39} t^2 c_0^{-2}(2^{20} c_0^{-2} \beta^2 k + 8 (m/8) ) \leq \frac{1}{8}(Rt)^2 e^{m/8 + \beta^2 k}$$
thus showing the second inequality at \eqref{eq:claimcondwalklcm1} and finishing the proof of the claim.\end{proof}

\vspace{4mm}

\noindent We now observe the simple consequence of Claim~\ref{claim:condWalkLCM1}.

\begin{claim} \label{claim:non-empty} We have that $S_{W_Y}(4m) \cap (\G_{2r_0,16} \setminus \G_{2r_0,s_0}) \not= \emptyset $.
\end{claim}
\begin{proof}[Proof of Claim \ref{claim:non-empty}]
 By Claim~\ref{claim:condWalkLCM1}, there exists $x,y \in S = S_{W_Y}(m)$ so that $y \in ( \G_{2r_0,16} \setminus \G_{2r_0,s_0} + x \big) \cap S $. 
Set $\phi := y-x$ and observe that $\phi \in S_{W_Y}(4m) \cap (\G_{2r_0,16} \setminus \G_{2r_0,s_0})$, by Fact~\ref{fact:LevelSetTriangleInq}.
\end{proof}

\vspace{3mm}

\noindent We now conclude the proof of Lemma~\ref{lem:CondWalkLCM} with the following claim.

\begin{claim}\label{claim:CondRW3}
If $\psi \in S_{W_Y}(4m) \cap (\G_{2r_0,16} \setminus \G_{2r_0,s_0})$ then there exists $i \in \{k+1,k+2\}$ so that
 \[  \|\psi_i Y\|_{\T} < \min\{\psi_i\| Y\|_{2}/2, \sqrt{\alpha d}\}\,. \]
\end{claim}
\begin{proof}[Proof of Claim \ref{claim:CondRW3}]
	Note that since $\psi \in S_{W_Y}(4m)$ there is a $p \in \Z^{2d}$ so that $W_Y \psi \in B_{2d}(p,2\sqrt{m})$. So if we express 
\[ W_Y\psi = W\psi_{[k]} + \psi_{k+1}  \begin{bmatrix} Y \\  {\bf{0}}_d \end{bmatrix}  + \psi_{k+2} \begin{bmatrix} {\bf{0}}_d \\ Y \end{bmatrix}\,,\]
we have that 
\begin{equation} \label{eq:almostLCM}  
\psi_{k+1} \begin{bmatrix} Y \\ {\bf{0}}_d \end{bmatrix}  + \psi_{k+2} \begin{bmatrix} {\bf{0}}_d \\ Y \end{bmatrix} \in   
B_{2d}(p,2\sqrt{m}) -  W \psi_{[k]} \subseteq B_{2d}(p, 2\sqrt{m} + 4\sqrt{k}),\end{equation}
where the last inclusion holds because $\psi \in \G_{2r_0,16}$ and so $\|\psi_{[k]}\|_2 \leq 2 r_0 \leq 2\sqrt{k}$ and $\|W\| \leq 2$.

Since $\psi \not\in \G_{2r_0,s_0}$ we have that at least one of $|\psi_{k+1}|,|\psi_{k+2}|$ are $> s_0$. So, assume without loss that $|\psi_{k+1}|>s_0$
and that $\psi_{k+1} > 0$ (otherwise replace $\psi$ with $-\psi$). Now project \eqref{eq:almostLCM} onto the first $d$ coordinates, to obtain
\begin{equation}\label{eq:phiY-close} \psi_{k+1} Y \in B_{d}(  p_{[d]} , 2\sqrt{m} + 4\sqrt{k}) . \end{equation}

We now observe that $\|\psi_{k+1} Y\|_{\T} < \frac{\psi_{k+1}\| Y\|_2}{2}$. Indeed,
\begin{equation} \label{eq:non-triv} \frac{\psi_{k+1}\| Y\|_2}{2} > \frac{s_0 \|Y\|_2}{2} \geq \bigg(\frac{2^{15}(\sqrt{m} + \sqrt{k})t}{c_0}\bigg)\bigg(2^{-10}\frac{c_0}{t}\bigg) 
> (2\sqrt{m} + 4\sqrt{k}), \end{equation}
where we have used the definition of $s_0$ and that $\|Y\|_2 > 2^{-10}c_0/t$.

Finally, we note that $m \leq 2^{-4}\alpha d$. To see this, we use \eqref{eq:gS-LB}, $\g_{k+2}(S) \leq 1$ and our lower bound $t \geq \exp(-2^{-9}\alpha d)$ to see
\[ e^{-m/8} \geq \g_{k+2}(S) e^{-m/8} \geq (Rt)^2e^{-2\beta'^2 k} \geq \exp(-2^{-7}\alpha d), \]
where we have used $k \leq 2^{-9}\alpha d$ and $\beta'<1$ for the last inequality, thus $m \leq 2^{-4}\alpha d$. Therefore from \eqref{eq:phiY-close} and \eqref{eq:non-triv} we have 
\[ \|\psi_{k+1} Y \|_{\T} \leq 2\sqrt{m} + 4\sqrt{k}\leq \sqrt{\alpha d }, \]
as desired. This completes the proof of the Claim~\ref{claim:CondRW3}. \end{proof}

\vspace{3mm}

Let $\psi$ and $i \in \{k+1,k+2\}$ be as guaranteed by Claim \ref{claim:CondRW3}. Then $\psi_i \leq 16 $, and 
\[ \|\psi_i  Y\|_{\T} < \min\{\|\psi_i Y\|_{2}/2, \sqrt{\alpha d}\},\] and so $D_\alpha(Y)\leq16$ thus completing the proof of Lemma~\ref{lem:CondWalkLCM}.
 \end{proof}

\subsection{Proof of Lemma \ref{lem:CondWalkLCMfinal}}\label{sec:pf-CondWalkLCMfinal}

Before turning to prove Lemma~\ref{lem:CondWalkLCMfinal}, we require one further result which tells us that $\| W\s \|_2$ is anti-concentrated when $\s$ is a random vector and $W$ is a fixed matrix. While there are several interesting results of this type in the literature \cite{RV-HW,halasz,ferberHadamard} (and we will encounter another in Subsection~\ref{subsec:AX-anticoncentration}), we state here a variant of the Hanson-Wright inequality with an explicit constant. We derive this from Talagrand's classical inequality in Appendix~\ref{sec:details}.
 
\begin{lemma}\label{lem:HansonWright}
For $d \in \N$, $\nu \in (0,1)$, let $\delta \in (0,\sqrt{\nu}/16)$, let
$\sigma \sim \cQ(2d, \nu)$, and let $W$ be a $2d \times k$ matrix satisfying $\|W\|_{\HS}\geq \sqrt{k}/2$ and $\|W\|\leq 2$.  Then
\begin{align}\label{eq:HW}
\Pr(\|W^T\sigma\|_2\leq \delta \sqrt{k})\leq 4\exp(-2^{-12}\nu k)
\end{align}
\end{lemma}
 
\vspace{3mm} 

\noindent We now turn to prove Lemma~\ref{lem:CondWalkLCMfinal}.

\vspace{2mm}

\begin{proof}[Proof of Lemma \ref{lem:CondWalkLCMfinal}]
Setting $\beta':=4\sqrt{c_0}$, we look to apply Lemma~\ref{lem:CondWalkLCM}. For this, note that the hypotheses in Lemma~\ref{lem:CondWalkLCMfinal}  imply the hypotheses in Lemma~\ref{lem:CondWalkLCM} with respect to $c_0, d,\alpha,k, Y, W$ and $\tau$ (and we have the extra condition on $\|W\|_{\HS}$). 
So if we additionally assume $D_{\alpha}(Y) > 16$, we may apply Lemma~\ref{lem:CondWalkLCM} (in the contrapositive) to obtain 
\begin{equation}\label{eq:lemCOndWalkLCMapp} \cL\left( W_Y^T \tau,  \beta \sqrt{k+1} \right)  \leq(2^{31} c_0^{-2} t/2)^{2} e^{4 \beta^2 k} \left(\P(\| W^T \tau' \|_2 \leq \beta' \sqrt{k}) + e^{- \beta'^2 k} \right)^{1/4}. \end{equation}

To deal with the right-hand side, we apply Lemma~\ref{lem:HansonWright} to take care of the quantity involving $\tau' \in \{-1, 0, 1 \}^{2d}$, our $\nu = 2^{-9}$ lazy random vector.
Note that $4\sqrt{c_0}\leq 2^{-10} \leq \sqrt{\nu}/16 $, and that our given $W$ satisfies $\|W\|_{\HS}\geq \sqrt{k}/2$ and $\|W\|\leq 2$. Thus we may apply Lemma~\ref{lem:HansonWright}, with $\delta=\beta'$ and $\sigma=\tau'$, to see
\begin{equation} \label{eq:HansonWrightapp} \Pr(\|W^T\tau'\|_2\leq \beta' \sqrt{k})\leq 4\exp(-2^{-12}\nu k).  \end{equation}
Plugging this into the right-hand side of \eqref{eq:lemCOndWalkLCMapp} yields
 \begin{align*}
\exp(4\beta^2 k)\left(\Pr(\|W^T\tau'\|_2 
\leq \beta' \sqrt{k})+\exp(-\beta'^2 k)\right)^{1/4}
&\leq 2\exp(4 c_0 k-2^{-21}k)+2\exp(2 c_0 k-4c_0 k)\\ 
&\leq 4\exp(-c_0 k) .
\end{align*}
Putting this together with \eqref{eq:lemCOndWalkLCMapp}, yields 
\[ \cL\left( W_Y^T \tau,  \beta \sqrt{k+1} \right)   \leq  (Rt)^2\exp(-c_0 k),\]
as desired. \end{proof}

\section{Inverse Littlewood-Offord for conditioned random matrices}\label{sec:matrixWalks}

In this section we lift the main result of the previous sections (Lemma~\ref{lem:CondWalkLCMfinal}) to study the concentration of the vector $H_1X$, where $H_1$
is a random $(n-d)\times d$ matrix, conditioned on having $k$ singular values which are much smaller than ``typical'' and $X$ is a fixed vector for which $|X_i| \approx N$
for each $i$.

Here $N$ should be thought of as $\approx 1/\eps$, in the context of the proof (see Section~\ref{sec:sketch}) and $H_1$ comes from its appearance in our matrix $M$,
\[ M =  
\begin{bmatrix}
{\bf 0 }_{[d] \times [d]} & H_1^T \\
H_1 & { \bf 0}_{[d+1,n] \times [d+1,n]}  
\end{bmatrix}. \] 

The main result of this section is the following theorem\footnote{For convenience, we define $\s_j(H)=0$ for $j>\rk(H)$.}.
\begin{theorem}\label{lem:rankH} For $n \in \N$ and $0 < c_0 \leq 2^{-24}$, let $d \leq c_0^2 n$, and for $\alpha \in (0,1)$, let $0\leq k\leq 2^{-10}\alpha d$ and $N\leq \exp(2^{-10}\alpha d)$. 
Let $X \in \R^d$ satisfy $\|X\|_2 \geq c_02^{-10} n^{1/2} N$, and let $H$ be a random $(n-d)\times 2d$ matrix with i.i.d.\ $(1/4)$-lazy entries in $\{-1,0,1\}$.

If $D_\alpha(r_n \cdot X)> 16$ then
\begin{equation} \label{eq:RankofH}
\Pr_H\left(\s_{2d-k+1}(H)\leq c_02^{-4}\sqrt{n} \text{ and } \|H_1X\|_2,\|H_2 X\|_2\leq n\right)\leq e^{-c_0nk/4}\left(\frac{R}{N}\right)^{2n-2d}\, ,\end{equation}
where we have set $H_1 := H_{[n-d]\times [d]}$, $H_2 := H_{[n-d] \times [d+1,2d]}$, $r_n := \frac{c_0}{16\sqrt{n}}$ and $R := 2^{39}c_0^{-3}$. 
\end{theorem}

To understand the numerology in Theorem~\ref{lem:rankH}, notice that if we only consider the ``soft'' constraints on the singular values (without the constraints imposed by $X$) we would expect something like
\begin{equation}\label{eq:leastSingValueHuristic}
 \Pr_H\left(\s_{2d-k+1}(H)\leq c_02^{-4}\sqrt{n} \right) \approx c^{nk},
\end{equation}
for some absolute $c \in (0,1)$, which depends on the value of $c_0$. Here we are using, crucially, that $H$ is a \emph{rectangular} matrix with aspect ratio bounded away from $1$. Indeed,
if $H$ were a square matrix then $\s_{\min}(H) \approx n^{-1/2}$, with high probability\footnote{While we can refer the reader to \cite{RV-rectangle, RV-ICM} for more on the singular values of rectangular random matrices, we were not able to find any result such as \eqref{eq:leastSingValueHuristic} in the literature. However, it is not so hard to deduce \eqref{eq:leastSingValueHuristic} from the Hanson-Wright inequality \cite{RV-HW} along with a ``random rounding'' step similar to that in Appendix~\ref{sec:randomrounding}.}.

On the other hand, the inverse Littlewood-Offord theorem of Rudelson and Vershynin \cite{RV} 
(with a bit of extra work) tells us that if $X$ is such that  $|X_i| \approx N$ for all $i \in [d]$, and 
\[ \PP( \|H_1X\|_2,\|H_2 X\|_2\leq n ) \geq \left(\frac{R}{N}\right)^{2n-2d}, \]
then $D_{\alpha}(n^{-1/2}X) =O(1)$. Thus Theorem~\ref{lem:rankH} is telling us that we maintain an inverse Littlewood-Offord type theorem even in the presence of many additional constraints imposed by the condition on the least singular values.

\subsection{A tensorization step}\label{sec:tensor}
We need the following basic fact.

\begin{fact}\label{fact:regularityofL} If $r\geq t >0$ and $X$ is a random variable taking values in $\R^{k+2}$, then
\[ \cL(X,t)\leq \cL(X,r)  \leq (1+2r/t)^{k+2} \cL(X, t).   \]
\end{fact}
\begin{proof}
The lower bound is trivial. The upper bound follows from the fact that a ball of radius $r$ in $\R^{k+2}$ can be covered by $(1+2r/t)^{k+2}$ balls of radius $t$.
\end{proof}

\vspace{3mm}
We now prove a ``tensorization'' lemma which shows that anti-concentration of a single row in a random matrix $H$ (with iid rows) implies the anti-concentration of matrix products involving $H$.

\begin{lemma}\label{lem:tensor} For $d < n$ and $k \geq 0$, let $W$ be a $2d \times (k+2)$ matrix and let $H$ be a $(n-d)\times 2d$ random matrix with i.i.d.\ rows. Let $\tau \in \R^{2d}$ be a random vector with the same distribution as the rows of $H$.
If $\beta \in (0,1/8)$ then
\[  \PP_H\big( \|HW\|_{\HS} \leq \beta^2 \sqrt{(k+1)(n-d)} \big)  \leq \left(2^{5}e^{2\beta^2 k}\cL\big( W^T \tau, \beta \sqrt{k+1} \big)\right)^{n-d}. \] 
\end{lemma}
\begin{proof}
Apply Markov's inequality to see that
\begin{equation} \label{eq:tensor-markov}
\PP\big( \|HW\|_{\HS} \leq \beta^2 \sqrt{(k+1)(n-d)} \big) \leq \exp\left(2\beta^2 (k+1)(n-d)\right) \EE_{H} e^{-2\|HW\|_{\HS}^2/\beta^2 }.  \end{equation}
Letting $\tau_1, \ldots, \tau_{n-d}$ denote the i.i.d.\ rows of $H$, we may rewrite
\begin{equation} \label{eq:tensor-independent}
\EE_{H}\, e^{-2\|HW\|_{\HS}^2/\beta^2 } = \prod_{i=1}^{n-d} \EE_{\tau_i}\, e^{-2\|W^T\tau_i\|^2/\beta^2} = \left( \EE_{\tau}\, e^{-2\|W^T\tau\|^2/\beta^2} \right)^{n-d}\, .
\end{equation}
Observe now that 
\[ \EE_{\tau}\, e^{-2\|W^T\tau\|^2/\beta^2} = \int_{0}^\infty \PP\left( e^{-2\|W^T\tau\|^2/\beta^2}>u \right)\, du 
=\int_{0}^\infty 4ue^{-2u^2}\PP\big( \| W^T\tau \|_2/\beta \leq u \big)\, du. \]
Splitting the integral on the right-hand side gives 
\[ \EE_{\tau}\, e^{-2\|W^T\tau\|^2/\beta^2} = \int_0^{\sqrt{k+1}} 4ue^{-2u^2} \PP\big( \| W^T\tau \|_2  \leq \beta u \big) + \int_{\sqrt{k+1}}^{\infty} 4ue^{-2u^2} \PP\big( \| W^T\tau \|_2  \leq \beta u \big).\]
We then appeal to Fact \ref{fact:regularityofL} to write 
\[ \EE_{\tau}\, e^{-2\|W^T\tau\|^2/\beta^2} \leq \cL\big( W^T \tau, \beta \sqrt{k+1} \big)\left( \int_0^{\sqrt{k+1}} 4ue^{-2u^2} \, du + \int_{\sqrt{k+1}}^{\infty} \left(1+\frac{ 2 u }{\sqrt{k+1} }\right)^{k+2}4u e^{-2u^2} \, du  \right).\]
Here the first integral is $\leq 1$, while the second integral is $\leq 8$ (see Fact~\ref{fact:infamous-int} in Appendix~\ref{sec:details}) and thus
\begin{equation}\label{eq:tensor-exp2} \EE_{\tau}\, e^{-2\|W^T\tau\|^2/\beta^2} \leq 9 \cL\big( W^T \tau, \beta \sqrt{k+1} \big)\, . 
\end{equation}
Combining lines \eqref{eq:tensor-exp2} with \eqref{eq:tensor-independent} and \eqref{eq:tensor-markov} gives 
$$\P_H(\|HW \|_{\HS} \leq \beta^2 \sqrt{(k+1)(n-d)}) \leq \bigg(9 \exp(2\beta^2 (k+1)) \cL\big( W^T \tau, \beta \sqrt{k+1} \big)\bigg)^{n-d}\, ,$$
and the result follows.\end{proof}

\subsection{Approximating matrices $W$ with nets}
Note that in Theorem~\ref{lem:rankH}, the least singular values of the matrix $H$ could, a priori, correspond to any of a huge number of possible directions.
To limit the number of directions we need to consider, we build nets for $k$-tuples of these directions. Luckily, the construction of these nets is rendered relatively simple (unlike the nets $\cN_{\eps}$) by appealing to a randomized-rounding technique pioneered in the context of random matrices by Livshyts~\cite{livshyts2018smallest}
(also see Section 3 of \cite{inhomogenous}). 

With this in mind, let $\cU_{2d,k}$ be the set of all $2d \times k$ matrices with orthonormal columns. The following theorem provides a net for $\cU_{2d,k}$,
when viewed as a subset of $\R^{[2d] \times [k]}$. The proof is deferred to Appendix~\ref{sec:randomrounding}. 

\begin{lemma}\label{lem:basis-net}
For $k \leq d$ and $\delta \in (0,1/2)$, there exists $\cN = \cN_{2d,k} \subset \R^{[2d]\times [k]}$ with $|\cN| \leq (2^6/\delta)^{2dk}$ 
so that for any $U\in \cU_{2d,k}$, any $r \in \N$ and $r \times 2d$ matrix $A$ there exists $W\in \cN$ so that 
\begin{enumerate}
	\item \label{it:rr-mat-1} $\|A(W-U)\|_{\HS}\leq \delta(k/2d)^{1/2}  \|A\|_{\HS} $, 
	\item \label{it:rr-mat-2} $\|W-U\|_{\HS}\leq \delta \sqrt{k}$ and
	\item \label{it:rr-mat-3} $\|W-U\|\leq 8\delta .$
	\end{enumerate}\end{lemma}

\vspace{2mm}

\noindent Recall, for a $2d \times k$ matrix $W$ and $Y \in \R^d$, we defined (at \eqref{eq:WYdef}) the augmented matrix 
\[  W_Y = \begin{bmatrix}\,  W  ,\, \begin{bmatrix} { \bf 0}_d \\ Y \end{bmatrix} ,\, \begin{bmatrix} Y \\ { \bf 0}_d \end{bmatrix} \end{bmatrix} . \]

\subsection{Proof of Theorem~\ref{lem:rankH}}

We recall a standard fact from linear algebra, reworded to suit our context.
\begin{fact}\label{fact:singvalues}
For $3d < n$, let $H$ be a $(n-d) \times 2d$ matrix. If $\s_{2d-k+1}(H) \leq x$
then there exist $k$ orthogonal unit vectors $w_1,\ldots,w_k \in \R^{2d}$ so that $\|Hw_i\|_2 \leq x$. In particular, there exists $W \in \cU_{2d,k}$ so that 
$\|HW\|_{\HS} \leq x\sqrt{k}$.\end{fact}

\vspace{3mm}

We also note that if $H$ is a $(n-d)\times 2d$ matrix with entries in $\{-1, 0, 1\}$ then we immediately have $\|H\|_{\HS}\leq \sqrt{2d(n-d)}$.

\vspace{4mm}

\begin{proof}[Proof of Theorem~\ref{lem:rankH}]
Write $Y := \frac{c_0}{16\sqrt{n}}\cdot X$. We use Fact~\ref{fact:singvalues}
to upper bound the left-hand-side of \eqref{eq:RankofH} as 
\begin{align*}\PP(\s_{2d-k+1}(H) &\leq c_02^{-4}\sqrt{n} \text{ and } \|H_1 X\|_2,\|H_2 X\|_2\leq n)\\
 &\leq \PP(\exists U\in \cU_{2d,k}:\|H U_Y\|_{\HS} \leq 3c_0\sqrt{n (k+1)}/16). \end{align*}
Set $\delta := c_0/16$, and let $\cW$ be the $\delta$-net for $\cU_{2d,k}$, given by Lemma~\ref{lem:basis-net}.

We fix a matrix $H$ for a moment. If there exists a matrix $U \in \cU_{2d,k}$ so that $\|HU_Y\|_{\HS} \leq 3c_0\sqrt{n (k+1)}/16$, apply Lemma~\ref{lem:basis-net} to find $W \in \cW$ so that 
\[ \|HW_Y\|_{\HS} \leq \|H(W_Y-U_Y)\|_{\HS} + \|HU_Y\|_{\HS} \leq \delta(k/2d)^{1/2} \|H\|_{\HS}+ 3c_0\sqrt{n(k+1)}/16 \]
which is at most $c_0\sqrt{n(k+1)}/4$, since $\|H\|_{\HS} \leq \sqrt{2nd}$. Thus
\[ \PP\left( \exists U\in \cU_{2d,k}:\|H U_Y\|_{\HS} \leq \frac{c_0}{16}\sqrt{n (k+1)}\right)
\leq \PP\left( \exists W \in \cW :\|H W_Y\|_{\HS} \leq \frac{c_0}{4}\sqrt{n (k+1)}\right) .   \]
So by the union bound, we have 
\[ \PP\left( \exists W \in \cW :~\|H W_Y\|_{\HS} \leq (c_0/4)\sqrt{n (k+1)}\right) \leq \sum_{W\in \cW}\Pr\left( \|H W_Y\|_{\HS}\leq (c_0/4)\sqrt{n (k+1)} \right). \]
Now 
\[ |\cW| \leq (2^6/\delta)^{2dk} \leq  \exp( 32 dk\log c_0^{-1} )  \leq  \exp( c_0 k(n-d)/4), \] 
where the last inequality holds since $d\leq c_0^2 n$, and so
\begin{equation}\label{eq:postTensor}  \sum_{W\in \cW}\Pr\left(\|H W_Y\|_{\HS}\leq \frac{c_0}{4}\sqrt{n (k+1)}\right) 
\leq e^{c_0 k(n-d)/4}\max_{W\in \cW}\Pr\left(\|H W_Y\|_{\HS}\leq \frac{c_0}{4}\sqrt{n (k+1)}\right)\,.\end{equation}
Let $W \in \cW$ be such that the maximum in \eqref{eq:postTensor} is attained, apply Lemma~\ref{lem:tensor} with $\beta :=\sqrt{c_0}/2$ to obtain
\begin{equation}\label{eq:tensorapp}
\Pr(\|H W_Y\|_{\HS}\leq (c_0/4)\sqrt{n (k+1)})\leq \left(2^{5}e^{c_0 k/2}\cL\big( W_Y^T \tau, c_0^{1/2} \sqrt{k+1} \big)\right)^{n-d}.
\end{equation}

We now look to apply Lemma~\ref{lem:CondWalkLCMfinal}. We define $t := 16/(c_0 N) \geq \exp(- 2^{-9}\alpha d)$ and\\ 
$R_0 := 2^{-7}c_0 R = 2^{-7}c_0(2^{39}c_0^{-3}) =  2^{32}c_0^{-2}$ so that we have \[ \|Y\|_2=c_0\|X\|_2/(16n^{1/2}) \geq 2^{-14}c_0^2 N = 2^{-10}c_0/t.\] By the construction of $\cW$ in Lemma~\ref{lem:basis-net} we have $\|W\|\leq 2$ and $\|W\|_{\HS} \geq \sqrt{k}/2$. We also have $k \leq 2^{-10}\alpha d$ and $ D_\alpha(\frac{c_0}{16\sqrt{n}} X) = D(Y) > 16$, therefore we may apply Lemma~\ref{lem:CondWalkLCMfinal} to see that
\[  \cL\big( W_Y^T \tau, c_0^{1/2} \sqrt{k+1} \big)\leq (R_0t)^2\exp(-c_0k)\leq \left(\frac{R}{8N}\right)^2\exp(-c_0k). \] 
Substituting this bound in \eqref{eq:tensorapp} we get 
\[ \max_{W \in \cW }\, \Pr_H(\|H W_Y\|_2\leq (c_0/4) \sqrt{n (k+1)} )\leq \left(\frac{R}{N}\right)^{2n-2d}\exp(-c_0 k(n-d)/2)\] and finally combining it with the previous bounds gives \[\PP(\s_{2d-k+1}(H)\leq c_0\sqrt{n}/16 \text{ and } \|H_1 X\|_2,\|H_2 X\|_2\leq n)\leq \left(\frac{R}{ N}\right)^{2n-2d}\exp(-c_0 k(n-d)/4).\]
This completes the proof of Theorem~\ref{lem:rankH}.
\end{proof}

\section{Nets for structured vectors: Size of the Net }\label{sec:Size-Net}
In this section we take a important step towards Theorem~\ref{thm:main} by bounding the size of our net
\[ \cN_{\eps}  := \left\lbrace  v \in \L_{\eps} : (L\eps)^n \leq    \PP(\|Mv\|_2\leq 4\eps\sqrt{n}) \text{ and }  \cL_{A,op}(v,\eps\sqrt{n}) \leq (2^8L\eps)^n \right\rbrace\, ,\] 
where we recall that 
\[ \L_{\eps} := B_n(0,2) \cap \big( 4\eps  n^{-1/2} \cdot \Z^n\big) \cap \cI'([d]). \]
In particular, our main goal of this section will be to prove the following theorem on the size of $\cN_{\eps}$.

\begin{theorem}\label{thm:netThm} For $L\geq 2$ and $0 < c_0 \leq 2^{-24}$, let $n \geq L^{64/c_0^2}$, let $d \in [c_0^2n/4, c_0^2 n] $ and let 
$\eps >0$ be such that $\log \eps^{-1} \leq  n L^{-32/c_0^2} $. Then
\[|\cN_{\eps}|\leq \left(\frac{C}{c_0^6L^2\eps}\right)^{n}, \]
where $C>0$ is an absolute constant.
\end{theorem}

As the geometry of the set $\L_{\eps}$ is a bit complicated, we follow an idea of Tikhomirov \cite{Tikhomirov}, by working with the intersection of 
$\cN_{\eps}$ with a selection of ``boxes'' which cover (an appropriately re-scaled) $\L_{\eps}$. 
\begin{definition}
Define a $(N,\kappa,d)$-\emph{box} to be a set of the form $\mathcal{B}=B_1 \times \ldots \times B_n\subset \mathbb Z^n$ where
$|B_i|\geq N$ for all $i\geq 1$;  $B_i = [-\kappa N,-N]\cup[N, \kappa N]$, for $i \in [d]$; and  $|\cB|\leq (\kappa N)^n$.
\end{definition}

The advantage of working with these boxes is that they lend themselves naturally to a probabilistic interpretation, which we now adopt. We ask ``what is the 
probability that 
\[ \Pr_M(\|MX\|_2\leq n)\geq \left(\frac{L}{N}\right)^n, \]
where $X$ is chosen uniformly at random from $\cB$?''. This interpretation was used to ingenious effect in the work of Tikhomirov, who called this the ``inversion of randomness''. While we do take this vantage point, our path forward is considerably different from that of Tikhomirov.

We now state our key ``box'' version of Theorem~\ref{thm:netThm}, in this probabilistic framework. Indeed, almost all of the work in proving Theorem~\ref{thm:netThm} goes into proving the following variant for boxes.

\begin{lemma}\label{thm:invertrandom} For $L \geq 2$ and $0 < c_0 \leq 2^{-24}$, let $n > L^{64/c_0^2}$ and let $\frac{1}{4}c_0^2n\leq d\leq c_0^2 n$. 
For $N \geq 2$, satisfying $\log N\leq c_0 L^{-8n/d} d$, and $\k \geq 2$, let $\cB$ be a $(N,\kappa ,d)$-box and let $X$ be chosen uniformly at random from $\cB$. 
Then 
\[\Pr_X\left(\Pr_M(\|MX\|_2\leq n)\geq \left(\frac{L}{N}\right)^n\right)\leq \left(\frac{R}{L}\right)^{2n},\]
where $R := C c_0^{-3}$ and $C>0$ is an absolute constant.
\end{lemma}

\subsection{Counting with the least common denominator}
In this subsection, we prove the following simple lemma, which says that the probability of choosing $X \in \cB$ with ``large'' least common denominator is super-exponentially small. This will ultimately allow us to apply Theorem~\ref{lem:rankH}, which requires an upper-bound on the $D_{\alpha}( X)$ for application.

We point out that in Lemma~\ref{lem:lcd-rare}, we rescale by a factor of $r_n = c_02^{-4}n^{-1/2}$, despite the fact we are working in $d < n$ dimensions. This is just a trace of the fact that $\R^n$ is our true point of reference. Additionally we will only need Lemma \ref{lem:lcd-rare} when $K=16$.

\begin{lemma}\label{lem:lcd-rare}
For $\alpha \in (0,1), K \geq 1$ and $\k \geq 2$, let $n \geq d\geq K^2/\alpha$ and let $N \geq 2$ be so that $ K N < 2^d $. 
Let $\cB=\left([-\k N,-N]\cup [N,\k N]\right)^d$ and let $X$ be chosen uniformly at random from $\cB$. Then
\begin{equation} \label{eq:lcd-rare} \Pr_X\left( D_\alpha\big( r_n \cdot X \big) \leq K \right) \leq (2^{20} \alpha)^{d/4}\, ,\end{equation}
where we have set $r_n := c_02^{-4} n^{-1/2}$.
\end{lemma}
\begin{proof}
If $D_\alpha\big( r_n \cdot X \big) \leq K$ then let $\psi \in (0,K]$ be the minimum\footnote{Technically the least common denominator is defined in terms of 
an infimum, however the minimum is always attained for non-zero vectors.} in the definition of least common denominator. 
Set $\phi := r_n \psi$ and observe that $\phi$ satisfies  
\begin{equation}\label{eq:lcd-conclusion} \| \phi  X \|_{\T} \leq  \sqrt{\alpha d} \quad \text{ and } \quad  \phi \in [(2\k N)^{-1}, r_n K]\,. \end{equation}
To see the bound  $\phi \geq (2\k N)^{-1}$, note that if $\phi < (2\k N)^{-1}$ then each coordinate of $\phi \cdot X$ lies in $(-1/2,1/2)$ which would imply $\|\phi X  \|_{\T} = \|\phi X \|_2 = \phi \|X\|_2$.  Using the non-triviality condition in the definition of least common denominator~\eqref{eq:LCD-def}, this would imply
\[  \phi \|X\|_2 = \|\phi \cdot X \|_{\T} = \|\psi (r_n \cdot X) \|_{\T} \leq \psi \|r_n \cdot X\|_2/2 = \phi \|X\|_2/2, \]
which is a contradiction.  Thus the bounds in \eqref{eq:lcd-conclusion} hold.

Now to calculate the probability in \eqref{eq:lcd-rare}, we discretize the range of possible $\phi$. For each integer 
$i \in [1/\alpha ,2KN/\alpha] =: I$ we define $\phi_i := i\alpha /(2\k N)$ and note that if $X,\phi$ satisfy \eqref{eq:lcd-conclusion} then there exists 
$\phi_i$ for which 
\[ \| \phi_i X \|_\T \leq 2\sqrt{\alpha d} \quad \text{ and } \phi_i \in [(2\k N)^{-1}, r_n K], \] 
by simply choosing $\phi_i$ for which $|\phi_i-\phi|\leq \alpha/(\k N)$ and using triangle inequality
\begin{equation} \label{eq:phi-i-properties} \|\phi_i X \|_{\T} \leq \|\phi X\|_{\T} + \|(\phi_i - \phi) X \|_2 \leq \sqrt{\alpha d}+ |\phi_i-\phi|\cdot  \sqrt{d} (\k N) \leq 2\sqrt{\alpha d}.\end{equation}
Thus we have that 
\begin{equation}\label{eq:lcd-union-bd}	
	\Pr_X(D_\alpha( r_n \cdot X ) \leq K ) \leq \sum_{i \in I}  \Pr_X \left( \|\phi_i X\|_{\T} \leq 2\sqrt{\alpha d}\right).\end{equation}

To bound the terms on the right-hand side, note that if $\| \phi_i X\|_{\T} \leq 2 \sqrt{\alpha d}$ then $$\frac{1}{d} \sum_{j=1}^d \| \phi_i X_j \|_{\T}^2 \leq 4 \alpha\,.$$
By averaging, there is a set $S(X,i) \subset [d]$ with $|S(X,i)| \geq d/2$ for which $\| \phi_i X_j \|_\T \leq  4\sqrt{\alpha}$ for all $j \in S(X,i)$.
Union bounding over all sets $S \subseteq [d]$ and using the independence of the coordinates $X_j$ we have
\begin{equation} \label{eq:lcd-bd-2}\Pr_X(D_\alpha( r_n \cdot X ) \leq K ) \leq 2^d \sum_{i \in I} \, \prod_{j=1}^{d/2}\, \PP_{X_j}\left( \|\phi_i X_j\|_{\T} \leq 4\sqrt{\alpha} \right) .\end{equation} 
We now claim that 
\begin{equation}\label{eq:Coord-bnd} \PP_{X_j}\left( \|\phi_i X_j\|_{\T} \leq 4 \sqrt{\alpha} \right) \leq  32 \sqrt{\alpha}.\end{equation} For this,
 note that if $\|\phi_i X_j\|_{\T} \leq 4 \sqrt{\alpha}$, then $|\phi_i X_j-p|\leq 4 \sqrt{\alpha}$, where $p \in \Z$ satisfies 
$|p|\leq |\phi_i X_j| + 1 \leq \phi_i \k N  +1 =: T_i$.  And so
\begin{align*} \PP_{X_j}(\|\phi_i X_j\|_{\T} &\leq 4\sqrt{\alpha}) \leq 
\sum_{p=-T_i}^{T_i} \PP_{X_j}(|X_j-p\phi_i^{-1}|\leq 4\sqrt{\alpha} \phi_i^{-1} ) \leq  \frac{(2 T_i + 1)(8 \alpha^{1/2} \phi_i^{-1} +1)}{2(\kappa - 1)N}. \end{align*}
where we have used that $X_j$ is uniform on $[-\k N, -N]\cup [N,\k N]$ and the lower bound $\k N \phi_i \geq 1/2$ from \eqref{eq:phi-i-properties} along with the assumption $\k \geq 2$. Also note that $8 \alpha^{1/2} \phi_i^{-1} \geq 1$ since $\phi \leq r_n K \leq d^{-1/2}K$, allowing us to conclude \eqref{eq:Coord-bnd}.

Now, plugging \eqref{eq:Coord-bnd} into \eqref{eq:lcd-bd-2} and bounding $|I| \leq (2KN/\alpha + 1) \leq 3^d$ completes the proof of Lemma~\ref{lem:lcd-rare}.	
\end{proof}

\subsection{Anti-concentration for linear projections of random vectors}\label{subsec:AX-anticoncentration}

In this subsection we prove the following anti-concentration result for random variables $HX$, where $H$ is a \emph{fixed} matrix and $X$ is a random vector with 
independent entries. One small remark regarding notation: $H$ as stated in Lemma~\ref{lem:LwO-for-AX} will actually be $H^T$ in Section~\ref{sec:invertrandom}.

\begin{lemma}\label{lem:LwO-for-AX}
Let $N \in  \N$, $n,d,k \in \N$ be such that $n-d \geq 2d > 2k$, $H$ be a $2d \times (n-d)$ matrix with $\s_{2d-k}(H)\geq c_0\sqrt{n}/16$ and $B_1,\ldots, B_{n-d}\subset \Z$ with $|B_i|\geq N$. 
If $X$ is taken uniformly at random from $\cB:=B_1\times \ldots \times B_{n-d}$, then
\[ \Pr_X(\|HX\|_2\leq n)\leq \left(\frac{Cn}{dc_0 N}\right)^{2d-k},\]
where $C>0$ is an absolute constant. 
\end{lemma}

We derive this from the following anti-concentration result of Rudelson and Vershynin. This is essentially Corollary 1.4 along with Remark 2.3 in their paper \cite{rudelson2015small}, but we have restated their result slightly to better suit our context.

\begin{theorem}\label{thm:RV-randv}
Let $N \in  \N$ and let $n,d,k \in \N$ be such that $n-d \geq 2d > k$. Let $P$ be an orthogonal projection of $\R^{n-d}$ onto a $(2d-k)$-dimensional subspace and let $X = (X_1,\ldots,X_{n-d})$ 
be a random vector with independent entries for which 
\[ \cL\big( X_i, 1/2 \big)\leq N^{-1}, \] 
for all $i \in [n-d]$. Then, for all $K\geq 1$,
\[ \max_{y \in \R^{n-d}}\, \Pr_X\big( \|PX-y\|_2\leq K\sqrt{2d-k} \big) \leq \left(\frac{CK}{N}\right)^{2d-k},\]
where $C >0$ is a absolute constant.
\end{theorem}

\vspace{3mm}

\noindent We can now deduce Lemma~\ref{lem:LwO-for-AX}.

\vspace{3mm}

\begin{proof}[Proof of Lemma~\ref{lem:LwO-for-AX}]
Since $H^TH$ is a symmetric $(n-d)\times (n-d)$ matrix with $\rk(H) \leq 2d$, by the spectral theorem we have $ H^TH = \sum_{i=1}^{2d} \s_i(H)^2 v_iv_i^T$, where $v_1,\ldots,v_{2d} \in \R^{n-d}$ are orthonormal. Define the orthogonal projection 
$P :=  \sum_{i=1}^{2d-k} v_iv_i^T$. 
Then we have 
\[\|HX\|_2^2 = \la X, H^TH X \ra = \sum_{j=1}^{2d} \sigma_j(H)^2\langle X,v_j\rangle^2\geq \sigma_{2d-k}(H)^2\sum_{j=1}^{2d-k}\langle X,v_j\rangle^2
\geq 2^{-8}c_0^2 n\|PX\|_2^2.\] Therefore 
\begin{equation}\label{eq:PboundH}
\PP_X(\|HX\|_2\leq n)\leq \Pr_X(\|PX\|_2\leq 16c_0^{-1}\sqrt{n}).
\end{equation} We now apply Theorem~\ref{thm:RV-randv} to the orthogonal projection $P$ with $K=16c_0^{-1}\sqrt{n/(2d-k)}$ to see
\begin{equation}\label{eq:RVappP} \Pr_X(\|PX\|_2\leq K\sqrt{2d - k}) \leq \left(\frac{Cn}{dc_0 N}\right)^{2d-k}, \end{equation}
which together with \eqref{eq:PboundH} completes the proof of Lemma~\ref{lem:LwO-for-AX}.\end{proof}

\subsection{Proof of Theorem~\ref{thm:invertrandom} }\label{sec:invertrandom}
We take a moment to prepare the ground for the proof of Theorem~\ref{thm:invertrandom}. We express our random matrix $M,$ as in the statement of Theorem~\ref{thm:invertrandom}, as 
\[ M =  
\begin{bmatrix}
{\bf 0 }_{[d]\times [d]} & H^T_1 \\
H_1 & { \bf 0}_{[n-d] \times [n-d]},  
\end{bmatrix} \] 
Where $H_1$ is a $(n-d) \times d$ random matrix with iid $1/4$-lazy entries in $\{-1,0,1\}$. We shall also let $H_2$ be an independent copy of $H_1$ and define $H$ to be the $ (n-d) \times 2d $ matrix
\[ H := \begin{bmatrix}
H_1 & H_2 \end{bmatrix} .\]
For a vector $X \in \R^n$, we define the event $\cA_1 = \cA_1(X)$ by
\[ \cA_1 := \left\lbrace H :  \|H_1 X_{[d]}\|_2\leq n \text{ and } \|H_{2} X_{[d]}\|_2\leq n \right\rbrace \]
and let $\cA_2 = \cA_2(X)$ be the event
\[ \cA_2 := \left\lbrace H : \|H^T X_{[d+1,n]}\|_2\leq 2n \right\rbrace.  \]  

We now note a simple inequality linking $H$, $\cA_1$ and $\cA_2$ with the event $\{ \|MX\|_2 \leq n \}$.

\begin{fact}\label{fact:2ndMoment} For $X \in \R^n$, let $\cA_1 =\cA_1(X)$, $\cA_2 = \cA_2(X)$ be as above. We have
\[ \left( \P_M(\|M X \|_2 \leq n) \right)^2 \leq \PP_{H}(\cA_1 \cap \cA_2) . \]
\end{fact}
\begin{proof} Let $M'$ be an independent copy of $M$.  Expand $\1( \| MX\|_2 \leq n)$ as a sum of indicators, apply $\EE_M$ and square to see
\[ \left( \P_M(\|M X \|_2 \leq n) \right)^2  = \sum_{M,M'} \PP(M')\PP(M)\1(\|MX\|_2,\|M'X\|_2 \leq n), \]
which is at most 
\[  \sum_{H_1,H_2} \PP(H_1)\PP(H_2)\1\big( \|H_1X_{[d]}\|_2 \leq n, \|H_2X_{[d]}\|_2 \leq n \text{ and } \|H^T X_{[d+1,n]}\|_2 \leq 2n \big) ,\] 
which is exactly $\P_H(\cA_1 \cap \cA_2)$. 
\end{proof}

\vspace{3mm}

We shall also need a ``robust'' notion of the rank of the matrix $H$: Define $\cE_k$ to be the event 
\[ \cE_k := \left\lbrace H : \s_{2d-k}(H)\geq c_0\sqrt{n}/16 \text{ and } \s_{2d-k+1}(H)<  c_0\sqrt{n}/16 \right\rbrace \]
and note that always exactly one of the events $\cE_0,\ldots,\cE_{2d}$ holds.

We now set 
\begin{equation} \label{eq:defalpha} \alpha:= 2^{13}L^{-8n/d} ,\end{equation}
and, given a box $\cB$, we define the set of \emph{typical} vectors $T(\cB) \subseteq \cB$ to be 
\begin{equation} T = T(\cB) := \left\lbrace X \in \cB :  D_{\alpha}(c_02^{-4} n^{-1/2} X_{[d]}) > 16 \right\rbrace. \end{equation}
Now set $K := 16$ and note that Lemma~\ref{lem:lcd-rare} implies that if $X$ is chosen uniformly from $\cB$ and $n  \geq L^{64/c_0^2}\geq 2^8/\alpha$ we have 
\begin{align}\label{eq:Tbd}
\Pr_X(X\not \in T)=\Pr_X\big( D_{\alpha}(c_0 2^{-4} n^{-1/2}  X_{[d]} ) \leq 16 \big)\leq \left(2^{33}L^{-8n/d}\right)^{d/4}\leq \left(\frac{2}{L}\right)^{2n}.
\end{align}

\begin{proof}[Proof of Lemma~\ref{thm:invertrandom}]
Let $M$, $H_1,H_2$, $H$, $\cA_1,\cA_2$, $\cE_k$, $\alpha$ and $T := T(\cB)$ be as above. We denote 
\[ \cE := \Big\lbrace X \in \cB : \PP_M(\|MX\|_2\leq n)  \geq (L/N)^n  \Big\rbrace \] and write
\[ \Pr_X( \cE )  \leq \PP_X( \cE  \cap \{ X \in T \} ) + \PP_X( X \not\in T).  \] 
Now define 
\[ f(X) := \PP_M(\| MX\|_2 \leq n)\1( X \in T ) \] 
and apply~\eqref{eq:Tbd}, the bound on $\PP_X(X \not\in T)$, to obtain
\begin{equation}\label{eq:ProbE-bound}  \Pr_X( \cE ) \leq \PP_X\left( f(X) \geq (L/N)^n\right)  + (2/L)^{2n} \leq (N/L)^{2n}\EE_X\, f(X)^2 + (2/L)^{2n}, \end{equation}
where the last inequality follows from Markov's inequality.  So to prove Lemma~\ref{thm:invertrandom}, it is enough to 
prove $\EE_X\, f(X)^2 \leq 2(R/N)^{2n}$.

From Fact~\ref{fact:2ndMoment} we may write
\begin{equation} \label{eq:PMexpress}
 \P_M(\|M X \|_2 \leq n)^2  \leq \PP_H(\cA_1 \cap \cA_2) = \sum_{k=0}^d \PP_H( \cA_2 | \cA_1 \cap \cE_k)\PP_H(\cA_1 \cap \cE_k) \end{equation}
and so 
 \begin{equation}\label{eq:fsquare} f(X)^2 \leq \sum_{k=0}^d \PP_H( \cA_2 | \cA_1 \cap \cE_k)\PP_H(\cA_1 \cap \cE_k)\1( X \in  T) . \end{equation}
We now look to apply Lemma~\ref{lem:rankH} to obtain upper bounds for the quantities $\PP_H(\cA_1 \cap \cE_k)$, when $X \in T$. For this, note that $d\leq c_0^2 n$, $N\leq \exp(L^{-8n/d}d)\leq \exp(2^{-10}\alpha n)$ and set $R_0 := 2^{39}c_0^{-3}$ (This is the ``$R$'' in Theorem~\ref{lem:rankH}). Also note that,
by the definition of a $(N,\kappa,d)$-box and the fact that $d\geq \frac{1}{4}c_0^2 n$, we have that $\|X_{[d]}\|_2 \geq d^{1/2}N \geq c_02^{-10}\sqrt{n}N$. Now set 
$\alpha':=2^{-10}\alpha$ to see that for $X \in T$ and $0\leq k \leq \alpha' d$,
\[ \P_H(\cA_1 \cap \cE_k ) \leq \exp(-c_0 n k/4)\left(\frac{R_0}{N} \right)^{2n - 2d}\, . \]
Moreover by Theorem~\ref{lem:rankH},
\[ \sum_{k \geq \alpha' d} \PP_H(\cA_1 \cap \cE_k) \leq \PP_H\big( \{ \sigma_{2d-\alpha' d}(H) \leq c_0\sqrt{n}/16 \} \cap \cA_1  \big) \leq \exp(-c_0 \alpha' dn/4).\]
Thus, for all $X \in \cB$, we have 
\begin{equation}\label{eq:sing-sq-uncond}
f(X)^2 \leq  \sum_{k = 0}^{\alpha' d} \P_H(\cA_2 \,|\, \cA_1 \cap \cE_k)\exp(-c_0 n k/4)\left(\frac{R_0}{N}\right)^{2n-2d} + \exp(-c_0 \alpha' dn/4)\,.
\end{equation}
We now consider the quantities $g_k(X) := \P_H(\cA_2 \,|\,\cA_1 \cap \cE_k)$ appearing in \eqref{eq:sing-sq-uncond}. Indeed, 
$$\E_X[ g_k(X) ] = \E_X \E_H\big[ \cA_2 \,|\,\cA_1 \cap \cE_k \big] = \E_{X_{[d]}}\, \E_H\left[ \E_{X_{[d+1,n]}} \1[\cA_2] \,\big\vert\, \cA_1 \cap \cE_k \right]\,.$$
We now consider a fixed $H\in \cA_1 \cap \cE_k$ for $k \leq \alpha'd$. Each such $H$ has $\sigma_{2d-k}(H) \geq c_0 \sqrt{n}/16$ and thus we may apply Lemma~\ref{lem:LwO-for-AX} to see that 
$$\E_{X_{[d+1,n]}}\, \1[\cA_2] = \P_{X_{[d+1,n]}}( \|H^T X_{[d+1,n]} \|_2 \leq n ) \leq \left(\frac{C'n}{c_0 d N}\right)^{2d - k} \leq \left(\frac{4C'}{c_0^3 N}\right)^{2d - k},  $$ 
for an absolute constant $C'>0$, using that $d\geq \frac{1}{4}c_0^2 n$.
And so for each $0\leq k \leq \alpha' d$, taking $R := \max\{ 8C' c_0^{-3}, 2R_0\} $, we have
\begin{equation} \label{eq:gk-bnd} \E_X[ g_k(X) ] \leq \left(\frac{R}{2N}\right)^{2d - k}. \end{equation}
We apply $\EE_X$ to \eqref{eq:sing-sq-uncond} and then use \eqref{eq:gk-bnd} to obtain
\[ \EE_X f(X)^2 \leq \left(\frac{R}{2N}\right)^{2n} \sum_{k=0}^{\alpha' d} \left(\frac{2N}{R}\right)^k \exp(-c_0nk/4)  +\exp(-c_0 \alpha' dn/4). \]
Using that $N \leq \exp(c_0 n/4)$ and $N\leq \exp(c_0L^{-8n/d} d)= \exp(c_0\alpha' d/8)$ gives 
\begin{equation}\label{Ef2-bnd} \EE_X\, f(X)^2 \leq   2 \left(\frac{R}{2N}\right)^{2n}. \end{equation}
Combining \eqref{Ef2-bnd} with \eqref{eq:ProbE-bound} completes the proof of Lemma~\ref{thm:invertrandom}.
\end{proof}

\subsection{Proof of Theorem~\ref{thm:netThm}}

The main work of this section is now complete with the proof of Lemma~\ref{thm:invertrandom}. We now just need to go from $X$ in a ``box'' to $X$ in 
a ``sphere'' $\L_{\eps}$.  To accomplish this step, we simply cover the sphere with boxes. Recall that 
\[ \cI'([d])  := \left\lbrace v \in \R^{n} :  \k_0 n^{-1/2} \leq |v_i| \leq  \k_1 n^{-1/2} \text{ for all } i\in [d]   \right\rbrace, \] 
\[ \L_{\eps} := B_n(0,2) \cap \big( 4\eps n^{-1/2} \cdot \Z^n\big) \cap \cI'([d]), \]
and that $0 < \k_0 < 1 < \k_1$ are absolute constants defined in Section~\ref{sec:Definitions}.

\begin{lemma}\label{lem:covZBall} For all $\eps\in[0,1]$, $\k \geq \max\{\k_1/\k_0,2^8 \kappa_0^{-4} \}$, there exists a family $\cF $ of $(N,\k,d)$-boxes with $|\cF| \leq \k^n$ so that 
\begin{equation}\label{eq:covZBall} \L_{\eps} \subseteq  \bigcup_{\cB \in \cF} (4\eps n^{-1/2}) \cdot \cB\, , \end{equation}
where $N =  \k_{0}/(4\eps)$.
\end{lemma}
\begin{proof}
For $\ell\geq 1$ define the interval of integers $I_\ell:=  \left[-2^{\ell}N, 2^{\ell}N \right]  \backslash \left[-2^{\ell-1}N, 2^{\ell-1}N \right]$
and $I_0 :=[-N, N]$. Also take $J := [-\k N, \k N] \backslash [-N, N]$. For $(\ell_{d+1},\ldots, \ell_n)\in \Z_{\geq 0}^n$ we define the box 
$B(\ell_{d+1},\ldots, \ell_n):= J^d \times \prod_{j=d+1}^n I_{\ell_j}\, $ and the family of boxes
\[ \cF:= \left\lbrace  B(\ell_{d+1},\ldots, \ell_n) : \sum_{j: \ell_j>0} 2^{2\ell_j}\leq 8n/\k_0^2 \right\rbrace.  \] 

We claim that $\cF$ is the desired family. For this, we first show the inclusion at \eqref{eq:covZBall}. Let $v \in \L_{\eps}$. Since $v \in 4\eps n^{-1/2} \Z^n$,
$X := v n^{1/2}/(4\eps) \in \Z^n$. For $i \in [d+1,n]$, define $\ell_i$ so that $X_i \in I(\ell_i)$. We claim 
$X \in B(\ell_{d+1},\ldots, \ell_n)$. For this, observe that $X_i \in J$ for $i \in [d]$: since $v \in \cI'([d])$, we have $\k_{0} \leq |v_i|n^{1/2} \leq \k_{1}$, for $i \in [d]$. So $ \k_{0}/(4\eps) \leq |X_i| \leq \k_{1}/(4\eps)$, for $i \in [d]$. Thus $X_i \in J$ since $N = \k_{0}/(4\eps)$ and $\k \geq \k_{1}/\k_{0}$.
Thus $v \in B(\ell_{d+1},\ldots, \ell_n) $. We now observe that $B(\ell_{d+1},\ldots, \ell_n) \in \cF$, since 
\[ \sum_{j: \ell_j>0}2^{2(\ell_j-1)}N^2 \leq \sum_{j=1}^n X_j^2 \leq n/(4\eps)^2\left( \sum_i v_i^2\right) \leq 4nN^2/\k_0^2 .\]
Thus we have \eqref{eq:covZBall}.

We now show $|\cF|\leq \kappa^n$. For this we only need to count the number of sequences $(\ell_{d+1},\ldots, \ell_n)$ of non-negative integers for which $\sum_{\ell_i > 0} 4^{\ell_i} \leq 8n/\k_0^2 $. For each $t\geq 0$ are at most $8n/(4^t \k_0^2) $ values of $i \in [d+1,n]$ for which $\ell_i = t$ and there are at most
 $\binom{n}{\leq 8n/(4^t\k_0^2)}$ choices for these values of $i$. Hence, there are at most 
\[ \prod_{t\geq 0} \binom{n}{\leq 8n/(\k_0^2 4^t)} \leq (\k_0/4)^{-4n} < \k^n\] such tuples.

It only remains to show an upper bound on the size of $B(\ell_{d+1},\ldots,\ell_n) \in \cF$. We have
\[ 
|B(\ell_{d+1}, \ldots, \ell_n)|\leq N^n \k^d 2^{n+\sum_j \ell_j} \leq \k^d(16/\k_0^2)^n N^n \leq (\k N)^n\]
where the second inequality holds due to the fact $\prod_j 2^{\ell_j}\leq \left(\frac{1}{n}\sum_j 2^{2\ell_j}\right)^n\leq (8/\k_0^2)^n$ and the last inequality holds due to the choice of $\k$.\end{proof}

\vspace{3mm}

We may now use our covering Lemma~\ref{lem:covZBall} to apply Theorem~\ref{thm:invertrandom} to deduce Theorem~\ref{thm:netThm},
the main result of this section.

\begin{proof}[Proof of Theorem~\ref{thm:netThm}]
Apply Lemma~\ref{lem:covZBall} with $\kappa = \max\{\k_1/\k_0,2^8 \kappa_0^{-4} \}$ and use the fact that $\cN_{\eps} \subseteq \L_{\eps}$ to write 
\[ \cN_{\eps} \subseteq \bigcup_{\cB \in \cF} \left(  (4\eps n^{-1/2}) \cdot  \cB \right) \cap \cN_{\eps}   \] 
and so
\[ |\cN_{\eps}| \leq \sum_{\cB \in \cF} | (4\eps n^{-1/2} \cdot \mathcal B ) \cap \mathcal N_{\eps}| 
\leq |\cF| \cdot \max_{\cB \in \cF}\, | (4\eps n^{-1/2} \cdot \mathcal B ) \cap \mathcal N_{\eps}|. \] 
By rescaling by $\sqrt{n}/(4\eps)$ and applying Lemma~\ref{thm:invertrandom}, we have 
\[ | (4\eps n^{-1/2} \cdot \mathcal B ) \cap \mathcal N_{\eps}| 
 \leq \left|\Big\lbrace X \in \cB : \PP_M(\|MX\|_2\leq n) \geq (L\eps)^n \Big\rbrace \right| \leq  \left(\frac{R}{L} \right)^{2n} |\cB|. \]
Here the application of Lemma~\ref{thm:invertrandom} is justified as $0 < c_0 \leq 2^{-24}$, $c_0^2 n/2 \leq d \leq c_0^2 n$; $\k \geq 2$; we have\ $\log 1/\eps \leq n/L^{32/c_0^2}$ and therefore
\[ \log N = \log \k_0/(4\eps) \leq n/L^{32/c_0^2} \leq c_0L^{-8n/d}d,\] as specified in Lemma~\ref{thm:invertrandom}, since $\k_0<1$, $d\geq L^{-1/c_0^2}n$, $c_0\geq L^{-1/c_0^2}$ and $8n/d\leq 16/c_0^2$. So, using that $|\cF| \leq \k^{n}$ and $|\cB| \leq (\k N)^n$ for each $\cB \in \cF$, we have
\[ |\cN_{\eps}| \leq \k^{n} \left(\frac{R}{L} \right)^{2n} |\mathcal B| 
\leq \k^{n}\left(\frac{R}{L} \right)^{2n} (\k N)^n \leq \left(\frac{C}{c_0^6L^2\eps}\right)^{n},\] where $C=\kappa^2 R^2c_0^{6}$, thus completing the proof of Theorem~\ref{thm:netThm}.
\end{proof}

\section{Nets for structured vectors: approximating with the net}\label{sec:approx-w-net}
While we have spent considerable energy up to this point showing that $\cN_{\eps}$ is small, we have so far not shown that it is in fact a \emph{net}. 
We now show just this, by showing that vectors in $\Sigma_{\eps}$ are approximated by elements of $\cN_{\eps}$. As we will see, this is considerably easier and is taken care of in Lemma~\ref{thmnet}, which, in a similar spirit to Lemma~\ref{lem:basis-net}, is based on randomized rounding. For this, we recall that we defined
\begin{align}\label{eq:Sigmaepsdef} \Sigma_\eps =\left\{v\in \cI([d]) : \cT_L(v)\in [\eps,2\eps] \right\}\subset \mathbb S^{n-1}\, ,\end{align}
where $\cT_L(v)= \sup\{t\in[0,1]: \PP(\|Mv\|_2\leq t\sqrt{n}) \geq (4Lt)^n\},$ and $d =  c_0^2 n < 2^{-32} n$. Also recall the definition of our net
\[ \cN_{\eps} = \left\lbrace  v \in \L_{\eps} :   \PP(\|Mv\|_2\leq 4\eps\sqrt{n}) \geq (L\eps)^n \text{ and }  \cL_{A,op}(v,\eps\sqrt{n}) \leq (2^8 L\eps)^n \right\rbrace. \] 
We also make the basic observation that if $\cT_L(v)=s$, then 
\[
(2sL)^n\leq\PP(\|Mv\|_2\leq s\sqrt{n}) \leq (8sL)^n\, .
\]

Until now, we have almost entirely been working with the matrix $M$. The following lemma allows us to make a comparison between $M$ and our central object of study: $A$, a uniform $n\times n$ symmetric matrix with entries in $\{-1,1\}$. The proof of the lemma is based on a comparison of Fourier transforms and is deferred to Appendix~\ref{sec:replacement}. This is similar to the replacement step in the work of Kahn Koml\'{o}s and Szemer\'{e}di \cite{KKS} and subsequent works \cite{TV-JAMS,BVW}. However here,
we only need to ``break even'', whereas they are looking for a substantial gain at this step.

\begin{lemma}\label{lem:replacement}
	For $v \in \R^n$ and $t \geq \cT_L(v)$ we have
	
\[ \cL(Av,t\sqrt{n}) \leq (50 L t)^n \,. \] 
\end{lemma}

\vspace{3mm}

\noindent We now prove Lemma~\ref{thmnet} which tells us that $\cN_{\eps}$ is a net for $\Sigma_{\eps}$.

\begin{lemma}\label{thmnet} Let $\eps\in (0,\k_0/8)$, $d \leq n/32$. 
If $v \in \Sigma_{\eps}$ then there is $u \in \cN_{\eps}$ with $\|u-v\|_{\infty} \leq 4\eps n^{-1/2}$.
\end{lemma}
\begin{proof}
Given $v \in \Sigma_{\eps}$, we define a random variable $r = (r_1,\ldots,r_n)$ where the $r_i$ are independent, $\EE\, r_i = 0 $, $|r_i| \leq  4\eps n^{-1/2}$
and such that $v - r\in 4 \eps n^{-1/2} \Z^n$, for all $r$. We then define the random variable $u := v - r$. We will show that with positive probability there is a choice of $u\in \cN_{\eps}$.

Note that $\|r\|_{\infty} = \|u - v\|_{\infty} \leq 4\eps n^{-1/2}$ for all $u$. Also, $u \in \cI'([d])$ for all $u$, since $v \in \cI([d])$ and 
$\|u-v\|_{\infty} \leq 4\eps/\sqrt{n} \leq \k_0/(2\sqrt{n})$. 
So, from the definition of $\cN_{\eps}$, we need only show that there exists such a $u$ satisfying 
\begin{equation}\label{eq:lem-net-goal} \PP(\|Mu\|_2\leq 4\eps\sqrt{n}) \geq (L\eps)^n   \text{ and }   \cL_{A,op}(u,\eps \sqrt{n}) \leq (2^8 L\eps)^n. \end{equation}
We first show that \emph{all} $u$ satisfy the upper bound at \eqref{eq:lem-net-goal}. To see this,  write 
 $\cE = \{\|A\|\leq 4\sqrt{n} \}$ and let $w(u) \in \R^n$, be such that  
\begin{align*}
 \cL_{A,op}(u,\eps \sqrt{n}) &= \PP\left( \|Av - Ar - w(u)\| \leq \eps \sqrt{n} \text{ and }  \cE \right) \\
  &\leq \PP\left( \|Av - w(u)\| \leq 5\eps \sqrt{n} \text{ and } \cE \right) \\
  &\leq \cL_{A,op}(v,5\eps\sqrt{n} )  \leq \cL(Av, 5\eps\sqrt{n}). \end{align*}
  Since $v \in \Sigma_{\eps}$, Lemma~\ref{lem:replacement} bounds 
  \begin{align} \cL(Av, 5\eps\sqrt{n})\leq ( 2^{8} L \eps)^n\,.\end{align}
   
We now show that 
\begin{equation} \label{eq:lemnetEgoal} \EE_u\,  \PP_M(\|Mu\|_2\leq 4\eps\sqrt{n}) \geq (1/2)\PP_M(\|Mv\|_2\leq 2\eps\sqrt{n}) \geq (1/2)(2\eps L)^n  \, , \end{equation}
where the last inequality holds by the fact $v \in \Sigma_{\eps}$. From \eqref{eq:lemnetEgoal}, it follows that there exists $u \in \L_{\eps}$ satisfying \eqref{eq:lem-net-goal}.

So to prove the first inequality in \eqref{eq:lem-net-goal}, we define the event $\cE := \{ M : \|Mv\|_2 \leq 2\eps \sqrt{n} \}$. For all $u$, we have  
\[ \PP_M(\|Mu\|_2\leq 4\eps\sqrt{n})  = \PP_M( \|Mv - Mr\|_2 \leq  4\eps\sqrt{n}) \geq \PP_M( \|Mr\|_2 \leq 2\eps \sqrt{n} \text{ and } \cE ); \]
Thus
\begin{align*} \PP_M(\|Mu\|_2\leq 4\eps\sqrt{n}) &\geq \PP_M( \|Mr\|_2 \leq 2\eps \sqrt{n} \, \big\vert \cE ) \PP( \cE ) \\
&\geq  \left(1 - \PP_M( \|Mr\|_2 > 2\eps \sqrt{n}\, \big\vert \cE )\right)\PP_M(\|Mv\|_2\leq 2\eps\sqrt{n}) \,. \end{align*}
Taking expectations with respect to $u$ gives,
\begin{equation}\label{eq:netLemEE} \EE_{u}\PP_M(\|Mu\|_2\leq 4\eps\sqrt{n}) \geq \big(1 - \EE_u \PP_M( \|Mr\|_2 > 2\eps \sqrt{n}\, \big\vert \cE ) \big)\PP_M(\|Mv\|_2\leq 2\eps\sqrt{n}) \end{equation}
and exchanging the expectations reveals that it is enough to show 
 \[ \EE_M\big[ \PP_r( \|Mr\|_2 > 2\eps \sqrt{n})\, \big\vert\, \cE \big] \leq 1/2. \]
 We will show that $\PP_r( \|Mr\|_2 > 2\eps \sqrt{n}) \leq 1/4$ for all $M \in \cE$, by Markov's inequality. For this, fix a $n\times n$ matrix $M$ with entries $|M_{i,j}|\leq 1$ and $M_{i,j}=0$, if $(i,j)\in [d+1,n]\times [d+1,n]$, and note that 
\[ \EE_r\, \|Mr\|_2^2 = \sum_{i,j} \EE \left( M_{i,j}r_i \right)^2 = \sum_{i} \EE\, r_i^2 \sum_{j} M_{i,j}^2 \leq 32\eps^2 d\leq \eps^2 n, \]
where, for the second equality, we have used that the $r_i$ are mutually independent and $\EE\, r_i = 0$, for the third inequality, we used $\|r\|_\infty\leq 4\eps/\sqrt{n}$ and for the final inequality we used $d\leq n/32$. Thus by Markov, we have 
\begin{equation} \label{eq:NetlemMarkov} \PP_{r}(  \|Mr\|_2 \geq 2\eps\sqrt{n}) \leq  (2\eps \sqrt{n})^{-2} \EE_r\, \|Mr\|_2^2 \leq 1/4 . \end{equation}
Putting \eqref{eq:NetlemMarkov} together with \eqref{eq:netLemEE} proves \eqref{eq:lemnetEgoal}, completing the proof of \eqref{eq:lem-net-goal}.\end{proof}

\section{Proof of Theorem~\ref{thm:main}}\label{sec:ProofOthm}

In this section we put together our results to prove Theorem~\ref{thm:main}. But before we get to this, we note a few reductions afforded by previous work.
Let us define
\begin{equation}\label{eq:defq_n} q_n(\g) :=\max_{w\in \R^n}\, \Pr_A(\exists v\in \R^n\setminus \{0\}:Av=w,~\rho(v)\geq \gamma),\end{equation}
where 
\[ \rho(v) = \max_{w \in \R} \PP\left(  \sum_{i=1}^n \eps_i v_i = w \right)
\] and $\eps_1,\ldots,\eps_n \in \{-1,1 \}$ are i.i.d.\ and uniform. One slightly irritating aspect of the definition \eqref{eq:defq_n} is that the existential 
quantifies over \emph{all non-zero} $v \in \R^n$, rather than all $v \in \S^{n-1}$, as we have been working with. So, as we will shortly see, we will need to approximate this extra dimension of freedom with a net. 

These small issues aside, we will use the following inequality, which effectively allows us to remove very unstructured vectors from consideration.
\begin{lemma}\label{lem:unstructured}
Let $A$ be a random $n \times n$ symmetric $\{-1,1\}$-matrix. For all $\g>0$ we have \[\Pr(\det(A)=0)\leq 16n \sum_{m=n}^{2n-2}\left(\gamma^{1/8}+\frac{q_{m-1}(\gamma)}{\gamma}\right)\]
\end{lemma}
We record the details of this lemma in Appendix~\ref{sec:unstructured}, although an almost identical lemma can be found in \cite{CMMM}, which collected elements
from \cite{nguyen-singularity,ferber-jain, costello-tao-vu}.

\subsection{Non-flat vectors}\label{subsec:nonflatvectors} Here we note a lemma due to Vershynin \cite{vershynin-invertibility} which tells us that it is enough 
for us to consider vectors $v \in \cI$. For this, we reiterate the important notion of  \emph{compressible vectors}, introduced by Rudelson and Vershynin \cite{RV}.
Say a vector in $\S^{n-1}$ is $(\delta,\rho)$-compressible if it has distance $\leq\rho$ from a vector with support $\leq \delta n$. Let $\Comp(\delta,\rho)$ denote 
the set of such compressible vectors. In~\cite[Proposition 4.2]{vershynin-invertibility}, Vershynin provides the following lemma which allows us to disregard all compressible vectors. 

\begin{lemma}\label{lem:compressible}
 There exist $\delta, \rho, c\in(0,1)$ so that for all $n \in \N$,
\[ \max_{w\in\mathbb{R}^n}\, \PP_A\left( \bigcup_{v \in \S^{n-1} \setminus \Comp(\delta,\rho)} \left\lbrace \|Av-w\|_2 \leq c \sqrt{n}\right\rbrace \right) \leq 2e^{-c n}, \] where $A$ is a random $n\times n$ symmetric $\{-1,1\}$-matrix.
\end{lemma}

\noindent The following lemma of Rudelson and Vershynin \cite[Lemma 3.4]{RV} tells us that incompressible vectors are ``flat''
for a constant proportion of coordinates. 

\begin{lemma}\label{lem:spread} For $\delta,\rho \in (0,1)$, let $v\in \Inc(\delta,\rho)$. Then
\[ (\rho/2) n^{-1/2} \leq |v_i| \leq \delta^{-1/2} n^{-1/2} \]
for at least $\rho^2\delta n/2$ values of $i\in[n]$.
\end{lemma}

\noindent Now recall that we defined 
\[ \cI(D) = \left\lbrace v \in \S^{n-1} :  (\k_0 + \k_0/2)n^{-1/2} \leq |v_i| \leq  (\k_1 -\k_0/2) n^{-1/2} \text{ for all } i\in D   \right\rbrace \]
and $\cI = \bigcup_{D \subseteq [n], |D| = d } \cI(D)$.  Here we fix $\k_0 = \rho/3$ and $\k_1 = \delta^{-1/2}+\rho/6$, where $\delta,\rho$ are as in Lemma~\ref{lem:compressible}. We also fix $c_0 = \min\{ 2^{-24}, \rho \delta^{1/2}/2 \}$.

The following lemma is what we will apply in the proof of Theorem~\ref{thm:main}.

\begin{lemma}\label{lem:invertOnIc} For $n \in \N$, let $d < c_0^2 n$. Then 
\[ \max_{w\in\mathbb{R}^n}\PP_A\left( \bigcup_{v \in \S^{n-1} \setminus \cI} \left\lbrace Av\in \{t\cdot w\}_{t>0},~ \|A\|\leq 4\sqrt{n}\right\rbrace \right) \leq 16 c^{-1}e^{-c n}. \]\end{lemma}
\begin{proof}  
Apply Lemma~\ref{lem:spread} along with the definitions of $\k_1,\k_2$ and $\cI$ to see $\S^{n-1}\setminus \cI \subseteq \Comp(\delta,\rho)$. Now take a $c\sqrt{n}$-net $\cX$ for $\{t\cdot w\}_{0<t\leq 4\sqrt{n}}$ of size $8c^{-1}$. 
Then \[\left\lbrace A : Av\in \{t\cdot w\}_{t>0},~ \|A\|\leq 4\sqrt{n}\right\rbrace \subset \bigcup_{w'\in \cX}\left\lbrace A :  \|Av-w'\|_2 \leq c \sqrt{n}\right\rbrace .\] Union bounding over $\cX$ and applying Lemma~\ref{lem:compressible} completes the lemma.\end{proof}

\subsection{Proof of Theorem~\ref{thm:main}}

As we noted in Section~\ref{sec:Definitions}, matrices $A$ with $\|A\|\geq 4\sqrt{n}$ will be a slight nuisance for us. The following concentration inequality for the operator norm of a random matrix will allow us to remove all such matrices $A$ from consideration. 

\begin{lemma}\label{lem:op-concentration}  Let $A$ be uniformly drawn from all $n\times n$ symmetric matrices with entries in $\{-1, 1\}$. Then for $n$ sufficiently large,
	\[ \PP\left(\| A \| \geq 4\sqrt{n}\right)\leq 4e^{-n/32}.\]\end{lemma}
This follows from a classical result of Bai and Yin~\cite{bai1988necessary} (see also~\cite[Theorem 2.3.23]{tao2012topics}) which implies that the median of $\|A\|$ is equal to $(2+o(1))\sqrt{n}$, combined with a concentration inequality due to Meckes~\cite[Theorem 2]{meckes2004concentration}.  A version of Lemma~\ref{lem:op-concentration} without explicit constants, is well-known and straightforward, though we have included a version with explicit constants for concreteness.

We will also need the following, rather weak, relationship between the threshold $\cT_{L}$, defined in terms of the matrix $M$, and $\rho(v)$, the ``one-dimensional'' concentration function of $v$.  For this we define one more bit of (standard) notation
\[ \rho_{\eps}(v) := \max_{b \in \R^n} \PP\left( \sum_i v_i\eps_i \in (b-\eps,b+\eps) \right).\]

\begin{lemma}\label{lem:RhoVtau} Let $v \in \S^{n-1}$ and $\eps = \cT_L(v)$.  Then 
	$\rho_\eps(v)^{4} \leq 2^{12} L \eps $.
\end{lemma}

We postpone the proof of this lemma to Appendix~\ref{sec:replacement} and move on to the proof of Theorem~\ref{thm:main}.

\begin{proof}[Proof of Theorem~\ref{thm:main}]
It is not hard to see that $\PP( \det(A) = 0  ) < 1$ for all $n$, and therefore it is enough to prove Theorem~\ref{thm:main} for all sufficiently large $n$.

Now, as in Section~\ref{sec:Definitions}, we set $\g = e^{-cn}$, where we now define, $c := L^{-32/c_0^2}/8 $, $L := \max\{ 2^{26}C_1, 16/\k_0\}$,
where $C_1 = C/c_0^6$ is the constant appearing in Theorem~\ref{thm:netThm}. We also let $c_0 >0$ be as defined above and $d := \lceil c_0^2n/2\rceil$.

 From Lemma~\ref{lem:unstructured} we have
\[\Pr(\det(A)=0)\leq 16n \sum_{m=n}^{2n-2}\left(\gamma^{1/8}+\frac{q_{m-1}(\g)}{\g}\right)\]
and so it is enough to bound $q_{n}(\g)$ for all large $n$.  Let $\Sigma = \{ v \in \S^{n-1} : \rho(v) \geq \g \}$, as defined in Section~\ref{sec:Definitions}, and note that 
\[\{A: \exists v\in \R^n, ~Av=w,~\rho(v)\geq \gamma\}\subset \{A: \exists v\in \Sigma, ~Av\in\{t\cdot w\}_{t>0}\}.\] 
Since $d = \lceil c_0^2n/2 \rceil$, by Lemma~\ref{lem:invertOnIc} and Lemma~\ref{lem:op-concentration}, we have 
\begin{equation}\label{eq:qn} q_n(\g) \leq \max_{w\in \R^n}\Pr_A\left( \{ \exists v\in \cI \cap \Sigma :~Av\in\{t\cdot w\}_{t>0} \} \cap \{ \|A\|\leq 4\sqrt{n} \}\right)\,  + 32c^{-1}e^{-cn} \end{equation}
and so it is enough to show the first term on the right-hand-side is $\leq 2^{-n}$. Using that $\cI = \bigcup_{D} \cI(D),$ we have the first term of \eqref{eq:qn} is  
\begin{align}
&\leq 2^n \max_{D \in [n]^{(d)}}\, \max_{w\in \R^n}\, \Pr_A\left( \{ \exists v\in \cI(D) \cap \Sigma :~Av\in\{t\cdot w\}_{t>0} \} \cap \{ \|A\|\leq 4\sqrt{n} \}\right) \\
 \label{eq:I([d])bnd}  &= 2^n \max_{w\in \R^n}\, \Pr_A\left( \{ \exists v\in \cI([d]) \cap \Sigma :~Av\in\{t\cdot w\}_{t>0} \} \cap \{ \|A\|\leq 4\sqrt{n} \}\right), \end{align}
where the last line holds by symmetry of the coordinates.  Thus it is enough to show that the maximum at \eqref{eq:I([d])bnd} is at most $4^{-n}$.

Now, for $v \in \Sigma$ we have $\rho(v)\geq \g$ and so, by Lemma~\ref{lem:RhoVtau}, we have that 
\[ \g^4 \leq \rho(v)^4 \leq \rho_{\cT_L(v)}(v)^4 \leq 2^{12}L \cT_L(v). \]
Define $ \eta := \g^4 /(2^{12}L) \leq \cT_L(v)$. Also note that by definition, $\cT_L(v) \leq 1/L \leq \k_0/8$.

Now, recalling definition~\eqref{eq:Sigmaepsdef} of $\Sigma_{\eps} = \Sigma_{\eps}([d])$ from Section~\ref{sec:Definitions}, we may write
\[\cI([d]) \cap \Sigma \subseteq \bigcup_{i=1}^n \left\{v\in \cI : \cT_L(v)\in [2^{j-1}\eta,2^{j}\eta] \right\}\, = \bigcup_{j=0}^{ \log_2 (\kappa_0/16\eta)} \Sigma_{2^j\eta}\,  \]
and so by the union bound, it is enough to show
\begin{equation*}
\max_{w\in \R^n}\Pr_A\left(\{ \exists v\in \Sigma_{\eps}  :~Av\in \{t\cdot w\}_{t>0} \} \cap \{ \|A\|\leq 4\sqrt{n} \}\right) \leq 8^{-n},\end{equation*}
for all $\eps \in [\eta,\k_0/16]$. Fix an $\eps \sqrt{n}$-net $\cX$ for $\{t\cdot w\}_{0<t\leq 4\sqrt{n}}$ of size $8/\eps\leq 2^n$ to get \[\{A:~ Av\in \{t\cdot w\}_{t>0},~\|A\|\leq 4\sqrt{n}\}\subset \bigcup_{w'\in \cX}\{A:~\|Av-w'\|_{2}\leq \eps \sqrt{n},~\|A\|\leq 4\sqrt{n}\}.\] So by taking the union bound over $\cX$ it is enough to prove that
\begin{equation}\label{eq:bndSEps} 
Q_{\eps} := \max_{w\in \R^n}\Pr_A\left(\{ \exists v\in \Sigma_{\eps}  :~\|Av-w\|_2\leq \eps \sqrt{n} \} \cap \{ \|A\|\leq 4\sqrt{n} \}\right) \leq 2^{-4n}.\end{equation}

Let $w\in \R^n$ be such that the maximum at \eqref{eq:bndSEps} is attained. Now, since $\eps < \k_0/8$ for $v \in \Sigma_{\eps}$, we apply Lemma~\ref{thmnet}, 
to find a $u \in \cN_{\eps} = \cN_{\eps}([d])$ so that $\|v - u\|_2 \leq 4\eps$. So if $\| A \|\leq 4\sqrt{n}$ and $\|Av - w\|\leq \eps \sqrt{n}$, we see that 
\[ \|Au -w\|_2\leq \|Av -w\|_2 + \|A(v-u)\|_2 \leq \|Av -w\|_2 + \|A\|\|(v-u)\|_2 \leq 32\eps\sqrt{n}\]
and thus
\[ \{ A : \exists v\in \Sigma_{\eps}  :~\|Av-w\|\leq \eps \sqrt{n} \} \cap \{ \|A\|\leq 4\sqrt{n} \} \subseteq \{ A : \exists u \in \cN_{\eps} : \| Au-w\|\leq  32\eps\sqrt{n} , \| A \| \leq 4\sqrt{n}\}.   \] 
So, by union bounding over our net $\cN_{\eps}$, we see that 
\[ Q_{\eps} \leq \PP_A\left(\exists u \in \cN_{\eps} : \|Au-w\|\leq  32\eps\sqrt{n} \text{ and } \|A\| \leq 4\sqrt{n} \right) 
\leq \sum_{u \in \cN_{\eps}} \cL_{A,op}\left(u, 32\eps \sqrt{n} \right). \]
Now note that if $u \in \cN_{\eps}$, then $\cL_{A,op}(u,\eps\sqrt{n}) \leq (2^8 L \eps)^n$ and so by Fact~\ref{fact:regularityofL} we have that $\cL_{A,op}\left(u, 32\eps \sqrt{n} \right) \leq (2^{16} L\eps)^n$. As a result, 
\[  Q_{\eps}  \leq |\cN_{\eps}|(2^{16} L\eps)^n \leq \left(\frac{C}{L^2\eps}\right)^n(2^{16} L\eps)^n \leq 2^{-4n}. \]
where the second to last inequality follows from our Theorem~\ref{thm:netThm} and the last inequality holds for our choice of $L = \max\{ 2^{26}C_1, 16/\k_0\}$. To see that the application 
of Theorem~\ref{thm:netThm} is valid, note that 
\[ \log 1/\eps \leq \log 1/\eta = \log 2^{12} L/\g^4 \leq nL^{-32/c_0^2}/2 + \log 2^{12}L \leq nL^{-32/c_0^2}, \]
where the last inequality hold for all sufficiently large $n$. This completes the proof.\end{proof}

\section*{Acknowledgments}
We thank Rob Morris for many helpful comments on the presentation of this paper. We also thank Vishesh Jain, Natasha Morrison, Ashwin Sah,
Mehtaab Sawhney and Van Vu for helpful remarks on the first preprint.

\appendix
\section{The Proofs of two Esseen-type lemmas}\label{sec:FourierPrep}

In this section we prove our two Esseen-type lemmas, Lemma~\ref{lem:esseen} and Lemma~\ref{lem:revEsseen}, for random variables of the form $W^T\tau$, where $\tau$ is a $\mu$-lazy random vector in 
$\{-1, 0,1\}^{2d}$ and $W$ is a (fixed) $2d \times \ell$ matrix for some $\ell\in \N$. Recall that for a vector $u\in \R^{\ell}$, we let $\|u\|_{\T}$ denote the Euclidean distance from $u$ to the integer lattice $\Z^{\ell}$. 

\vspace{3mm}

\subsection{Basics of Fourier representation}

As above, we let $\tau$ be a $\mu$-lazy random vector in $\{-1,0,1\}^{2d}$ and let $W$ be a $2d \times \ell $ matrix. Recall the characteristic function $\vp_X$ of 
a vector valued random variable $X$ is defined as  
\[\vp_X(\theta) = \E \exp(2\pi i \langle X, \theta\rangle),  \]
and so we may express characteristic function of $W^T\tau$ as
\[ \vp(\theta) = \E \exp(2\pi i \langle \tau, W\theta\rangle) =\prod_{j=1}^{2d}\big((1-\mu) + \mu \cos(2\pi(W\theta)_j)\big)\, .\]
We note the elementary fact that for $\mu\in [0,1/4]$ we have
\begin{equation}\label{eq:cos-approx}  \mu \| x \|_{\T}^2 \leq - \log \left(1 - \mu + \mu \cos(2\pi x) \right) \leq 32\mu \|x \|_{\T}^2\, , \end{equation}	
from which we deduce
\begin{equation} \label{eq:phiBnds} \exp \left(- 32\mu \left\| W\theta\right\|_{\T}^2   \right) \leq 
\vp(\t) \leq  \exp \left(-\mu \left\| W\theta\right\|_{\T}^2  \right).
\end{equation}

We now note a standard fact regarding Fourier inversion (see \cite{tao-vu-book} p.290).

\begin{fact}[Fourier inversion]\label{fact:inversion}
	Let $X$ be a random vector in $\R^\ell$, then for $w\in\R^\ell$ we have 
$$\E \exp\left(- \frac{\pi\|X-w\|_2^2}{2}\right) =\int_{\R^\ell} e^{-\pi\|\theta\|_2^2 }\cdot e^{-2\pi i\langle w, \theta\rangle} \vp_X(\theta)\,d\theta\,. $$
In particular, letting $g \sim \cN(0, (2\pi)^{-1} I_{\ell})$, we have
\[
\E \exp\left(- \frac{\pi\|X-w\|_2^2}{2}\right) = \E_g(e^{-2\pi i\langle w, g\rangle}\varphi_X(g))\, .
\]\end{fact}

\vspace{3mm}

\subsection{Proof of Lemma~\ref{lem:esseen} and Lemma~\ref{lem:revEsseen}} 

Recall that for $\ell\in\N$, $\g_\ell$ denotes the $\ell$ dimensional Gaussian measure defined by $\g_\ell(S) = \PP(g \in S)$, where $g \sim \cN(0, (2\pi)^{-1} I_{\ell})$. We begin with the proof of Lemma~\ref{lem:esseen}.

\begin{proof}[Proof of Lemma~\ref{lem:esseen}] Let $w\in \R^\ell$. We apply Markov's inequality to obtain
\[ \Pr_\tau\big( \|W^T\tau-w\|_2\leq \beta\sqrt{\ell}\big) \leq \exp\left(\frac{\pi}{2} \beta^2 \ell \right) \E_\tau \exp\left(- \frac{\pi \|W^T\cdot \tau - w\|_2^2}{2}\right)\, . \] 
As above, let $\vp$ be the characteristic function of $W^T\tau$. We apply Fact~\ref{fact:inversion} and \eqref{eq:phiBnds} to obtain
\[ \Ex_{\tau} \exp\left(-\frac{ \pi\|W^T\cdot \tau - w\|_2^2}{2}\right) = \EE_{g}[e^{-2\pi i\langle w, g\rangle}\vp(g)] \leq \E_{g}[\exp(-\nu\| W g\|_{\T}^2)].  \]
The right-hand-side of the above may be rewritten as
\[ \int_{0}^{1} \PP_{g}(\exp(-\nu\| W g\|_{\T}^2)\geq t)\, dt = \nu\int_{0}^{\infty} \PP_{g}(\| W g\|_{\T}^2\leq u) e^{-\nu u}\, du = \nu\int_{0}^{\infty} \gamma_{\ell}(S_W(u)) e^{-\nu u}\, du , \]
where for the first equality we made the change of variable $t= e^{-\nu u}$.

Choosing $m$ to maximize  $\gamma_{\ell}(S_W(u)) e^{-\nu u/2}$ (as a function of $u$), we may bound
\[
\nu\int_{0}^{\infty} \gamma_{\ell}(S_W(u)) e^{-\nu u} du \leq \nu \gamma_{\ell}(S_W(m))e^{-\nu m/2} \int_{0}^{\infty}e^{-\nu u/2}du 
= 2\gamma_{\ell}(S_W(m))e^{-\nu m/2}\, .
\]
Putting everything together we obtain
\[
\Pr_\tau(\|W^T\tau-w\|_2\leq 2\beta\sqrt{\ell}) \leq 2e^{ \pi\beta^2 \ell/2 } e^{-\nu m/2} \gamma_{\ell}(S_W(m))\, .
\]
\end{proof}

The proof of Lemma~\ref{lem:revEsseen} proceeds in much the same way.

\begin{proof}[Proof of Lemma~\ref{lem:revEsseen}]
Let us set $X = \|W^T\cdot \tau\|_2$ and write 
\[ \E_X e^{-\pi X^2/2}  = \EE_X\, \1( X\leq \beta\sqrt{\ell} )e^{-\pi X^2/2} 
+ \EE_X\,\1\big(  X \geq \beta\sqrt{\ell} \big) e^{-\pi X^2/2} \leq \PP_X(X\leq \beta\sqrt{\ell} ) + e^{-\pi \beta^2\ell/2}\,   \]
and therefore, using that $\exp(-\pi \beta^2\ell/2)\leq \exp(-\beta^2\ell)$,
\[  \E_\tau \exp\left(\frac{-\pi \|W^T\cdot \tau\|_2^2}{2}\right) \leq \Pr_\tau(\|W^T\cdot \tau\|_2\leq \beta\sqrt{\ell}) + e^{-\beta^2\ell}.  \]  
As before, we let $\vp$ be the characteristic function of $W^T\tau$, and let $g$ be a standard $\ell$-dimensional Gaussian random variable with standard deviation $(2\pi)^{-1/2}$. By Fact~\ref{fact:inversion} and \eqref{eq:phiBnds} we obtain
\[\Ex_\tau \exp\left(-\frac{ \pi\|W^T\cdot \tau\|_2^2}{2}\right) = \EE_g[\varphi(g)] \geq \E_{g}[\exp(-32\mu\| W g\|_{\T}^2)]. \] 
Similar to the proof of Lemma~\ref{lem:esseen}, we write
\[ \EE_g[\exp(-32\mu \| W g\|_{\T}^2)] = 32\mu\int_{0}^{\infty} \g_{\ell}(S_W(u)) e^{-32\mu u} du \geq 32\mu\g_{\ell}(S_W(t))\int_t^{\infty} e^{-32 \mu u}\, du,\]
where we have used that $\g_{\ell}(S_W(b)) \geq \g_{\ell}(S_W(a))$ for all $b \geq a$. This completes the proof of Lemma~\ref{lem:revEsseen}.\end{proof}

\section{Relating $A$ to the zeroed out matrix $M$.}\label{sec:replacement}

In this section we prove Lemma \ref{lem:replacement} and Lemma~\ref{lem:RhoVtau}. To prove these results, we compare Fourier transforms 
(that is the \emph{characteristic functions}) of the random variables $Mv$ and $Av$, for fixed $v$.  We first record the characteristic functions 
of these random variables.  For $\xi \in \R^n$ we have
$$\psi_{v}(\xi) := \E\, e^{2\pi i\langle Av, \xi\rangle} = \left(\prod_{k = 1}^n \cos(2\pi v_k \xi_k) \right)\cdot \left(\prod_{j < k} \big(2\pi (\xi_j v_k + \xi_k v_j)\big) \right) $$ 

and
$$ \chi_{v}(\xi) := \E\, e^{2\pi i \langle Mv, \xi \rangle} = \prod_{j = 1}^d \prod_{k = d+1}^n \left(\frac{3}{4} + \frac{1}{4}\cos\big(2\pi (\xi_j v_k + \xi_k v_j)\big) \right)\,. $$
Our comparison is based on two main points. First we have that $\chi_v(\xi)\geq 0$. Second, we have
\begin{equation}\label{eq:fourier-comparison}
\psi_v(\xi) \leq \chi_v(2\xi)\,,
\end{equation}
which follows from $|\cos(t)| \leq \frac{3}{4} + \frac{1}{4} \cos(2 t)$ and $|\cos(t)|\leq 1$.

\begin{fact}\label{fact:expForm} For $v \in \R^n$, and $t \geq \cT_L(v)$, we have
\[  \E \exp(-\pi \|Mv\|_2^2 / t^2)  \leq  (9 Lt )^n .\]
\end{fact}
\begin{proof}
Now $\E \exp(-\pi \|Mv\|_2^2 / t^2)$ is at most \begin{equation} \label{eq:Mv-split}
 \P(\|M v\|_2 \leq t \sqrt{n}) + \sqrt{n} \int_{t}^\infty \exp\left(- \frac{s^2 n }{t^2}\right)\P(\|M v \|_2 \leq s \sqrt{n})\,ds\,.
\end{equation}
 and since $t \geq \cT_L(v)$, we have $\P(\|Mv \|_2 \leq s\sqrt{n}) \leq (8Ls)^n$ for all $s\geq t$, and so we may bound 
\[ \sqrt{n}\int_{t}^\infty \exp\left(- \frac{s^2 n }{t^2}\right)\P(\|M v \|_2 \leq s \sqrt{n})\,ds 
\leq \sqrt{n}(8Lt)^n \int_t^\infty  \exp\left(- \frac{s^2 n }{t^2}\right)(s/t)^n \,ds\, . \]
Changing variables $u=s/t$, the right hand side is equal to
\[  t^{-1} \sqrt{n}(8Lt)^n \int_1^\infty \exp(-u^2n) u^n \,du \nonumber \leq t^{-1}\sqrt{n}(8Lt)^n \int_1^\infty \exp(-u^2/2)\,du \leq (9 Lt )^n, \]
as desired.
\end{proof}

\vspace{3mm}

\begin{proof}[Proof of Lemma \ref{lem:replacement}]
Apply Markov's inequality to bound \begin{equation}\label{eq:Av-markov}
\P(\|A v - w\|_2 \leq t \sqrt{n}) \leq  \exp(\pi n/2) \E\, \exp\left(- \pi\| A v - w\|_2^2 / 2t^2\right)\,.
\end{equation}
Using the Fourier inversion formula in Fact~\ref{fact:inversion} we write
\begin{align}\label{eq:Av-FI}
\E_A\, \exp\left(- \pi \| A v - w\|_2^2 / 2t^2\right) 
&= \int_{\R^n} e^{-\pi \| \xi \|_2^2} \cdot e^{-2\pi it^{-1}\langle w, \xi\rangle}  \psi_v(t^{-1}\xi)\,d\xi\,. \end{align}
Rescaling, applying \eqref{eq:fourier-comparison} and non-negativity of $\chi_v$ yields that the RHS of \eqref{eq:Av-FI} is at most
\[ \int_{\R^n} e^{-\pi \| \xi \|_2^2 } \chi_v(2t^{-1}\xi)\,d\xi \leq \E_M \exp(-2\pi \|Mv\|_2^2 / t^2).  \]
Now use Fact~\ref{fact:expForm} along with the assumption $t \geq \cT_L(v)$ to obtain
\[  \E_M \exp(-2\pi \|Mv\|_2^2 / t^2)\leq  (9 Lt )^n, \]
as desired.\end{proof}

\vspace{4mm}

\noindent We prove Lemma \ref{lem:RhoVtau} in a similar manner. Recall $\rho_{\eps}(v) = \max_{b \in \R^n} \PP\left( \sum_i v_i\eps_i \in (b-\eps,b+\eps) \right)$.
\begin{proof}[Proof of Lemma \ref{lem:RhoVtau}]
Set $\eps = \cT_L(v)$ and let $B$ be a $n \times n$ matrix uniformly drawn from all matrices with entries in $\{\pm 1\}$ and apply Markov's inequality to bound \begin{equation}\label{eq:rho-eps-Markov}
\rho_\eps(v)^n \leq \max_{w \in \R^n} \P(\|B v - w\|_2 \leq \eps \sqrt{n}) \leq \max_{w \in \R^n} \exp(\pi n/2) \E \exp\left(-\pi \|Bv - w\|_2^2 / 2\eps^2 \right)\,.
\end{equation}
Apply Fact \ref{fact:inversion} to write 
\begin{align} \label{eq:B-moment} \E \exp\left(- \pi\| B v - w\|_2^2 / 2\eps^2 \right) 
&=  \int_{\R^n} e^{-\pi \| \xi\|_2^2}\cdot e^{-2\pi i\eps^{-1}\langle w, \xi\rangle}   \prod_{1\leq j,k \leq n} \cos(2\pi\eps^{-1}v_j \xi_k) \, d\xi \end{align}
and use H\"older's inequality to bound the RHS of \eqref{eq:B-moment}
\begin{equation}\label{eq:replace-Holder} \leq \left(\int_{\R^n} e^{-2\pi\| \xi\|_2^2 / 3} \,d\xi\right)^{3/4} \left(\int_{\R^n}e^{-2\pi \| \xi\|_2^2} \prod_{1\leq j, k \leq n} \cos(2\pi \eps^{-1}v_j \xi_k)^4 \, d\xi\right)^{1/4}\,. \end{equation}
Now use $\int_{\R^n} e^{-2\pi \| \xi\|_2^2 / 3} \,d\xi = \left(\frac{3}{2}\right)^{n/2}$ and $(\cos(a)\cos(b))^4 \leq \frac{3}{4} + \frac{1}{4}\cos(2(a+b))$, to see \eqref{eq:replace-Holder} is
\begin{equation} \label{eq:moreequations} 
\leq \left(\frac{3}{2}\right)^{3n/8}\left( 2^{-n/2}\int_{\R^n} e^{-\pi\|\xi\|_2^2} \chi_v(\sqrt{2}\eps^{-1}\xi)\,d\xi \right)^{1/4}
\leq \left(\frac{27}{128}\right)^{n/8} \left( \E \exp\left(- \pi \| M v\|_2^2 /\eps^2\right) \right)^{1/4}. \end{equation}
Taken together, lines \eqref{eq:rho-eps-Markov}, \eqref{eq:B-moment}, \eqref{eq:replace-Holder}, \eqref{eq:moreequations} tell us that 
\begin{equation} \label{eq:rho-eps-UB-M}
	\rho_\eps(v)^n \leq (3/2)^{3n/8} (\exp(\pi/2)/\sqrt{2})^n \left( \E \exp\left(- \pi\| M v\|_2^2 /\eps^2\right) \right)^{1/4}\,. 
	\end{equation}
Now apply Fact~\ref{fact:expForm} to bound $\E \exp\left(- \pi\| Mv \|_2^2/\eps^2 \right) \leq (9 L \eps)^n$ and so 
$$\rho_\eps(v)^n \leq (2^{12} L \eps)^{n/4}\,.$$	
\end{proof}

\section{Dealing with unstructured vectors}\label{sec:unstructured}
In this short section we will prove Lemma \ref{lem:unstructured}. This lemma is very similar to Lemma 2.1 in \cite{CMMM}, which is essentially the same but over $\Z_p$. Their lemma, in turn, uses ideas from previous sources, mainly \cite{costello-tao-vu} but also \cite{nguyen-singularity,ferber-jain}. No originality is claimed on our part, for the contents of this appendix. Here we let $A_m$ denote a $m\times m$ symmetric random matrix with entries in $\{-1,1\}$, coupled so that 
$A_{m-1}$ is $A_m$ with the first row and columns removed.

We define the quantity
\[q_n(\gamma)=\max_{w\in \R^n}\Pr(\exists v\in \R^n\setminus \{0\}:~A_nv=w,~\rho(v)\geq \gamma)\,\]
and bound the singularity probability of $A_n$ in terms of $q_n(\g)$. 
\begin{lemma}
	For $\g>0$, we have 
	\[\Pr(\det(A_n)=0)\leq 16n \sum_{m=n}^{2n-2}\left(\gamma^{1/8}+\frac{q_{m-1}(\gamma)}{\gamma}\right).\]
\end{lemma}

We prove Lemma~\ref{lem:unstructured} with aid of three further lemmas. The first is essentially \cite[Lemma A.1]{CMMM} (see also \cite[Section 2]{nguyen-singularity}).
\begin{lemma}\label{lem:decrease-rank} We have that 
	\begin{equation}\label{eq:decrease-rank} \Pr(\det(A_n)=0)\leq 4n\sum_{m=n}^{2n-2}\Pr(\rk (A_m)=m-1~\cap~ \rk(A_{m-1})\in \{m-1,m-2\}).\end{equation}
\end{lemma}
\begin{proof}This lemma appears as \cite[Lemma A.1]{CMMM} but over $\Z_p$. If we apply that lemma with $p \gg n^n$ then all ranks are unchanged when viewed mod $p$ and thus the lemma holds over $\R$.\end{proof}

We deal with the sum \eqref{eq:decrease-rank}, in two different cases: 
$\rk(A_{m-2}) = \rk(A_{m-1}) = m-1$ and $\rk(A_{m-1}) = m-2$, $\rk(A_{m}) = m-1$. We deal with this latter case first.

\begin{lemma}\label{lem:step-down} We have 
	\[\Pr(\rk (A_n)=n-1~\cap~ \rk(A_{n-1})=n-2)\leq q_{n-1}(\gamma)+\gamma . \]
\end{lemma}
\begin{proof}
We define a map 
\[ \vp  : \{ A_n : \rk (A_n)=n-1~\cap~ \rk(A_{n-1})=n-2 \} \rightarrow \R^{n-1}, \]
 so that $A = A_n$ satisfies $A_{[n] \times [2,n]} \vp(A) = 0 $. For this, let $A_n$ be such that $\rk(A_n) = n-1$ and $\rk(A_{n-1}) = n-2$
and let $X$ be the first row of $A_n$ with the first entry removed.  Notice that $X\in \mathrm{Im}(A_{n-1})$, otherwise it is not hard to see that
$\rk(A_n)=\rk(A_{n-1})+2=n$, using that $A_n$ is symmetric.  Thus there is $a\in \R^{n-1}\setminus \{0\}$ such that $A_{n-1}\cdot a=0$ and $\langle a, X \rangle=0$; further, since $\rk(A_{n-1}) = n-2$, this vector $a$ is unique up to scalar multiples. Thus define $\vp(A_n) := a$.
	
With $\vp$ in hand, we now bound
$$\Pr(\rk (A_n)=n-1~\cap~ \rk(A_{n-1})=n-2)\leq \Pr(\rk(A_{n-1}) = n-2, \langle X, \varphi(A_{n-1})\rangle =0 )\,.$$
Considering the cases of $\rho(\varphi(A_{n-1})) \geq \gamma$ and $ \rho(\varphi(A_{n-1})) < \gamma$ separately allows us to write 
$$ \Pr(\rk(A_{n-1}) = n-2, \langle X, \varphi(A_{n-1})\rangle =0 ) \leq q_{n-1}(\gamma) + \gamma\, ,$$ as desired. \end{proof}

\vspace{3mm}

We now treat the case when $\rk(A_n) = \rk(A_{n-1}) =n-1$. For this lemma, we adopt notation different from what we have been using in the body of the paper. 
If $v \in \R^{m}$ and $S \subset [m]$ we let $v_S \in \R^{m}$ be the vector $v_S = ( v_k \1(k \in S) )_{k \in [m]}$. 

\begin{lemma}\label{lem:rank-t}
For $t\in [n-2]$ we have \[\Pr(\rk (A_n)=n-1~\cap~ \rk(A_{n-1})=n-1)\leq 3^{t} q_{n-1}(\gamma)+(2^t \gamma + 2^{-t})^{1/4}\,.\]
\end{lemma}
\begin{proof}
Let $I \cup J$ be a non-trivial partition of $[n-1]$ and set $\cA = \{A_{n-1}  : \det A_{n-1} \neq 0 \}$.
By a decoupling argument (e.g. line (9) of \cite{ferber-jain} or \cite[Lemma A.9]{CMMM}) we have \begin{align}
&\Pr(\rk (A_n) =\rk(A_{n-1})=n-1) \nonumber \\
&\leq \E_{A_{n-1}}\, \P\bigg(\langle A_{n-1}^{-1}(X - X')_I,(X - X')_J\rangle = 0 \, \big\vert \,A_{n-1}\bigg)^{1/4}\1( A_{n-1} \in \cA ) \label{eq:decouple}
\end{align}
where $X, X'$ are independent, and chosen uniformly at random from $\{1,-1\}^{n-1}$.  Following \cite{CMMM}, set 
\[ W(I) = \left\lbrace v \in \{-2,0,2\}^{n-1}: v_i = 0 \text{ for all }i \notin I  \right\rbrace \]
 and $$U_\g^{(I)} = \left\lbrace A_{n-1} : \rho(v) \leq \g \text{ for every }v \in \R^{n-1}\setminus \{0\} \text{ with } A_{n-1} v \in W(I)  \right\rbrace \,.$$  
	
\noindent We bound \eqref{eq:decouple} with three short claims.

\begin{claim}\label{cl:b1} $\P(A_{n-1} \notin U_{\gamma}^{(I)}) \leq 3^{|I|} q_{n-1}(\gamma)$
\end{claim}
\begin{proof}[Proof of Claim \ref{cl:b1}]
If $A_{n-1} \notin U_\gamma^{(I)}$ then there is a $v \in \R^{n-1} \setminus \{0\}$ with $\rho(v) \geq \gamma$ so that $A_{n-1} v \in W(I)$.  Union bounding over $W(I)$ completes the claim.\end{proof}

\vspace{2mm}

\noindent We now write $w = X - X'$ and for $A_{n-1} \in \cA$, write $x = A_{n-1}^{-1}w_I$, as we see in \eqref{eq:decouple}.

\begin{claim}\label{cl:b2}
For $A_{n-1} \in \cA$ we have $\P(x = 0 \,|\, A_{n-1}) = 2^{-|I|}$.\end{claim}
\begin{proof}[Proof of Claim \ref{cl:b2}] Simply note that $\P(x = 0 \,|\, A_{n-1})= \P(X_I = X_I') = 2^{-|I|}$.\end{proof}
	
\vspace{2mm}	

\noindent Finally we have:
 
\begin{claim}\label{cl:b3} If $A_{n-1} \in U_\gamma^{(I)} \cap \cA$ then 
$$\P_x(\langle x,w_J\rangle = 0 \text{ and } x \neq 0 \,|\, A_{n-1}) \leq 2^{|I|} \g\,.$$
\end{claim}
\begin{proof}[Proof of Claim \ref{cl:b3}]
Since $A_{n-1} x = w_I$ with $x \neq 0$ and $A_{n-1} \in U_{\gamma}^{(I)}$, we have $\rho(x) \leq \gamma$. Conditioned on a given $x \neq 0$ with $\rho(x) \leq \gamma$ as well as $A_{n-1}$ bound \begin{align*}
		\P_{X_J,X_J'}(\langle x ,w_J\rangle = 0 ) = \P_{X_J,X_J'}(\langle x_J,X_J\rangle = \langle x_J,X_J' \rangle ) \leq \rho(x_J) \leq 2^{|I|} \rho(x) \leq 2^{|I|} \gamma \end{align*}
where we have used \cite[Lemma 2.9]{ferber-jain} for the bound $\rho(x_J)\leq 2^{|I|} \rho(x)$.\end{proof} 

Choosing $I = [t]$ and combining the previous three claims with \eqref{eq:decouple} completes the proof of Lemma~\ref{lem:rank-t}.
\end{proof}

\begin{proof}[Proof of Lemma \ref{lem:unstructured}]
	For each $\g > 1/2$ we have $q_n(\g) \geq 2^{-n}$ and so we may assume that $\g < 1/n$ and $\g > 2^{-n}$.  Set $t = \lfloor \log_4(1/\gamma) \rfloor \in [n-2]$ and apply Lemmas \ref{lem:decrease-rank}, \ref{lem:step-down} and \ref{lem:rank-t} to show \begin{align*}\P(\det(A_n) = 0) &\leq 4n \sum_{m = n}^{2n - 2} \left(\gamma + q_{m-1}(\gamma) + (3 \gamma^{1/2})^{1/4} + \frac{q_{m-1}(\gamma)}{\gamma} \right) \\
	&\leq 16n\sum_{m = n}^{2n - 2} \left(\gamma^{1/8} + \frac{q_{m-1}(\gamma)}{\gamma}\right)\, ,
	\end{align*} thus completing the proof of Lemma~\ref{lem:unstructured}.
\end{proof}

\section{A few more details}\label{sec:details}

We check the details on Lemma~\ref{lem:HansonWright}. We do this by deriving this lemma from a classical inequality due to Talagrand \cite{talagrand}. 

\begin{theorem}[Talagrand's Inequality]\label{thm:talagrand}
	Let $F:\R^n \to \R$ be a convex $1$-Lipschitz function and $\sigma = (\sigma_1,\ldots,\sigma_n)$ where $\sigma_i$ are i.i.d. random variables with $|\sigma_i|\leq 1$.  Then for any $t \geq 0$ we have 
$$\P\left( \left| F(\sigma) - m_F \right| \geq t  \right) \leq 4 \exp\left(-t^2/16 \right)\, ,$$ where $m_F$ is the median of $F(\sigma)$. 
\end{theorem}

From this, we derive the following consequence, which appears in \cite{RV-HW} without explicit constants.

\begin{lemma}
For $d \in \N$, $\nu \in (0,1)$, let $\delta \in (0,\sqrt{\nu}/16)$, let
$\sigma \sim \cQ(2d, \nu)$, and let $W$ be a $2d \times k$ matrix satisfying $\|W\|_{\HS}\geq \sqrt{k}/2$ and $\|W\|\leq 2$.  Then
\begin{align} \label{eq:Psigma}
\Pr(\|W^T\sigma\|_2\leq \delta \sqrt{k})\leq 4\exp(-2^{-12}\nu k)
\end{align}
\end{lemma}
\begin{proof}
	Note \eqref{eq:Psigma} is trivial if $k \leq 2^{12}/\nu$, so we assume $k > 2^{12}/\nu$.
	Define 
	\[ F(x) :=\|W\|^{-1}\|W^T x\|_2 \] and note that $F$ is convex and $1$-Lipschitz. Theorem~\ref{thm:talagrand} immediately tells
	us that $F$ is concentrated about the median $m_F$ and so we only need to estimate $m_F$. For this, let us write down
	\[ m:= \EE\, \|W^T \sigma\|_2^2 =\sum_{i,j}W_{ij}^2 \Ex\, \sigma^2 =\nu \|W\|_{\HS}^2\] and 
	\[m_2:= \EE\, \|W^T \sigma\|_2^4-(\EE\, \|W^T \sigma\|_2^2)^2 = \sum_{i,j}W_{ij}^2\big( \EE\, \sigma_i^4 -(\EE\, \sigma_i^2)^2\big)
	\leq  \nu \|W\|_{\HS}^2.\]
	So for $t>0$ we have, by Markov,
	\[\Pr(\|W^T \sigma\|_2^2\leq m-t)\leq t^{-2}\Ex\, \left( \|W^T \sigma\|_2^2-m \right)^2  = t^{-2}m_2 \leq  t^{-2}\nu \|W\|_{\HS}^2 .\]
	So set $t = \nu\|W\|_{\HS}^2/2$ to get 
	 \[\Pr(\|W^T \sigma\|_2^2\leq \nu\|W\|_{\HS}^2/2)\leq 4 (\nu \|W\|_{\HS}^2)^{-1}<1/2\, ,\] since $\|W\|_{\HS}^2\geq k/4>8\nu^{-1}$. It follows that 
	 \begin{align}\label{eq:mfbd}
	  m_F\geq \|W\|^{-1}\sqrt{\nu/2}\|W\|_{\HS}\, .
	  \end{align}
	  Now we may apply Talagrand's Inequality \ref{thm:talagrand} with $t=m_F-\delta \sqrt{k}\|W\|^{-1}$
	   to obtain 
\[\P\left(\|W^T \sigma\|_2 \leq \delta\sqrt{k} \right) \leq 4 \exp\left(-t^2/16\right)\, .\]
To complete the proof we note that 
 \[t=m_F-\delta \sqrt{k} \|W\|^{-1}\geq \|W\|^{-1}(\sqrt{\nu/2}\|W\|_{\HS}-\delta\sqrt{k})\geq \sqrt{\nu k}/16\, ,\ \]
	 where we have used~\eqref{eq:mfbd}, $\|W\|\leq 2$, $\|W\|_{\HS}\geq \sqrt{k}/2$ and $\delta\leq \sqrt{\nu}/8$.
\end{proof}

\vspace{3mm}

The following integral appears in the proof of Lemma~\ref{lem:tensor}.

\begin{fact}\label{fact:infamous-int} For $k\geq 0$, we have the integral inequality
\[ \int_{\sqrt{k+1}}^{\infty} \left(1+\frac{ 2 u }{\sqrt{k+1} }\right)^{k+2}u e^{-2u^2} \, du \leq 2 . \]
\end{fact}
\begin{proof}
Using the inequality $1+x \leq e^{x^2/3}$ for $x\geq 2$, we have
\[
 \int_{\sqrt{k+1}}^{\infty} \left(1+\frac{ 2 u }{\sqrt{k+1} }\right)^{k+2}u e^{-2u^2} \, du 
 \leq
 \int_{\sqrt{k+1}}^{\infty} \left(1+\frac{ 2 u }{\sqrt{k+1} }\right) ue^{-2u^2/3}\, .
\]
Since $k\geq0$, the right hand side is at most 
\[
\int_{1}^{\infty} u\left(1+ 2 u\right) e^{-2u^2/3}\leq 2\, .
\]
\end{proof}

\section{Proof of Lemma~\ref{lem:basis-net}}\label{sec:randomrounding}
We will require the following standard concentration result for the norm of a random matrix with independent, mean-zero, sub-Gaussian entries (see, e.g. \cite[Theorem 4.4.5]{vershynin2018high}). 

\begin{theorem}\label{thm:Versh-norm}
Let $B$ be an $m\times n$ random matrix with independent, mean zero entries with $\max_{i,j} |B_{ij}|\leq K$. Then for $t>0$ we have
\[
\PP(\|B\|\geq 8K(\sqrt{m}+\sqrt{n}+t))\leq 2\exp(-t^2)\, .
\]
\end{theorem}
We remark that the explicit constant $8$ does not appear in the statement of \cite[Theorem 4.4.5]{vershynin2018high}, but can easily be extracted from the proof.

We also note the basic bound on the number of integer points in a ball, 
\[ |B_{r\sqrt{n} }(0) \cap \Z^n| \leq |B_{(r+1)\sqrt{n}}(0)| \leq (4r)^{n},\]
where the first inequality holds by the fact that the boxes 
\[ p+\left[-\tfrac{1}{2},\tfrac{1}{2}\right]^{2kd} \text{ where } p \in \Z^{2kd}\cap B_{16\sqrt{kd}/\delta}(0)\]
are pairwise disjoint, have volume $1$ and are contained in $B_{(r+1)\sqrt{n}}(0)$. 

\begin{proof}[Proof of Lemma~\ref{lem:basis-net}]
We claim that the following net works for Lemma~\ref{lem:basis-net}. Define 
\[\cN:=\left\{W\in \R^{[2d]\times [k]}:~ W_{i,j}\in \frac{\delta}{8\sqrt{d}}\cdot\Z \text{ and } \|W\|_{\HS} \leq 2\sqrt{k} \right\}. \]
To bound $|\cN|$, note that $(8\sqrt{d}/\delta) \cdot \cN \subseteq B_{16\sqrt{kd}/\delta}(0) \cap \Z^{2kd}.$ Thus
\[ |\cN| \leq |B_{16\sqrt{kd}/\delta}(0) \cap \Z^{2kd}| \leq (2^6/\delta)^{2dk}.\]

Now let $U\in \cU_{2d,k}$ and $A$ be a $r \times 2d$ matrix. We show that there exists a $W \in \cN$ satisfying properties
\eqref{it:rr-mat-1}, \eqref{it:rr-mat-2} and \eqref{it:rr-mat-3}. We find such a $W$ by randomly rounding $W$ onto our net $\cN$ and showing that the
properties each hold with probability $> 3/4$.

For this, let $U'=\tfrac{8\sqrt{d}}{\delta}U$ and let $W'$ be a $2d \times k$ random matrix with independent entries where $W'_{i,j} \in \{\lfloor U'_{i,j} \rfloor, \lceil U'_{i,j} \rceil \}$ and $\E W'_{i,j} = U'_{i,j}$
and let $W=\tfrac{\delta}{8\sqrt{d}}W'$. We remark that for \emph{all} possible $W$, we have 
\[\|W\|_{\HS}\leq \|U\|_{\HS}+\delta\sqrt{k}\leq 2\sqrt{k}, \] and so we always have $W\in \cN$. 

For Property~\eqref{it:rr-mat-1}, we write 
\[\EE_{W}\, \|A(W-U)\|_{\HS}^2
=\sum_{i,j}\EE\, \langle A^T e_j, (W-U) e_i\rangle^2 =\sum_{i,j,\ell}A_{j\ell}^2\Ex\,(W_{\ell i}-U_{\ell i})^2 \leq \frac{\delta^2 \|A\|_{\HS}^2 k}{64d},\] using that
 $\Ex\, (W_{\ell i}-U_{\ell i})^2 \leq \frac{\delta^2}{64d}$ in the last inequality. Thus by Markov's inequality 
\begin{equation}\label{eq:RR-HS-norm}
\PP_{W}\left(\|A(W-U)\|_{\HS}\leq \frac{\delta \|A\|_{\HS}\sqrt{k}}{\sqrt{2d}} \right) \geq 3/4.
\end{equation}
For Property~\eqref{it:rr-mat-3}, we repeat the same argument, replacing $A$ with a $2d\times 2d$ identity matrix. Thus Property~\eqref{it:rr-mat-3} also holds with probability $>3/4$ in the choice of $W$.

For property \eqref{it:rr-mat-2}, note that $W-U$ is a random matrix with independent random entries satisfying $\EE_{W}\, (W-U)_{ij} = 0 $  and 
$ |(W - U )_{ij}| \leq \frac{\delta}{8\sqrt{d}}$ for all $i,j$.  We may therefore apply Theorem~\ref{thm:Versh-norm} with $B=W-U$ and $K=\frac{\delta}{8\sqrt{d}}$ to get  
\begin{equation}\label{eq:RR-op-norm}
\PP_{W}\left(\|W-U\|\leq 8\delta \right) \geq 3/4 \,. \end{equation} 
Thus $W$ satisfies properties \eqref{it:rr-mat-1},\eqref{it:rr-mat-2},\eqref{it:rr-mat-3}, each with probability $>3/4$ and therefore there must exist a $W$ satisfying all three properties. \end{proof}

\bibliographystyle{abbrv}
\bibliography{Bib}

\end{document}